\documentclass [11pt,oneside,a4paper,mathscr]{amsart}

\usepackage{amscd,amsmath,amssymb,euscript}
\usepackage[frame,cmtip,curve,arrow,matrix,line,graph]{xy}
\usepackage{tikz-cd}
\usepackage{hyperref}
\usepackage{stmaryrd}

\title[]{Quasi-BPS categories for symmetric quivers with potential}
\author{Tudor P\u adurariu and Yukinobu Toda}
\newtheorem{thm}[equation]{Theorem}
\newtheorem{cor}[equation]{Corollary}
\newtheorem{prop}[equation]{Proposition}

\newtheorem{lemma}[equation]{Lemma}

\theoremstyle{definition}
\newtheorem{defn}[equation]{Definition}

\newtheorem{thm*}[equation]{Theorem$^*$}
\newtheorem{remark}[equation]{Remark}

\newtheorem{example}[equation]{Example}

\newtheorem{step}{Step}
\newtheorem{assum}[equation]{Assumption}
\newcommand{\comment}[1]{}

\renewcommand{\leq}{\leqslant}
\renewcommand{\geq}{\geqslant}

\newcommand{\X}{\mathscr{X}}
\newcommand{\Y}{\mathscr{Y}}

\newcommand{\Coh}{\operatorname{Coh}}

\newcommand{\id}{\operatorname{id}}

\newcommand{\Ind}{\operatorname{Ind}}
\newcommand{\Hom}{\operatorname{Hom}}

\newcommand{\Spec}{\operatorname{Spec}}

\newcommand{\ee}{\underline{e}}
\newcommand{\dd}{\underline{d}}

\newcommand{\inclusion}{\ar@<-0.3ex>@{^{(}->}[r]}
\newcommand{\dinclusion}{\ar@<-0.3ex>@{^{(}->}[d]}

\newcommand{\Tr}{\mathop{\rm Tr}}

\newcommand{\ssslash}{/\!\!/}

\usetikzlibrary{positioning,fit,calc}
\tikzstyle{block}=[draw=black, width=1cm, minimum height=2cm, align=center] 
\tikzstyle{block2}=[draw=black, text width=2cm, minimum height=1cm, align=center] 
\tikzstyle{block3}=[draw=black, text width=2cm, minimum height=1cm, align=center]

\makeatletter
\@namedef{subjclassname@2020}{%
  \textup{2020} Mathematics Subject Classification}
\makeatother

\makeatletter

\@addtoreset{equation}{section}	
\makeatother
\begin{document}

\begin{abstract}
We study 
certain 
categories associated to symmetric quivers with potential, called quasi-BPS categories. 
We construct semiorthogonal decompositions of the categories of matrix 
factorizations for moduli stacks of representations of (framed or unframed) symmetric quivers with potential, where the summands are 
categorical Hall products of quasi-BPS categories. These results generalize our previous results about the three loop quiver. 

We prove several properties of 
quasi-BPS categories: wall-crossing equivalence, strong generation, and categorical support lemma 
in the case of tripled quivers with potential. 
We also introduce reduced quasi-BPS categories for preprojective algebras, 
which have
trivial relative Serre functor and are indecomposable when the weight is coprime with the total dimension. In this case, we regard the reduced quasi-BPS categories as noncommutative local hyperkähler varieties, and as (twisted) categorical versions of crepant resolutions of singularities of good moduli spaces of representations of preprojective algebras.

The studied categories include the local models of quasi-BPS categories of K3 surfaces. In a follow-up paper, we establish analogous properties for quasi-BPS categories of K3 surfaces.

%and of moduli spaces of framed representations (generalizing previous work on the three loop quiver), where the summands are Hall products of quasi-BPS categories.
%We study quasi-BPS categories for tripled quivers with potential, and thus of preprojective algebras. 
%We construct semiorthogonal decompositions of moduli stacks of representations of certain symmetric quivers %(different from the ones studied in previous work by the first author) 
%and of moduli spaces of framed representations (generalizing previous work on the three loop quiver), where the summands are Hall products of quasi-BPS categories.
%We prove several properties of 
%quasi-BPS categories: wall-crossing equivalence, categorical support lemma, and strong generation. For a large class of quivers,
%we introduce reduced quasi-BPS categories, which have
%trivial relative Serre functor and are indecomposable when the weight is coprime with the sum of components of the dimension. In this case, we regard the reduced quasi-BPS categories as noncommutative local Kyperkähler varieties, and as (twisted) categorical versions of crepant resolutions of singularities of good moduli spaces of preprojective algebras.
%The studied categories are the local model of quasi-BPS categories of local K3 surfaces. In a follow-up paper, we establish analogous properties for quasi-BPS categories of K3 surfaces.
\end{abstract}

\maketitle

\section{Introduction}

\renewcommand{\thefootnote}{\fnsymbol{footnote}} 
\footnotetext{\emph{$2020$ Mathematics Subject Classification.} Primary: 14N35, 18N25. Secondary: 19E08. \\ \emph{Keywords:} BPS invariants, matrix factorizations, semiorthogonal decompositions, noncommutative resolutions of singularities, wall-crossing.}     
\renewcommand{\thefootnote}{\arabic{footnote}} 

\subsection{Motivation}
The BPS invariants \cite[Section 2 and a half]{MR3221298} and BPS cohomologies \cite{DM} are
central objects in the study of Donaldson-Thomas (DT) theory and of (Kontsevich--Soibelman \cite{MR2851153}) cohomological Hall algebras of a Calabi-Yau $3$-fold
or a quiver with potential. 
In this paper, we study certain subcategories of matrix factorizations associated with 
symmetric quivers with potential, called \textit{quasi-BPS categories}. 
They were introduced by the first named author in \cite{P} to prove 
a categorical version of the PBW theorem 
for cohomological Hall algebras \cite{DM}. 
As proved by the second named author~\cite{Todstack}, quivers with potential describe the local structure of moduli of sheaves on a Calabi-Yau $3$-fold (CY3). 
Thus, 
the study of quasi-BPS categories for quivers with potential 
is expected to help in understanding the (yet to be defined) Donaldson-Thomas (DT) categories
or quasi-BPS categories for global CY3 geometries. 
%before studying global CY3 geometries from the point of view of Donaldson-Thomas (DT) theory or Kontsevich-Soibelman Hall algebras, it is worth understanding 
%the local geometries modeled by quivers with potential. 

A particular case of interest is that of \textit{tripled quivers with potential}. 
A subclass of tripled quivers with potential gives a local 
model of the moduli stack of (Bridgeland semistable and compactly supported) sheaves on the local K3 surface
\begin{equation}\label{XSC}
X=S\times \mathbb{C},
\end{equation}
where $S$ is a K3 surface. 
This local description 
was used 
by Halpern-Leistner~\cite{halpK32} to prove 
the D-equivalence 
conjecture for moduli spaces of stable 
sheaves on K3 surfaces, see~\cite[Corollary~5.4.8]{T}
for its generalization.
Tripled quivers with potential are also 
of interest in representation theory: the Hall algebras of a tripled quiver with potential are Koszul equivalent to the preprojective Hall algebras introduced by Schiffmann--Vasserot \cite{SV}, Yang--Zhao \cite{YangZhaopre}, Varagnolo--Vasserot \cite{VaVa}, which are categorifications of positive halves of quantum affine algebras \cite{NeSaSc}.

The tripled quiver with potential for the Jordan quiver is the quiver with one vertex and three loops $\{X,Y,Z\}$, and with potential $X[Y,Z]$. 
In our previous papers~\cite{PTzero, PT1}, 
motivated by the search for a categorical analogue of the MacMahon formula, 
we studied 
quasi-BPS categories for the three loop quiver. 
In particular, we constructed semiorthogonal decompositions for the framed and unframed stacks of representations of the tripled quiver, we proved a categorical support lemma, and so on. 
The purpose of this paper is to generalize 
the results in~\cite{PTzero, PT1} to more general 
symmetric quivers with potential, with special attention to tripled quivers with potential. 
We also prove new results on quasi-BPS categories: first, we show that quasi-BPS categories are equivalent under wall-crossing; next, we introduce \textit{reduced quasi-BPS categories} and show that they are indecomposable when the weight is coprime with the total dimension.   
% e.g. categorical wall-crossing equivalence, the triviality 
% of the relative Serre functor over a good moduli space, and prove that they are indecomposable. 

%In \cite{PTtop}, we compute the topological K-theory of quasi-BPS categories for (a large class of) symmetric quivers with potential and preprojective algebras.
In \cite{PTK3}, we use the results of this paper to 
introduce and study quasi-BPS categories for (local) K3 
surfaces, and discuss their relationship with 
(twisted) categorical crepant resolutions of 
singular symplectic moduli spaces of semistable 
sheaves on K3 surfaces. 
%study the derived category of coherent sheaves on moduli of sheaves on K3 surfaces. 
%Wall-crossing for certain tripled quivers with potential was used by Halpern-Leistner to prove the derived equivalence of moduli of semistable sheaves on a K3 surface \cite{HalpK3}. 
%Further, tripled quivers with potentials can be used (also in conjunction with the Koszul equivalence) to study moduli of sheaves on a K3 surface, see Halpern-Leistner's derived equivalence \cite{halpK32}. 

\subsection{Quasi-BPS categories}

For a symmetric quiver $Q=(I,E)$ and 
a dimension vector $d\in \mathbb{N}^I$, consider \begin{align*}
\Tr W \colon 
\X(d):=R(d)/G(d) \to \mathbb{C}
\end{align*}
the moduli stack of representations of $Q$ of 
dimension $d$, together with the 
regular function determined by the potential $W$. 
Let 
$M(d)_{\mathbb{R}}^{W_d}$
be the set of Weyl invariant real weights of the maximal 
torus $T(d) \subset G(d)$. 
For $\delta \in M(d)_{\mathbb{R}}^{W_d}$, consider
the (ungraded or graded) 
\textit{quasi-BPS category}:
\begin{equation}\label{def:quasiBPS}
\mathbb{S}^\bullet(d; \delta)\subset \mathrm{MF}^\bullet(\X(d), \mathrm{Tr}\,W)\text{ for }\bullet\in\{\emptyset, \mathrm{gr}\},
\end{equation}
which is 
the category of matrix factorizations with factors in
\begin{align}\label{intro:def:M}\mathbb{M}(d; \delta)\subset D^b(\X(d)),\end{align}
a noncommutative resolution of the coarse space of $\X(d)$ constructed by \v{S}penko--Van den Bergh~\cite{SVdB}.

When $d$ is primitive and $\delta, \ell\in M(d)^{W_d}_\mathbb{R}$ are generic 
of weight zero with respect to the 
diagonal torus $\mathbb{C}^{\ast} \subset G(d)$,  
Halpern-Leistner--Sam's magic window theorem~\cite{hls} says that 
there is an equivalence: 
\begin{align}\label{intro:magic}
    \mathbb{M}(d; \delta) \stackrel{\sim}{\to} D^b\big(X(d)^{\ell\text{-ss}}\big).
\end{align}
Here, $\X(d)^{\ell\text{-ss}} \to X(d)^{\ell\text{-ss}}$ is the GIT quotient 
of the $\ell$-semistable locus, which is 
a $\mathbb{C}^{\ast}$-gerbe, 
and $X(d)^{\ell\text{-ss}}$ is 
a smooth quasi-projective variety. 
However, there is no equivalence (\ref{intro:magic}) for non-primitive $d$. In this case, the stack 
$\X(d)^{\ell\text{-ss}}$
contains strictly semistable 
representations, the morphism 
$\X(d)^{\ell\text{-ss}} \to X(d)^{\ell\text{-ss}}$ is more complicated, 
and $X(d)^{\ell\text{-ss}}$ is usually singular. 
Nevertheless, under some conditions on $\delta$, 
we expect $\mathbb{M}(d; \delta)$
to behave as the derived category of 
a smooth quasi-projective variety. 
Thus it is interesting to investigate 
the structure of $\mathbb{M}(d; \delta)$ or $\mathbb{S}(d; \delta)$, especially when $d$ is non-primitive. 

As its name suggests, the category 
$\mathbb{S}(d; \delta)$ was introduced in \cite{P}
as a
categorical version of \textit{BPS invariants}. 
The BPS invariants for CY3-folds are fundamental enumerative invariants which determine other enumerative invariants of interest, such as Donaldson-Thomas (DT) and Gromov-Witten invariants \cite[Section 2 and a half]{MR3221298}. 
There are BPS cohomologies whose Euler characteristics equal the BPS invariants, defined by Davison--Meinhardt~\cite{DM} in the case of symmetric quivers with potential 
and by Davison--Hennecart--Schlegel Mejia~\cite{DHSM} in the case of local K3 surfaces. 
For general CY 3-fold, up to the existence of a certain orientation data, 
the BPS cohomologies are defined in~\cite[Definition~2.11]{TodGV}. 
%One expects to define, for general CY3-folds, \textit{BPS cohomology theories} whose Euler characteristics are the BPS invariants. 
%There is no definition of BPS cohomologies for general CY3-folds, but
%Davison--Meinhardt defined BPS cohomologies in the local case of symmetric quivers with potential \cite{DM}.

In~\cite{PTtop}, we make the relation between quasi-BPS categories and BPS cohomologies more precise:
we describe the topological K-theory 
of \eqref{def:quasiBPS} in terms of 
BPS cohomologies, and show that, under some 
extra condition, they are isomorphic.

% It is an important problem to construct quasi-BPS categories for more general CY3 geometries, which we also discuss in \cite{PTK3} and \cite{PTtop}.

% In~\cite{PTtop}, we make the above relation more precise:
% we describe the topological K-theory 
% of quasi-BPS category in terms of 
% BPS invariants, and show that, under some 
% extra condition, they are isomorphic. 

%We studied the quasi-BPS category of the quiver with one vertex and three edges $\{X, Y, Z\}$ and potential $X[Y,Z]$ (alternatively, of $\mathbb{C}^3$) in \cite{PTzero, PT1}. 

%The purpose of this paper is twofold. First, we extend some of the results in \cite{PTzero, PT1} to the more general setting of tripled quivers with potential (alternatively, of preprojective algebras). Second, for a (reduced, and further classical) stack of representations of a preprojective algebra $\mathscr{P}^{\mathrm{red}}$ with good moduli space $P$, we introduce and study a \textit{reduced quasi-BPS category}
%\[\mathbb{T}\subset D^b\left(\mathscr{P}^{\mathrm{red}}\right)\]
%which we regard as a twisted categorical crepant resolution of singularities of $P$. Note that, in general, $P$ may not have (geometric) crepant resolutions of singularities. 
%In \cite{PTK3}, we use the local study in this paper to obtain analogous constructions and results about quasi-BPS categories of K3 surfaces. 

\subsection{Semiorthogonal decompositions}

In \cite{P}, the first named author constructed semiorthogonal decompositions of the categorical Hall algebra of $(Q, W)$ in Hall products of quasi-BPS categories for all symmetric quivers $Q$ and all potentials $W$. However, the (combinatorial) data which parametrizes the summands is not easy to determine, and it is not very convenient for studying explicit wall-crossing geometries. 

In this paper, 
we construct, for certain symmetric quivers, a different semiorthogonal decomposition which is more amenable to wall-crossing.
%For simplicity, we assume the potential is zero, but there are analogous semiorthogonal decompositions for an arbitrary potential case by \cite[Proposition 2.1]{P0}, \cite{PTzero}. 
We state the result in a particular case, 
see Theorem \ref{sodfullstackB} for a more general statement which applies to all tripled quivers with potential. 

Before stating Theorem \ref{thm:intro1}, we introduce some notations.
For a dimension vector $d=(d^i)_{i\in I}$, let $\underline{d}=\sum_{i\in I}d^i$ be its total length.
We set $\tau_d=\frac{1}{\underline{d}}\left(\sum_{i\in I}\sum_{j=1}^{d^i}\beta_i^j\right)$,
where $\beta_i^j$ are weights of the standard 
representation of $G(d)$. We consider the following particular examples of quasi-BPS categories \eqref{def:quasiBPS}:
\[\mathbb{S}^\bullet(d)_v:=\mathbb{S}^\bullet(d; v\tau_d).\]

\begin{thm}\emph{(Theorem~\ref{sodfullstackBW})}\label{thm:intro1}
Let $(Q, W)$ be a symmetric quiver with potential such that
the number of loops at each vertex is odd and the number of edges between any two different vertices is even. Let $\bullet\in \{\emptyset, \mathrm{gr}\}$.
     There is a semiorthogonal decomposition
    \begin{equation}\label{intro:sod1}
    \mathrm{MF}^\bullet(\X(d), \Tr W)=\left\langle \bigotimes_{i=1}^k \mathbb{S}^\bullet(d_i)_{v_i}: 
    \frac{v_1}{\dd_1} <\cdots<\frac{v_k}{\dd_k}
    \right\rangle,
    \end{equation}
    where $(d_i)_{i=1}^k$ is a partition
 of $d$ and $(v_i)_{i=1}^k\in\mathbb{Z}^k$. %There is an analogous semiorthogonal decomposition for graded categories.
\end{thm}

As in \cite{P}, we regard the above semiorthogonal decomposition as a categorical version of the PBW theorem for cohomological Hall algebras of Davison--Meinhardt \cite{DM}. 
% In loc. cit. the authors define a ``primitive" part of the Hall algebra, called BPS cohomology. 
% %In \cite{}, the authors extend the definition of BPS cohomology to more global situations. 
% BPS invariants for CY3-folds are fundamental enumerative invariants which determine other enumerative invariants of interest, such as Donaldson-Thomas (DT) and Gromov-Witten invariants \cite[Section 2 and a half]{MR3221298}. One expects to define, for general CY3-folds, BPS cohomology theories whose Euler characterstics are the BPS invariants. We regard \eqref{def:quasiBPS} as the categorical version of BPS cohomology in the local case of symmetric quivers with potential. It is an important problem to construct quasi-BPS categories for more general CY3 geometries, which we also discuss in \cite{PTK3} and \cite{PTtop}.

We also study 
semiorthogonal 
decompositions for the moduli spaces $\X^f(d)^{\mathrm{ss}}$ of semistable framed representations of $Q$, 
consisting of framed $Q$-representations 
generated by the image of the maps from the framed 
vertex. 
We state a particular case, see Theorem \ref{thmsodC} for a general statement which includes all tripled quivers with potential:

\begin{thm}\emph{(Theorem~\ref{sodfullstackBW0})}\label{thm:intro2}
In the setting of Theorem~\ref{thm:intro1}, 
we further take 
$\mu\in \mathbb{R}\setminus\mathbb{Q}$. Then there is a semiorthogonal decomposition
    \begin{equation}\label{intro:sod2}
    \mathrm{MF}^\bullet\left(\X^f(d)^{\text{ss}}, \Tr W\right)=\left\langle \bigotimes_{i=1}^k \mathbb{S}^\bullet(d_i)_{v_i}: 
    \mu\leq \frac{v_1}{\dd_1}<\cdots<\frac{v_k}{\dd_k}<1+\mu
    \right\rangle,
     \end{equation}
    where $(d_i)_{i=1}^k$ is a partition 
    of $d$ and $(v_i)_{i=1}^k\in\mathbb{Z}^k$.
\end{thm}

We proved Theorem \ref{thm:intro2} for the three loop quiver in \cite[Theorem 1.1]{PTzero} in order to give a categorical 
analogue of the MacMahon formula for Hilbert schemes 
of points on $\mathbb{C}^3$. Theorem~\ref{thm:intro2} gives
a generalization of~\cite[Theorem 1.1]{PTzero}. 
As explained in~\cite{PTzero}, the semiorthogonal 
decomposition (\ref{intro:sod2}) is regarded 
as a categorical analogue of DT/BPS
wall-crossing formula~\cite{BrH, T5, MR2892766}, 
whose motivic/ cohomological version 
is due to Meinhardt--Reineke~\cite{MeRe}. 

 Theorems \ref{thm:intro1} and \ref{thm:intro2} are two of the main tools we use to further investigate quasi-BPS categories. For example, they are central in the proof of Theorem \ref{thm:intro6} and in the results of \cite{PTtop}.

%Note that there are also semiorthogonal decompositions as \eqref{intro:sod1} and \eqref{intro:sod2} for categories of matrix factorizations with respect to a potential by \cite{P0}. 

\subsection{Categorical wall-crossing of quasi-BPS categories}
Let $d$ be a primitive dimension vector 
and let $\ell, \ell'\in M(d)^{W_d}_\mathbb{R}$ be generic stability conditions. 
Then there is a birational map between 
two crepant resolutions of $X(d)$:
\begin{equation}\label{resolutions}
\begin{tikzcd}
X(d)^{\ell\text{-ss}}\arrow[dr] \arrow[rr, dotted] & & X(d)^{\ell'\text{-ss}} \arrow[ld] \\
& X(d). &
\end{tikzcd}
\end{equation}
As a corollary of 
the magic window 
theorem (\ref{intro:magic}) of Halpern-Leistner--Sam, there is 
a derived equivalence: 
\[D^b\big(X(d)^{\ell\text{-ss}}\big)\simeq D^b\big(X(d)^{\ell'\text{-ss}}\big),\]
which proves the 
D/K equivalence conjecture of Bondal--Orlov~\cite{B-O2}, 
Kawamata~\cite{MR1949787} for the resolutions \eqref{resolutions}. 
%for any two stability conditions $\ell$ and $\ell'$ such that the stacks $\X(d)^{\ell-\text{ss}}$ and $\X(d)^{\ell'-\text{ss}}$ are Deligne-Mumford. 

We prove an analogous result when 
$d$ is not necessary primitive, hence 
there is no stability condition such that the 
$\mathbb{C}^{\ast}$-rigidified moduli stack of semistable 
representations is a Deligne-Mumford stack. 
%there is no stability condition for which $\X(d)^{\ell-\text{ss}}$ is Deligne-Mumford. 
For a stability condition $\ell$ on $Q$, 
we define a quasi-BPS category
\[\mathbb{S}^\ell(d; \delta)\subset \mathrm{MF}\big(\X(d)^{\ell\text{-ss}}, \Tr W\big)\]
which is, 
locally on the good moduli space 
$\X(d)^{\ell\text{-ss}} \to X(d)^{\ell\text{-ss}}$,
modeled by a category 
(\ref{def:quasiBPS}). 

\begin{thm}\emph{(Theorem~\ref{cor:lgen2})}\label{thm:intro5}
Let $(Q, W)$ be a symmetric quiver with potential, 
and let $\ell, \ell'$ be stability conditions.
Then there is a dense open subset $U \subset M(d)_{\mathbb{R}}^{W_d}$
such that, for $\delta \in U$, there is an equivalence:
    \begin{equation}\label{derivedequivBPSzero}
    \mathbb{S}^\ell(d; \delta) \simeq \mathbb{S}^{\ell'}(d; \delta).
     \end{equation}
\end{thm}

% One can also define quasi-BPS categories and obtain a derived equivalence analogous to \eqref{derivedequivBPSzero} 
% without the potential. 
Note that BPS invariants are preserved under wall-crossing~\cite[Lemma~4.7]{TodGV}. We regard Theorem~\ref{thm:intro5} as the 
categorical analogue of this property.

\subsection{Categorical support lemma for quasi-BPS 
categories of tripled quivers}
For a quiver $Q^{\circ}=(I, E^{\circ})$, 
let $(Q^{\circ, d}, \mathscr{I})$ be its doubled quiver
with relation $\mathscr{I}$, 
and let $(Q, W)$ be its tripled 
quiver with potential, see Subsection~\ref{subsec22}. 
Tripled quivers with potential 
form an important class of 
symmetric quivers with potential. Hall algebras of tripled quivers with potential are isomorphic to preprojective Hall algebras, which are themselves positive parts of quantum affine algebras \cite{NeSaSc}.
An important ingredient in the study of these Hall algebras is Davison's support lemma~\cite[Lemma~4.1]{Dav} for BPS sheaves of tripled quivers with potential, 
which is used to prove purity of various cohomologies \cite{Dav, DavPurity}.

Inspired by Davison's support lemma, we studied in \cite[Theorem 1.1]{PT1} the support of objects in quasi-BPS categories for the tripled quiver with potential of the Jordan quiver, and we used it to obtain generators for the integral equivariant K-theory of certain quasi-BPS categories \cite[Theorem 1.2]{PT1}.

We prove an analogous result to \cite[Theorem 1.1]{PT1} for (certain) tripled quivers with potential, see Theorem \ref{thm:intro6}.
The examples we study include all Ext-quivers of polystable sheaves (for a Bridgeland stability condition) on a local K3 
surface as in \eqref{XSC}. We use Theorem \ref{thm:intro6} to show relative properness of reduced quasi-BPS categories in Theorem \ref{thm:intro7}. 

Let $d=(d^i)_{i\in I}\in\mathbb{N}^I$ be a dimension vector and let $\mathfrak{g}(d)$ be the 
Lie algebra of $G(d)$. 
There is a projection map which remembers the linear map on the added loops (i.e. edges of the tripled quiver which are added to the doubled quiver):
\[\X(d)\to \mathfrak{g}(d)/G(d).\]
It induces a map 
\[\pi\colon \mathrm{Crit}\left(\mathrm{Tr}\,W\right)\hookrightarrow \X(d)\to \mathfrak{g}(d)/G(d)\to \mathfrak{g}(d)\ssslash G(d)=\prod_{i \in I} \mathrm{Sym}^{d^i}(\mathbb{C}).\]
Consider the diagonal 
map 
\[\Delta\colon \mathbb{C} \hookrightarrow \prod_{i\in I} \mathrm{Sym}^{d^i}(\mathbb{C}).\] %We state the categorical support theorem is Theorem \ref{lem:support}, 
%We state a corollary of the categorical support Theorem \ref{lem:support}:
%We denote by $\Delta$ the image of the above map.
% Recall $\alpha_{a,b}$ from \eqref{alphaab} and define$\alpha_{Q^{\circ}}:=\mathrm{min}\{\alpha_{a, b} : a, b \in I\}$.
% and 
% $\alpha_{Q^{\circ}}:=\mathrm{min}\{\alpha_{a, b} : a, b \in I\}$. 
For two vertices $a, b \in I$ of the quiver $Q^{\circ}$, let $\delta_{ab}=1$ if $a=b$ and $\delta_{ab}=0$ otherwise, and define:
\begin{align}\label{alphaab}
    \alpha_{a, b} &:=\sharp(a \to b\text{ in }E^\circ)+\sharp(b \to a\text{ in }E^\circ)-2\delta_{ab}.
\end{align}

\begin{thm}\emph{(Theorem~\ref{lem:support})}\label{thm:intro6}
    Let $Q^{\circ}=(I, E^{\circ})$ be a quiver such that 
    $\alpha_{a, b}$ is even for any $a, b \in I$. 
    Let 
    $(Q,W)$ be the tripled quiver with potential
    of $Q^{\circ}$. 
    If $\gcd(v, \underline{d})=1$, 
    then any object 
    of $\mathbb{S}(d)_v$ is supported over $\pi^{-1}(\Delta)$.  
\end{thm}

 % A similar result for three loop quiver 
 % was obtained in~\cite{PT1}, and 
 % Theorem~\ref{thm:intro6} gives its generalization. 
 % We use Theorem~\ref{thm:intro6} to show the properness of 
 % quasi-BPS categories
 % for reduced stacks in Theorem~\ref{thm:intro7} below. 

\subsection{Quasi-BPS categories for reduced stacks}

We now explain a modification of the categories \eqref{def:quasiBPS} with better geometric properties
in the case of tripled quivers with potential. We first introduce notations related to stacks of representations of doubled quivers. 

Let $Q^\circ=(I, E^{\circ})$ be a quiver
with stack of representations 
$\X^{\circ}(d)=R^{\circ}(d)/G(d)$. 
Let $
\mathscr{P}(d):=\mu^{-1}(0)/G(d)$
be the derived moduli stack of dimension $d$ representations of the preprojective algebra of $Q^\circ$ (equivalently, of
$(Q^{\circ, d}, \mathscr{I})$-representations), where 
$\mu^{-1}(0)$ is the derived zero locus of the moment map
\begin{equation}
    \mu\colon T^{\ast}R^{\circ}(d) \to \mathfrak{g}(d).
\end{equation} 
Consider  the good moduli space
\begin{align*}
    \mathscr{P}(d)^{\rm{cl}} \to P(d)=\mu^{-1}(0)\ssslash G(d). 
\end{align*}
In many cases, $P(d)$ is a singular symplectic 
variety and the study of
its (geometric or non-commutative or categorical)
resolutions is related to the study 
of hyperkähler varieties. Note that the variety $P(d)$ may not have 
geometric crepant resolutions of singularities, for example if $Q^\circ$ is the quiver with one loop and $g\geq 2$ loops, $d\geq 2$, and $(g, d)\neq (2,2)$, see \cite[Proposition 3.5, Theorem 6.2]{KaLeSo}.

Under the Koszul equivalence, the graded quasi-BPS category $\mathbb{S}^{\mathrm{gr}}(d)_v$ for the tripled quiver with potential of the quiver $Q^\circ$ is equivalent to 
the preprojective quasi-BPS category: 
\[\mathbb{T}(d)_v\subset D^b\left(\mathscr{P}(d)\right).\]
The stack $\mathscr{P}(d)$ is never classical because the image of the moment map $\mu$ lies in the Lie subalgebra $\mathfrak{g}(d)_0\subset \mathfrak{g}(d)$ of traceless elements. 
We consider the reduced stack 
\[\mathscr{P}(d)^{\mathrm{red}}:=\mu_0^{-1}(0)/G(d),\] where $\mu_0\colon T^*R^\circ(d)\to \mathfrak{g}(d)_0$. 
We study the reduced quasi-BPS category
\begin{equation}\label{def:redquasiBPS}
\mathbb{T}:=\mathbb{T}(d)_v^{\mathrm{red}}\subset D^b\big(\mathscr{P}(d)^{\mathrm{red}}\big).
\end{equation}
%We assume in the remaining of the section that the sum of coefficients of $\delta$ is $w\in\mathbb{Z}$, that $\gcd(w, \underline{d})=1$, and that the stack $\mathscr{P}(d)^{\mathrm{red}}$ is classical. 
Recall $\alpha_{a, b}$ from (\ref{alphaab}) and 
define 
$\alpha_{Q^{\circ}}:=\mathrm{min}\{\alpha_{a, b} \mid a, b \in I\}$. We use Theorem~\ref{thm:intro6}
to prove the following:

\begin{thm}\emph{(Propositions~\ref{thm:stronggenerator} and \ref{lem:bound}, Theorem \ref{thm:Serretriv}, Corollary \ref{mainprop:dec2})}\label{thm:intro7}
In the setting of Theorem~\ref{thm:intro6}, 
suppose that $\gcd(v, \dd)=1$. 
Then: 

(i) If $\alpha_{Q^{\circ}} \geq 2$, the category $\mathbb{T}$ is regular, and it is proper over $P(d)$. 

(ii) Suppose furthermore that $P(d)$ is Gorenstein, e.g. $\alpha_{Q^{\circ}} \geq 3$. 
Then there exists a relative Serre functor $\mathbb{S}_{\mathbb{T}/P(d)}$ of $\mathbb{T}$ over $P(d)$, and it satisfies $\mathbb{S}_{\mathbb{T}/P(d)}\cong \mathrm{id}_{\mathbb{T}}$. 

(iii) In the situation of (ii), $\mathbb{T}$ does not 
admit any non-trivial semiorthogonal decomposition. 
\end{thm}

Inspired by the above theorem, we regard \eqref{def:redquasiBPS} as a noncommutative local hyperkähler variety, which is a (twisted) categorical version of a crepant resolution of singularities of $P(d)$. 
It is an interesting question to 
see the relation with categorical crepant resolutions in the sense of Kuznetsov~\cite{KuzICM} or noncommutative 
crepant resolutions in the sense of Van den Bergh \cite{VdB22}. We plan to investigate this relation in future work.

In \cite{PTK3}, we use Theorem \ref{thm:intro7} to study reduced quasi-BPS categories for a K3 surface $S$. In particular, we show that these categories are a (twisted) categorical version of a crepant resolution of the moduli space \[M^{H}_S(v)\] of $H$-Gieseker semistable sheaves on $S$, where $H$ is a generic stability condition and $v$ is a non-primitive Mukai vector such that $\langle v, v\rangle \geq 2$, compare with \cite{KaLeSo} in the geometric case.

\subsection{Acknowledgments}
We thank Tasuki Kinjo, Davesh Maulik, Yalong Cao, Junliang Shen, Georg Oberdieck, and Jørgen Rennemo for discussions 
related to this work. T.~P.~is grateful to Columbia University in New York and to Max Planck Institute for Mathematics in Bonn for their hospitality and financial support during the writing of this paper.
The project of this paper started when Y.~T.~was visiting Columbia University 
in April 2023. Y.~T.~thanks Columbia University for their hospitality. 
	Y.~T.~is supported by World Premier International Research Center
	Initiative (WPI initiative), MEXT, Japan, and Grant-in Aid for Scientific
	Research grant (No.~19H01779) from MEXT, Japan.

 \subsection{Notations}\label{notation}

We list the main notation used in the paper in Table \ref{table:notation}.

All the spaces $\X=X/G$ considered are quasi-smooth (derived) quotient stacks over $\mathbb{C}$, where $G$ is an algebraic group. 
The classical truncation of $\X$ is denoted by 
$\X^{\rm{cl}}=X^{\mathrm{cl}}/G$. We assume that $X^{\mathrm{cl}}$ is a quasi-projective scheme.
We denote by $\mathbb{L}_\X$ the cotangent complex of 
$\X$. 
Any dg-category considered is a $\mathbb{C}$-linear 
pre-triangulated dg-category, in particular its homotopy category is a 
triangulated category. 
We denote by $D_{\rm{qc}}(\X)$ the unbounded derived category of quasi-coherent sheaves, by
$D^b(\X)$ the bounded derived category of coherent sheaves, 
and by $\mathrm{Perf}(\X)$ its subcategory of perfect complexes. 
%For a pre-triangulated dg-category $\mathcal{D}$ and 
%its full subcategory $ \mathcal{C} \subset \mathcal{D}$, we say 
%that $\mathcal{C}$ classically generates $\mathcal{D}$ if $\mathcal{D}$ 
%coincides with the smallest thick pre-triangulated subcategory 
%of $\mathcal{D}$ which contains $\mathcal{C}$. 
%If $\mathcal{D}$ is furthermore cocomplete, then 
%we say that $\mathcal{C}$ generates $\mathcal{D}$ if 
%$\mathcal{D}$ coincides with the smallest thick pre-triangulated 
%subcategory of $\mathcal{C}$ which contains $\mathcal{C}$ and closed 
%under taking colimit. 

Let $R$ be a set. Consider a set $O\subset R\times R$ such that for any $i, j\in R$ we have $(i,j)\in O$, or $(j,i)\in O$, or both $(i,j)\in O$ and $(j,i)\in O$. 
%We say $i\in I$ is maximal if the following holds: if $(j, i)\in I$, then $j=i$.
%If $(i,j)\in O$, write $i>j$.
Let $\mathbb{T}$ be a pre-triangulated dg-category. We will construct semiorthogonal decompositions
\begin{equation}\label{SOD}
\mathbb{T}=\langle \mathbb{A}_i \mid i \in R \rangle
\end{equation} 
with summands 
pre-triangulated subcategories $\mathbb{A}_i$ indexed by $i\in R$
such that for any $i,j\in R$ with $(i, j)\in O$ and for any objects $\mathscr{A}_i\in\mathbb{A}_i$, $\mathscr{A}_j\in\mathbb{A}_j$, we have 
$\Hom_{\mathbb{T}}(\mathscr{A}_i,\mathscr{A}_j)=0$. 

Consider a morphism $\pi\colon \X\to S$. We say the semiorthogonal decomposition \eqref{SOD} is \textit{$S$-linear} if $\mathbb{A}_i\otimes \pi^*\mathrm{Perf}(S)\subset \mathbb{A}_i$ for all $i\in R$. 

We use the terminology of \textit{good moduli spaces} of Alper, see \cite[Section 8]{MR3237451} for examples of stacks with good moduli spaces.

\begin{figure}
	\centering
\scalebox{0.7}{
	\begin{tabular}{|l|l|l|}
		\hline
	Notation & Description & Location defined \\\hline
$\alpha_{a,b}$ & quantity computed from the number of edges of a quiver &
Equation \eqref{alphaab}
\\ \hline
$\X(d)=R(d)/G(d)$ & stack of representations of a quiver $Q$ &
Subsection \ref{subsub221}
\\ \hline
$\pi_{X,d}\colon \X(d)\to X(d)$ & good moduli space map for $\X(d)$ &
Subsection \ref{subsub221}
\\ \hline
$W$, $\mathrm{Tr}\,W$ & potential and the regular function induced by it &
Equation \eqref{def:regularfunction}
\\ \hline
$T(d)\subset G(d)$ & maximal torus & Subsection \ref{subsec:theweightlattice}
\\ \hline
$M(d)$, $M(d)_\mathbb{R}$ & weight spaces & Subsection \ref{subsec:theweightlattice}
\\ \hline
$N(d)$ & cocharacter lattice & Subsection \ref{subsec:theweightlattice}
\\ \hline
$1_d$ & diagonal cocharacter & Subsection \ref{subsec:theweightlattice}
\\ \hline
$\chi$ & (dominant, integral) weight &
Subsection \ref{subsec:theweightlattice}
\\ \hline
$ST(d)$ & subtorus of determinant one elements &
Subsection \ref{subsec:theweightlattice}
\\ \hline
$ST(d)$ & subtorus of determinant one elements &
Subsection \ref{subsec:theweightlattice}
\\ \hline
$\langle\,,\,\rangle$ & pairing between cocharacters and characters &
Subsection \ref{subsec:theweightlattice}
\\ \hline
$\rho$ & half the sum of positive roots &
Subsection \ref{subsec:theweightlattice}
\\ \hline
$\tau_d$ & Weyl-invariant weight with sum of coefficients $1$ & Subsection \ref{subsec:theweightlattice}
\\ \hline
$\underline{e}\geq \underline{d}$ & comparison of partitions & Definition \ref{compa}
\\ \hline
$\mathcal{T}$ & tree of partitions & Definition \ref{tree}
\\ \hline
$\X^f(d)$ & stack of representations of the framed quiver $Q^f$ &
Subsection \ref{subsub224}
\\ \hline
$\lambda$ & cocharacter (in general antidominant) & 
Subsection \ref{subsub225}
\\ \hline
$p_\lambda, q_\lambda$ & maps used to define the Hall product & 
Subsection \ref{subsub225}
\\ \hline
$\mathscr{Y}(d)=\overline{R}(d)/G(d)$ & stack of representations of the doubled quiver $Q^{\circ,d}$ &
Subsection \ref{subsec22}
\\ \hline
$\pi_{Y,d}\colon\mathscr{Y}(d)\to Y(d)$ & good moduli space map of $\mathscr{Y}(d)$ &
Subsection \ref{subsec22}
\\ \hline
$\mathscr{P}(d)$ & stack of representations of the preprojective algebra of $Q^{\circ}$ &
Subsection \ref{subsec22}
\\ \hline
$\pi_{P,d}\colon\mathscr{P}(d)^{\mathrm{cl}}\to P(d)$ & good moduli space map of $\mathscr{P}(d)^{\mathrm{cl}}$ &
Subsection \ref{subsec22}
\\ \hline
$(Q,W)$ & tripled quiver with potential & 
Subsection \ref{subsub:triple}
\\ \hline
$j$  & inclusion of a Koszul stack in a smooth stack & Equation \eqref{jeta}
\\ \hline
$\eta$  & projection from the total space of a vector bundle & Equation \eqref{jeta}
\\ \hline
$\Theta$  & Koszul equivalence & Equation \eqref{equiv:theta}
\\ \hline
$\textbf{W}(d)$ & polytope used to define quasi-BPS categories
& Equation \eqref{polyt:W} 
\\ \hline
$\textbf{V}(d)$ & polytope
& Equation \eqref{polyt:W} 
\\ \hline
$\delta_d$  & Weyl-invariant real weight & Definition \ref{def:defmdw}
\\ \hline
$n_\lambda$  & width of magic categories & Equation \eqref{nlambdadef}
\\ \hline
$\mathbb{M}(d; \delta_d), \mathbb{M}(d)_v$ & magic categories
& Equations \eqref{defmdw}
\\ \hline
$\mathbb{S}(d; \delta_d),  \mathbb{S}^{\mathrm{gr}}(d; \delta_d),\mathbb{S}(d)_v$ &  quasi-BPS categories for a quiver with potential
& Equations \eqref{defsdw}, \eqref{defsdwgr} 
\\ \hline
$\mathbb{T}(d; \delta_d), \mathbb{T}(d)_v$ & preprojective quasi-BPS categories 
& Equation \eqref{def:qbps:double} 
\\ \hline
$\langle\,,\,\rangle$ & pairing &
Equation \ref{newpairing}
\\ \hline
$\lambda_1,\ldots,\lambda_m$ & cocharacters associated to codimension one faces of $\textbf{W}(d)$ &
Equation \eqref{lambdacha}
\\ \hline
$\mathrm{wt}(A), \mathrm{wt}^{\mathrm{max}}(A)$ & weight set of a complex $A$ &
Equation \eqref{wtsetA}
\\ \hline
$\mathbb{M}^\ell(d;\delta_d)$ & magic category for a linearization $\ell$ & Equation \eqref{BPS:Sl} 
\\ \hline
$\mu_1,\ldots,\mu_N$ & cocharacters associated to Kempf-Ness loci & Equation \eqref{mucha}
\\ \hline
$\mathbb{S}^\ell(d;\delta_d), \mathbb{T}^\ell(d; \delta_d)$ & quasi-BPS categories for a linearization $\ell$ & Equations \eqref{defuqias}, \eqref{def:int}
\\ \hline
$\mathrm{forg}$ & forget-the-grading functor & Equation \eqref{def:forg}
\\ \hline
$\theta_i$ & weights associated to a partition of $d$ &
Equations \eqref{def:deltai}, \eqref{defthetai}
\\ \hline
$L^d_{\delta_d}$ & set of partitions &
Subsection \ref{subsub421} 
\\ \hline
$T^d_w$ & set of partitions & Equation \eqref{varphi}
\\ \hline
$Q^\gimel, \X^\gimel(d)$ etc. & auxiliary quiver, its space of representations etc. &
Subsection \ref{symmquivers}
\\ \hline
$\mathscr{P}(d)^{\mathrm{red}}$ & reduced preprojective stack &
Equation \eqref{def:Pd}
\\ \hline
$\X_0(d)$ & reduced stack of representations of the tripled quiver &
Subsection \ref{subsec52}
\\ \hline
$\mathrm{Tr}\,W_0$ & function on $\X_0(d)$ &
Subsection \ref{subsec52}
\\ \hline
$\pi_C, p_C$ & maps from the critical locus of $\mathrm{Tr}\,W_0$ &
Diagram \eqref{dia:crit}
\\ \hline
$\mathscr{N}_{\mathrm{nil}}$ & substack where a linear map is nilpotent &
Equation \ref{def:Nnil}
\\ \hline
$\mathbb{S}^{\mathrm{gr}}(d)^{\mathrm{red}}_v$, $\mathbb{T}(d)^{\mathrm{red}}_v$ & quasi-BPS categories for reduced stacks &
Equations \eqref{def:quasiBPSred}, \eqref{def:quasiBPSredgr}
\\ \hline
$\alpha_{Q^\circ}$ & quantity computed from the number of edges of a quiver &
Equation \eqref{subsec:relproper}
\\ \hline
	\end{tabular}
}
	\vspace{.5cm}
	 \caption{Notation used in the paper}
	 \label{table:notation}
\end{figure}

\section{Preliminaries}\label{section:prelim}
\subsection{Matrix factorizations}\label{subsection:matrixfactorizations}
We briefly review the definition of categories of matrix factorizations. For more details, see \cite[Subsection 2.6]{PTzero}.

Consider a smooth quotient stack $\X=X/G$, where $G$ is an algebraic group acting on a smooth affine scheme $X$, with a regular function $f\colon\X\to\mathbb{C}$. Consider the category of matrix factorizations by 
\[\mathrm{MF}(\X, f),\] 
whose objects are tuples 
\begin{align}\label{obj:AB}
(\alpha \colon A \leftrightarrows B \colon \beta), \ 
\alpha \circ \beta=\cdot f, \ \beta \circ \alpha=\cdot f,
\end{align}
where $A, B \in \Coh(\X)$. 
If $\mathbb{M}\subset D^b(\X)$ is a subcategory, let 
\begin{align}\label{sub:MMF}\mathrm{MF}(\mathbb{M}, f)\subset \mathrm{MF}(\X, f)
\end{align}
the subcategory consisting of totalizations of tuples (\ref{obj:AB}) 
with $A, B \in \mathbb{M}$, 
see~\cite[Subsection~2.6]{PTzero} for the precise definition. 
If $\mathbb{M}$ is generated by
a set of vector bundles $\{\mathscr{V}_i\}_{i\in I}$
on $\X$, then (\ref{sub:MMF}) is generated by 
matrix factorizations whose 
factors are direct sums of vector bundles 
from $\{\mathscr{V}_i\}_{i\in I}$, 
see~\cite[Lemma~2.3]{PTzero}. 

Given an action of $\mathbb{C}^*$ on $\X$ for which $f$ is of weight $2$, we also consider the category of graded matrix factorizations
$\mathrm{MF}^{\mathrm{gr}}(\X, f)$.
Its objects consist of tuples (\ref{obj:AB}) where 
$A, B$ are $\mathbb{C}^{\ast}$-equivariant
and $\alpha, \beta$ are of $\mathbb{C}^{\ast}$-weight one. 
For a subcategory $\mathbb{M}\subset D^b(\X)$, we define 
$\mathrm{MF}^{\mathrm{gr}}(\mathbb{M}, f)\subset \mathrm{MF}^{\mathrm{gr}}(\X, f)$
similarly to (\ref{sub:MMF}). 

Let $\mathscr{Z}\subset \X$ be a closed substack.
A matrix factorization $F$ in $\mathrm{MF}(\X, f)$ has support in $\mathscr{Z}$ if its restriction to $\mathrm{MF}(\X\setminus\mathscr{Z}, f)$ is zero. Every matrix factorization $F$ has support included in $\mathrm{Crit}(f)\subset \X$,
so for any open substack $\mathscr{U} \subset \X$ which contains 
$\mathrm{Crit}(f)$, the following restriction functor is an equivalence: 
\begin{align}\label{rest:MFU}
    \mathrm{MF}(\X, f) \stackrel{\sim}{\to} \mathrm{MF}(\mathscr{U}, f). 
\end{align}

We deduce semiorthogonal decompositions for a quiver with potential 
$(Q,W)$ from the case of zero potential, see for example \cite[Proposition 2.5]{PTzero}, \cite[Proposition 2.1]{P0}. We extensively use the Koszul equivalence, 
see Theorem~\ref{thm:Koszul}.

We consider either ungraded categories of matrix factorizations or graded categories which are Koszul equivalent to derived categories of bounded complexes of coherent sheaf on a quasi-smooth stack.
When considering the product of two categories of matrix factorizations, which is in the context of the Thom-Sebastiani theorem, we consider the product of dg-categories over $\mathbb{C}(\!(\beta)\!)$ for $\beta$ of homological degree $-2$ in the ungraded case, see \cite[Theorem 4.1.3]{Preygel}, and the product of dg-categories over $\mathbb{C}$ in the graded case, see \cite[Corollary 5.18]{MR3270588} (alternatively in the graded case, one can use the Koszul equivalence).

\subsection{Quivers, weights, and partitions}\label{subsec22two}
\subsubsection{Basic notions}\label{subsub221}
Let $Q=(I,E)$ be a quiver, i.e. a directed graph with set of vertices
$I$ and set of edges $E$. 
Let $d=(d^a)_{a\in I}\in\mathbb{N}^I$ be a dimension vector. Denote by 
\[\X(d)=R(d)/G(d)\]
the stack of representations of $Q$ of dimension $d$.
Here $R(d)$, $G(d)$ are given by 
\begin{align*}
    R(d)=\bigoplus_{(a\to b) \in E}\Hom(V^a, V^b), \ 
    G(d)=\prod_{a\in I}GL(V^a). 
\end{align*}
We say that $Q$ is \textit{symmetric} if for any $a, b \in I$, the number of arrows 
from $a$ to $b$ is the same as those from $b$ to $a$. In this case, 
$R(d)$ is a self-dual $G(d)$-representation. 
We have the good moduli space morphism (or GIT quotient)
\begin{align*}
    \pi_{X,d}=\pi_X \colon \X(d) \to X(d):=R(d)\ssslash G(d). 
\end{align*}

For a quiver $Q$, let $\mathbb{C}[Q]$ be its path algebra. 
A potential $W$ of a quiver $Q$ is an element 
\begin{align*}W \in \mathbb{C}[Q]/[\mathbb{C}[Q], \mathbb{C}[Q]].
\end{align*}
A pair 
$(Q, W)$ is called a \textit{quiver with potential}. 
Given a potential $W$, there is a regular function 
\begin{align}\label{def:regularfunction}
    \Tr W \colon \X(d) \to \mathbb{C}. 
\end{align}
By the property of the good moduli space, the function
$\Tr W$ 
factors through the good moduli space $\Tr W \colon \X(d) \xrightarrow{\pi_{X,d}} X(d) \to \mathbb{C}$. 

We will consider the derived category $D^b(\X(d))$ of coherent sheaves on 
$\X(d)$ and the category of matrix factorizations 
$\mathrm{MF}(\X(d), \Tr W)$. 
Since the diagonal torus $\mathbb{C}^{\ast} \subset T(d)$ acts on $R(d)$ trivially, 
there are orthogonal decompositions 
\begin{align}\label{orthogonaldecomposition}
    D^b(\X(d))=\bigoplus_{w\in \mathbb{Z}} D^b(\X(d))_w, \ 
    \mathrm{MF}(\X(d), \Tr W)=\bigoplus_{w\in \mathbb{Z}} \mathrm{MF}(\X(d), \Tr W)_w
\end{align}
where each summand corresponds to the diagonal $\mathbb{C}^{\ast}$-weight $w$-part. 
\subsubsection{The weight lattice}\label{subsec:theweightlattice}
We fix a maximal torus $T(d)$ of $G(d)$. Let $M(d)$ be the weight lattice of $T(d)$. 
For $a\in I$ and for $d^a\in\mathbb{N}$, denote by $\beta^a_i$ for $1\leq i\leq d^a$ the weights of the standard representation of $T(d^a)$.
We have 
\begin{align*}
    M(d)=\bigoplus_{a \in I} \bigoplus_{1\leq i\leq d^a} \mathbb{Z} \beta^a_i. 
\end{align*}
By abuse of notation, for a $T(d)$-representation $U$
we also denote by $U \in M(d)$ the sum of weights 
in $U$, equivalently, the 
class of the character $\det U$. 
A weight \[\chi=\sum_{a\in I}\sum_{1\leq i\leq d^a} x_i^a \beta^a_i\] is 
\textit{dominant (or antidominant)}
if $x_i^a \leq x_{i+1}^a$ (or $x_i^a \geq x_{i+1}^a$) for all $a\in I$ and $1\leq i\leq d^a$. 
For a dominant weight $\chi$, we denote by $\Gamma_{G(d)}(\chi)$ the 
irreducible representation of $G(d)$ with highest weight $\chi$. 
The dominant or antidominant cocharacters are 
also defined for elements of 
the cocharacter lattice $N(d)=\Hom(M(d), \mathbb{Z})$. 
We denote by $1_d \in N(d)$ the diagonal cocharacter. 

We denote by $M(d)_0\subset M(d)$ the hyperplane of weights with sum of coefficients equal to zero, 
and set 
\begin{align*}
M(d)_{\mathbb{R}}:=M(d)\otimes_{\mathbb{Z}}\mathbb{R}, \ M(d)_{0,\mathbb{R}}:=M(d)_0\otimes_{\mathbb{Z}}\mathbb{R}.
\end{align*}
Note that $M(d)_0$ is the weight lattice of the 
subtorus $ST(d) \subset T(d)$ defined by 
     \begin{align*}
    ST(d):=\text{ker}(\det\colon T(d)\to \mathbb{C}^*), 
    \ (g^a)_{a \in I} \stackrel{\det}{\mapsto} \prod_{a \in I} \det (g^a). 
    \end{align*}
We denote by $\langle \,,\,\rangle\colon N(d)\times M(d)\to \mathbb{Z}$ the natural pairing, and we use the same notation for its real version. If $\lambda$ is a cocharacter of $T(d)$ and $V$ is a $T(d)$-representation, we may abuse notation and write
\[\langle \lambda, V\rangle=\langle \lambda, \det(V)\rangle\]
to ease notation.

We denote by $W_d$ the Weyl group of $G(d)$ 
and $M(d)^{W_d} \subset M(d)$ the Weyl-invariant 
subset. 
%We denote by $1_d$ the cocharacter of $T(d)$ which 
%acts on $\beta_i^a$ by weight one. 
For $d=(d^a)_{a\in I}$,
let $\underline{d}=\sum_{a\in I}d^a$ be its total length. Define the Weyl-invariant weights in $M(d)_{\mathbb{R}}$: 
\begin{align*}
    \sigma_{d}:=\sum_{a \in I} \sum_{1\leq i \leq d^a}\beta^a_i, \ 
    \tau_d:=\frac{\sigma_d}{\dd}.
\end{align*}
We denote by $\mathfrak{g}(d)$ the Lie algebra of $G(d)$, 
and by $\rho$ half the sum of positive roots of $\mathfrak{g}(d)$: 
\begin{align*}
    \rho=\frac{1}{2} \sum_{a \in I}\sum_{1\leq i<j \leq d^a} (\beta_j^a -\beta_i^b). 
\end{align*}

\subsubsection{}\label{id}
 Let $(d_i)_{i=1}^k$ be a partition of $d$. There is an identification \[\bigoplus_{i=1}^k M(d_i)\cong M(d),\] where 
$\beta^a_1, \ldots, \beta^a_{d_1}$ correspond
to the 
the weights
of standard representation of $GL(d^a_1)$ in $M(d^a_1)$ for $a\in I$, etc.

% \subsubsection{}\label{comparisoncocharacters}\label{subsubsub:order}
% For cocharacters $\lambda,\mu \colon \C^*\to SL(d)\cap T(d)$, 
% %\textcolor{red}{(Here and ond other places, I think the cocharacter should be a map $\lambda \colon \mathbb{C}^{\ast} \to T(d)$ rather than $\mathbb{C}^{\ast} \to GL(d)$or $\mathbb{C}^{\ast} \to SG(d)$, since we often take the pairng with an element of $M$ which is a character of $T(d)$and such a pairing is not defined for a cocharacter $\mathbb{C}^{\ast} \to GL(d)$not lying on $T(d)$. I have changed all other `cocharacter of $GL(d)$' to `cocharacter of $T(d)$'. Please let me know that a cocharacter not lying in $T(d)$ is necessary somewhere.)}
% we write $\mu\succeq \lambda$ 
% if for every weight $\beta \in \mathcal{W}$ with $\langle \lambda, \beta \rangle> 0$, we have that $\langle \mu, \beta\rangle > 0$. If both $\mu\succeq \lambda$ and $\lambda\succeq \mu$, write $\mu\sim\lambda$.

\begin{defn}\label{compa}
Let $\underline{e}=(e_i)_{i=1}^l$ and $\dd=(d_i)_{i=1}^k$ be two partitions of $d\in \mathbb{N}^I$. We write $\ee\geq\dd$
%that \textit{$\underline{e}$ is a refinement of $\underline{d}$} 
if there exist integers \[a_0=0< a_1<\cdots<a_{k-1}\leq a_k=l\] such that for any $0\leq j\leq k-1$, we have
\[\sum_{i=a_{j}+1}^{a_{j+1}} e_i=d_{j+1}.\]
%We write $\ee\geq\dd$. 
% There is a similarly defined order on pairs $(d,w)\in\mathbb{N}\times\mathbb{Z}$.
\end{defn}

%If $\lambda$ and $\mu$ are cocharacters with associated partitions $\dd$ and $\underline{e}$ such that $\mu\geq \lambda$, then $\ee\geq\dd$.

We next define a tree which is useful in decomposing dominant weights of $M(d)_\mathbb{R}$.

\begin{defn}\label{tree}
We define 
$\mathcal{T}$ to be the unique (oriented) tree such that: 
\begin{enumerate}
\item each vertex is indexed by a partition $(d_1, \ldots, d_k)$ of 
	some $d \in \mathbb{N}^I$, 
 \item for each $d \in \mathbb{N}^I$, there is a unique vertex 
	indexed by the partition $(d)$ of size one, 
  \item if $\bullet$ is a vertex indexed by $(d_1, \ldots, d_k)$ 
	and $d_m=(e_1, \ldots, e_s)$ is a partition of $d_m$ for some $1\leq m \leq k$, then there is a unique vertex $\bullet'$ indexed by 
	$(d_1, \ldots, d_{m-1}, e_1, \ldots, e_s, d_{m+1}, \ldots, d_k)$ and
	with an edge from $\bullet$ to $\bullet'$, and 
 \item all edges in $\mathcal{T}$ are as in (3).
	\end{enumerate}
 \end{defn}
Note that  
 each partition $(d_1, \ldots, d_k)$ of 
	some $d \in \mathbb{N}^I$ 
 gives an index of some (not necessary unique) vertex.  
A subtree $T \subset \mathcal{T}$ is called a \textit{path of partitions} 
if it is connected, contains a vertex indexed by $(d)$ for some 
$d \in \mathbb{N}^I$ and 
a unique end vertex $\bullet$. 
The partition $(d_1, \ldots, d_k)$ at the end vertex $\bullet$
is called the associated partition of $T$. 
We define the Levi group associated to $T$ to be 
\begin{align*}
	L(T):=\times_{i=1}^k G(d_i). 
	\end{align*}
 % Note that each edge $\ell\in \mathcal{T}$ corresponds to a partition of some natural number: if $\ell$ connects $(d_1, \ldots, d_k)$ and $(d_1, \ldots, d_{m-1}, e_1, \ldots, e_s, d_{m+1}, \ldots, d_k)$ as in (3), then $\ell$ corresponds to the partition $(e_1,\ldots, e_s)$ of $d_m$.

%$\mathcal{T}=(\mathcal{I}, \mathcal{E})$ as follows: the set $\mathcal{I}$ of %vertices has elements corresponding to partitions $\dd$ of any dimension vector %$d\in\mathbb{N}$ and the set $\mathcal{E}$ has edges from $\dd=(d_i)_{i=1}^k$ %to $\ee$ if $\ee$ is a partition of $d_i\neq d$ for some $1\leq i\leq k$. 
%A set $T\subset \mathcal{I}$ is called a \textit{tree of partitions} if it is %finite and has a unique element $v\in T$ with in-degree zero.
%Let $\Omega\subset T$ be the set of vertices with out-degree zero. Define the %Levi group associated to $T$:
%$$L(T):=\prod_{\ee\in \Omega} L(\ee).$$

\subsubsection{Framed quivers}\label{subsub224}
Consider a quiver $Q=(I, E)$. Define the \textit{framed quiver}: 
\begin{align*}
    Q^f=(I^f, E^f)
\end{align*}
with set of vertices $I^f=I\sqcup\{\infty\}$ and set of edges $E^f=E\sqcup \{e_a\mid a\in I\}$,
where $e_a$ is an edge from $\infty$ to $a\in I$. 
Let $V(d)=\bigoplus_{a\in I}V^a$, where $V^a$ is a $\mathbb{C}$-vector space of dimension $d^a$. 
Denote by \[R^f(d)=R(d)\oplus V(d)\] the affine space of representations of $Q^f$ of dimension $d$ and consider the moduli stack of framed representations
\[\X^f(d):=R^f(d)/G(d).\]
We consider GIT stability on $Q^f$
given by the character $\sigma_{\dd}$. 
It coincides with the 
King stability condition on $Q^f$ such that the (semi)stable representations of dimension $(1,d)$ are the representations of $Q^f$ with no subrepresentations of dimension $(1,d')$ for $d'$ different from $d$, 
see~\cite[Lemma~5.1.9]{T}. 
Consider the smooth variety obtained as a GIT quotient:
\[\X^f(d)^{\text{ss}}:=R^f(d)^{\text{ss}}/G(d).\]
%Let $W$ be a potential of $Q$.
% We refer to the categories 
% \[D^b\left(\X^f(d)^{\text{ss}}\right)\text{ and } \mathrm{MF}\left(\X^f(d)^{\text{ss}}, \mathrm{Tr}\,W\right)\] as the \textit{Donaldson-Thomas (DT) categories} of $Q$ and $(Q,W)$, respectively.

\subsubsection{The categorical Hall product}\label{subsub225}
For a cocharacter $\lambda \colon \mathbb{C}^{\ast} \to T(d)$ and a 
$T(d)$-representation $V$, let $V^{\lambda\geq 0}\subset V$ be the subspace generated by weights $\beta$ such that $\langle \lambda, \beta\rangle\geq 0$, and let $V^\lambda\subset V$ be the subspace generated by weights $\beta$ such that $\langle \lambda, \beta\rangle=0$. 
% we set
% \begin{align*}
%     V^{\lambda \geq 0}=\sum_{\langle \lambda, \beta\rangle \geq 0}\beta \in M(d), \ 
%     V^{\lambda}=\sum_{\langle \lambda, \beta \rangle=0}
%     \beta \in M(d)
% \end{align*}
% where the sums are after weights $\beta$ of $V$.
We denote by $G(d)^{\lambda} \subset G(d)^{\lambda \geq 0}$
the associated Levi and parabolic subgroup of $G(d)$. 
If $V$ is a $G(d)$-representation, 
consider the quotient stack $\X=V/G(d)$ and set:
\begin{align*}
    \X^{\lambda \geq 0}=V^{\lambda \geq 0}/G(d)^{\lambda \geq 0}, \ 
    \X^{\lambda}=V^{\lambda}/G(d)^{\lambda}. 
\end{align*}
The 
projection $V^{\lambda \geq 0} \twoheadrightarrow V^{\lambda}$ and the 
inclusion $V^{\lambda \geq 0} \hookrightarrow V$ 
induce maps 
\begin{align}\label{dia:att}
\X^{\lambda} \leftarrow \X^{\lambda \geq 0} \to \X. 
\end{align}
We apply the above construction for $V=R(d)$ to obtain maps:
\[\X(d)^\lambda=\times_{i=1}^k\X(d_i)\xleftarrow{q_\lambda}\X(d)^{\lambda\geq 0}\xrightarrow{p_\lambda}\X(d).\]
Suppose that $\lambda$ is antidominant 
with associated partition $(d_i)_{i=1}^k$ of $d\in\mathbb{N}^I$, meaning that 
\[\X(d)^\lambda=\times_{i=1}^k\X(d_i).\]
The multiplication for the categorical Hall algebra of $(Q, 0)$ (or of $(Q,W)$ for a potential $W$ of $Q$ and possibly a grading) is defined by the functors \cite{P0}, where $\bullet\in\{\emptyset, \mathrm{gr}\}$:
\begin{align}\label{def:Hallprod}
&m_\lambda:=p_{\lambda*}q_\lambda^{\ast} \colon 
\boxtimes_{i=1}^k D^b(\X(d_i)) \to D^b(\X(d)), \\
\notag &m_\lambda:=p_{\lambda*}q_\lambda^{\ast} \colon 
\boxtimes_{i=1}^k \mathrm{MF}^\bullet(\X(d_i), \Tr W) \to \mathrm{MF}^\bullet(\X(d), \Tr W). 
\end{align}
% where $p_\lambda$ and $q_\lambda$ are the maps
% in the diagram (\ref{dia:att})
% \[\X(d)^\lambda=\times_{i=1}^k\X(d_i)\xleftarrow{q_\lambda}\X(d)^{\lambda\geq 0}\xrightarrow{p_\lambda}\X(d).\]
% %Consider weights $(w_i)_{i=1}^k\in\mathbb{Z}^k$. 

\subsubsection{Doubled quiver}\label{subsec22}
%In this section, we recall the construction of double quivers and of tripled quivers with potential.
Let $Q^\circ=(I, E^\circ)$ be a quiver.
Its \textit{doubled quiver} is the quiver
\begin{align*}
Q^{\circ, d}=(I, E^{\circ, d})
\end{align*}
with set of edges $E^{\circ, d}=\{e, \overline{e} \mid e \in E^{\circ}\}$,
where $\overline{e}$ is is the edge with the opposite orientation of $e\in E^\circ$. 
Consider the relation $\mathscr{I}$ of $\mathbb{C}[Q^{\circ, d}]$:
\begin{align}\label{def:relation}
    \mathscr{I}:=\sum_{e \in E^{\circ}}[e, \overline{e}] \in \mathbb{C}[Q^{\circ, d}]. 
\end{align}
For $d\in \mathbb{N}^I$, 
consider the stack of representations of the quiver $Q^{\circ, d}$ of dimension $d$:
\[\mathscr{Y}(d):=\overline{R}(d)/G(d):=T^*R^\circ(d)/G(d)\] with good moduli space map
\[\pi_{Y,d}=\pi_Y\colon \mathscr{Y}(d)\to Y(d):=\overline{R}(d)\ssslash G(d).\]
Note that
\begin{align*}
\overline{R}(d):=T^*R^\circ(d)=\bigoplus_{(a\to b) \in E^{\circ}}\Hom(V^a, V^b) \oplus 
\Hom(V^b, V^a). 
\end{align*}
For $x \in T^*R^\circ(d)$ and an edge $(a \to b) \in E^{\circ}$, 
consider the components 
$x(e) \in \Hom(V^a, V^b)$ and $x(\overline{e}) \in \Hom(V^b, V^a)$. 
The relation (\ref{def:relation}) determines the moment map 
\begin{align}\label{def:moment}
\mu \colon T^*R^\circ(d) \to \mathfrak{g}(d), \ 
x \mapsto \sum_{e \in E^{\circ}}[x(e), x(\overline{e})]. 
\end{align}
Let $\mu^{-1}(0)$ be the derived zero locus of $\mu$.
%Here $\mathfrak{g}(d)$ is the Lie algebra of $G(d)$. 
Define the stack of representations of $Q^{\circ, d}$ with relation $\mathscr{I}$, alternatively of the preprojective algebra $\Pi_{Q^\circ}:=\mathbb{C}[Q^{\circ, d}]/(\mathscr{I})$:
\begin{align}\label{def:Pj}
   j \colon \mathscr{P}(d):=\mu^{-1}(0)/G(d) \hookrightarrow 
    \mathscr{Y}(d).
\end{align}
The stack $\mathscr{P}(d)^{\mathrm{cl}}$ has a good moduli space map:
\begin{equation}\label{gmdpd}
    \pi_{P,d}=\pi_P\colon \mathscr{P}(d)^{\mathrm{cl}}\to P(d):=\mu^{-1}(0)^{\rm{cl}}\ssslash G(d).
\end{equation}

% Here $\mu^{-1}(0)$ is the derived zero locus of $\mu$. 
% The stack $\mathscr{Y}(d)$ is the moduli stack of representations of $Q^{\circ, d}$ and 
% $\mathscr{P}(d)$ is the derived moduli stack of $(Q^{\circ, d}, \mathscr{I})$-representations. 
Let $\lambda$ be an antidominant cocharacter of $T(d)$
corresponding to 
the decomposition $(d_i)_{i=1}^k$ of $d$. 
Similarly to (\ref{def:Hallprod}), there is a categorical 
Hall product \cite{PoSa, VaVa}, see also (\ref{cathall:P}):
\begin{align}\label{cathall:Pd}
m_{\lambda} \colon \boxtimes_{i=1}^k 
D^b(\mathscr{P}(d_i)) \to D^b(\mathscr{P}). 
\end{align}

%For $d\in \mathbb{N}^I$, consider the stack of representations of dimension $d$ of $Q$:
%\[\X^\circ(d)=R^\circ(d)/G(d).\]
%We consider the stack of representations of the preprojective algebra of $Q$:
%\[\mathscr{P}(d):=T^*\X^\circ(d):=\mu^{-1}(0)/G(d),\]
%where \[\mu\colon \overline{R}(d):=T^*R^\circ(d)\to \mathfrak{g}(d)^\vee\cong \mathfrak{g}(d)\] is the moment map and $\mu^{-1}(0)$ is the %derived zero of $\mu$. 
%For an edge $e$ of $Q^\circ$, denote by $\overline{e}$ the edge with opposite orientation.
%Define the multiset $E^{\circ, d}=\{e, \overline{e}|\,e\in E\}$.
%The affine space $T^*R^\circ(d)$ can be also described as the space of representations $\overline{R}(d)$ of \textit{the double quiver} $Q^{\circ, d}=(I, E^{\circ, d})$. 
\begin{remark}
% The algebra $\mathbb{C}[Q^{\circ, d}]/\mathscr{I}$
% is called \textit{preprojective algebra} associated with $Q^{\circ}$. 
As mentioned in the introduction, moduli stacks of representations of preprojective algebras are interesting for at least two reasons: they describe locally the moduli stack of (Bridgeland semistable) sheaves on a Calabi-Yau surface, and their K-theory (or Borel-Moore homology) can be used to construct positive halves of quantum affine algebras (or of Yangians).
\end{remark}

% To state some of the results in this paper, it is useful to introduce the following integers.  
% \begin{defn}\label{def:alpha}
%     Consider a quiver $Q^{\circ}=(I, E^{\circ})$. 
% For $a, b \in I$, let $\delta_{ab}=1$ if $a=b$ and $0$ otherwise. 
% Define 
% \begin{align*}
%     \alpha_{a, b} :=\sharp(a \to b\text{ in }E^\circ)+\sharp(b \to a\text{ in }E^\circ)-2\delta_{ab},\\
%     \alpha_{Q^{\circ}}:=\mathrm{min}\{\alpha_{a, b} : a, b \in I\}
% \end{align*}
% Here $\delta_{ab}=1$ if $a=b$ and $0$ otherwise. 
% We also set 
% $\alpha_{Q^{\circ}}:=\mathrm{min}\{\alpha_{a, b} : a, b \in I\}$. 
% \end{defn}

\subsubsection{Tripled quivers with potential}\label{subsub:triple}
Consider a quiver $Q^{\circ}=(I, E^{\circ})$. Its \textit{tripled quiver}
\begin{align*}
    Q=(I, E)
\end{align*} has set of edges $E=E^{\circ, d}\sqcup\{\omega_a \mid a\in I\}$,
where $\omega_a$ is a loop at the vertex $a \in I$. 
%We next define \textit{the tripled quiver with potential} $(Q,W)$. 
%The quiver $Q=(I, E)$, where $E=E^{\circ, d}\sqcup\{\omega_a|\,a\in I\}$, where $\omega_a$ is a loop at the vertex $\in I$. 
The tripled potential $W$ (which is a potential of $Q$) is defined as follows:
\begin{equation}\label{poten:W}
W:=\left(\sum_{a\in I}\omega_a\right)
\left(\sum_{e\in E^\circ}[e, \overline{e}]\right)\in \mathbb{C}[Q].
\end{equation}
The quiver with potential $(Q, W)$ constructed above 
is called the \textit{tripled quiver with potential} associated to $Q^{\circ}$. 

For $d \in \mathbb{N}^I$, let $\X(d)$ be the stack of representations of $Q$.
It is given by 
\begin{align}\label{stack:triple}
    \X(d)=\left(T^*R^\circ(d) \oplus \mathfrak{g}(d)\right)/G(d)=\left(\overline{R}(d) \oplus \mathfrak{g}(d)\right)/G(d). 
\end{align}
Recall the function $\Tr W$ on $\X(d)$ 
from (\ref{def:regularfunction}). 
The critical locus 
\begin{align*}
    \mathrm{Crit}(\Tr W) \subset \X(d)
\end{align*}
is the moduli stack of $(Q, W)$-representations, alternatively of the Jacobi algebra $\mathbb{C}[Q]/\mathscr{J}$, where $\mathscr{J}$ is the two-sided ideal generated by the partial derivatives $\partial W/\partial e$ for all $e\in E$. 

Consider the grading on $\X(d)$ which is of weight $0$ on $\overline{R}(d)$ and of weight $2$
on $\mathfrak{g}(d)$. 
%the linear maps corresponding to the loops $\{\omega_a|\,a\in I\}$ and fixes the linear maps in $E^{\circ, d}$. 
The Koszul equivalence (which we recall later, in Theorem~\ref{thm:Koszul})
says that there is an equivalence:
\begin{equation}\label{Kosz}
\Theta \colon 
D^b\left(\mathscr{P}(d)\right)\stackrel{\sim}{\to} \mathrm{MF}^{\mathrm{gr}}\left(\X(d), \mathrm{Tr}\,W\right).
\end{equation}
%We say $(Q,W)$ is a tripled quiver with potential if it is obtained as above for some quiver $Q^\circ$.
\begin{remark}
Tripled quivers with potential are interesting for at least two reasons: they model the local structure of moduli stacks of semistable sheaves on an arbitrary  CY3, and they can be used, in conjunction with dimensional reduction techniques (such as the Koszul equivalence \eqref{Kosz}), to study (moduli of representations of) preprojective algebras.
\end{remark}

% The picture of $\mathcal{T}_3$ is depicted below. 

% \begin{align*}
% \xymatrix{
% & (3) \ar[ld] \ar[dd] \ar[rd]  & \\
% (2, 1) \ar[d] & & (1, 2) \ar[d] \\
% (1, 1, 1) & (1, 1, 1) & (1, 1, 1)
% }
% \end{align*}

\subsection{The Koszul equivalence}\label{subsec:Koszul}
Let $Y$ be an affine smooth scheme with an action of a reductive group $G$.
Consider the quotient stack $\mathscr{Y}=Y/G$. Let $\mathscr{V} \to \mathscr{Y}$ be
a vector bundle and let $s\in \Gamma(\mathscr{Y}, \mathscr{V})$. 
Let $\mathscr{P}=s^{-1}(0)$ be the derived 
zero locus of $s$, 
so its structure complex is the 
Koszul complex
\[\mathcal{O}_{\mathscr{P}}:=\left(\mathrm{Sym}_{\mathcal{O}_{\mathscr{Y}}}(\mathscr{V}^{\vee}[1]), d_{\mathscr{P}}\right)\]
with differential $d_{\mathscr{P}}$ induced by $s$. 
Let 
$\X=\mathrm{Tot}_\X\left(\mathscr{V}^{\vee}\right)$ and define 
the function $f$ by
\begin{align*}
f \colon \mathscr{X} \to \mathbb{C}, \ f(y, v)=\langle s(y), v \rangle
\end{align*}
for $y \in \mathscr{Y}$ and $v \in \mathscr{V}^{\vee}|_{y}$.
There are maps: 
\begin{align}\label{jeta}
    \mathscr{P} \stackrel{j}{\hookrightarrow} \mathscr{Y} \stackrel{\eta}{\leftarrow} 
    \mathscr{X} \stackrel{f}{\to} \mathbb{C}, 
\end{align}
where $j$ is the natural inclusion and $\eta$ is the natural projection. 
The following is the Koszul equivalence: 
\begin{thm}\emph{(\cite{I, Hirano, T})}\label{thm:Koszul}
There are equivalences:
\begin{align*}
    \Theta \colon D^b(\mathscr{P}) \stackrel{\sim}{\to} \mathrm{MF}^{\rm{gr}}(\mathscr{X}, f), \ 
    \Theta \colon \Ind D^b(\mathscr{P}) \stackrel{\sim}{\to} \mathrm{MF}^{\rm{gr}}_{\rm{qc}}(\mathscr{X}, f).
\end{align*}
The grading on the right hand side has weight zero on $\Y$ and weight $2$ on
the fibers of $\eta\colon \mathscr{X} \to \mathscr{Y}$. 
    \end{thm}
    The equivalence $\Theta$ is given by the functor
  \begin{align}\label{equiv:theta}
  \Theta(-)=(-) \otimes_{\mathcal{O}_{\mathscr{P}}}\mathcal{K},
  \end{align}
  where $\mathcal{K}$ is the Koszul factorization $\mathcal{K}:=\left(\mathcal{O}_{\mathscr{P}} \otimes_{\mathcal{O}_{\mathscr{Y}}}\mathcal{O}_{\mathscr{X}}, d_\mathcal{K}\right)$
with $d_{\mathcal{K}}=d_{\mathscr{P}} \otimes 1 +\kappa$, 
where $\kappa \in \mathcal{V}^{\vee} \otimes \mathcal{V}$
corresponds to $\id \in \Hom(\mathcal{V}, \mathcal{V})$, see~\cite[Theorem~2.3.3]{T}. 
The following lemma is easily proved 
from the above description of $\Theta$. 
\begin{lemma}\label{lem:genJ}
    Let $\{V_j\}_{j\in J}$ be a set of
    $G$-representations
    and let $\mathbb{S} \subset \mathrm{MF}^{\rm{gr}}(\X, f)$
    be the subcategory generated by 
    matrix factorizations whose 
    factors are direct sums of vector bundles $\mathcal{O}_{\X} \otimes V_j$. 
    Then an object $\mathcal{E} \in D^b(\mathscr{P})$
    satisfies $\Theta(\mathcal{E}) \in \mathbb{S}$
    if and only if $j_{\ast}\mathcal{E} \in D^b(\mathscr{Y})$
    is generated by $\mathcal{O}_{\mathscr{Y}} \otimes V_j$ for $j\in J$. 
\end{lemma}
\begin{proof}
    The same argument used to prove~\cite[Lemma~4.5]{PTzero} applies here. 
\end{proof}
    
For the later use, we will compare internal homomorphisms under Koszul equivalence. 
For $\mathcal{E}_1, \mathcal{E}_2 \in D^b(\mathscr{P})$, 
there exists an internal homomorphism, see~\cite[Remark~3.4]{MR3037900}:
\begin{align*}
    \mathcal{H}om(\mathcal{E}_1, \mathcal{E}_2) \in D_{\rm{qc}}(\mathscr{P})
\end{align*}
It satisfies the following tensor-Hom adjunction: for any $A \in D_{\rm{qc}}(\mathscr{P})$, there are natural isomorphisms: 
\begin{align*}
    \Hom_{D_{\rm{qc}}(\mathscr{P})}(A, \mathcal{H}om(\mathcal{E}_1, \mathcal{E}_2))
    \cong \Hom_{\Ind D^b(\mathscr{P})}(A \otimes \mathcal{E}_1, \mathcal{E}_2). 
\end{align*}
On the other side of the Koszul equivalence, consider the internal Hom of $\mathcal{F}_1, \mathcal{F}_2 \in \mathrm{MF}^{\rm{gr}}(\mathscr{X}, f)$:
\begin{align*}
\mathcal{H}om(\mathcal{F}_1, \mathcal{F}_2) \in \mathrm{MF}^{\rm{gr}}(\mathscr{X}, 0).  
\end{align*}
\begin{lemma}\label{lem:internal}
For $\mathcal{E}_1, \mathcal{E}_2 \in D^b(\mathscr{P})$, 
the equivalence $\Theta$ induces an isomorphism
\begin{align}\label{isom:interhom}
    j_{\ast}\mathcal{H}om(\mathcal{E}_1, \mathcal{E}_2) \stackrel{\cong}{\to}
    \eta_{\ast} \mathcal{H}om(\Theta(\mathcal{E}_1), \Theta(\mathcal{E}_2))    
\end{align}
in $D_{\rm{qc}}(\mathscr{Y})=\mathrm{MF}^{\rm{gr}}_{\mathrm{qc}}(\mathscr{Y}, 0)$.     
\end{lemma}
\begin{proof}
For $A \in D_{\rm{qc}}(\mathscr{Y})$, we have 
\begin{align*}
    \Hom(A, j_{\ast}\mathcal{H}om(\mathcal{E}_1, \mathcal{E}_2)) &\cong
    \Hom(j^{\ast}A, \mathcal{H}om(\mathcal{E}_1, \mathcal{E}_2)) \\
    &\cong\Hom(j^{\ast}A \otimes \mathcal{E}_1, \mathcal{E}_2).
\end{align*}
We also have 
\begin{align*}
\Hom(A, \eta_{\ast}\mathcal{H}om(\Theta(\mathcal{E}_1), \Theta(\mathcal{E}_2)) &\cong
\Hom(\eta^{\ast}A, \mathcal{H}om(\Theta(\mathcal{E}_1), \Theta(\mathcal{E}_2))) \\
&\cong \Hom(\eta^{\ast}A \otimes \Theta(\mathcal{E}_1), \Theta(\mathcal{E}_2)) \\
&\cong \Hom(\Theta(j^{\ast}A \otimes \mathcal{E}_1), \Theta(\mathcal{E}_2)), 
\end{align*}
where the last isomorphism follows from the 
 explicit form of $\Theta$ in (\ref{equiv:theta}). 
 Therefore $\Theta$ induces an isomorphism 
 \begin{align*}
      \Hom(A, j_{\ast}\mathcal{H}om(\mathcal{E}_1, \mathcal{E}_2))
      \stackrel{\sim}{\to} \Hom(A, \eta_{\ast}\mathcal{H}om(\Theta(\mathcal{E}_1), \Theta(\mathcal{E}_2))),
 \end{align*}
 which implies the isomorphism (\ref{isom:interhom}). 
\end{proof}

Let $\lambda$ be a cocharacter of $G$. Consider the sections induced by $s$:
\begin{align*}
    s^{\lambda \geq 0}\in \Gamma(\Y^{\lambda \geq 0}, \mathscr{V}^{\lambda \geq 0}), \ s^\lambda\in \Gamma(\Y^{\lambda}, 
    \mathscr{V}^{\lambda})
\end{align*}
and their derived zero loci \[\mathscr{P}^{\lambda \geq 0}:=\left(s^{\lambda\geq 0}\right)^{-1}(0),\, \mathscr{P}^{\lambda}:=\left(s^\lambda\right)^{-1}(0).\] 
% of the sections induced by $s$:
% \begin{align*}
%     s^{\lambda \geq 0}\in \Gamma(\Y^{\lambda \geq 0}, \mathscr{V}^{\lambda \geq 0}), \ s^\lambda\in \Gamma(\Y^{\lambda}, 
%     \mathscr{V}^{\lambda})
% \end{align*}
% respectively. 
Similarly to (\ref{dia:att}), consider the maps: 
\begin{align*}
\mathscr{P}^{\lambda} \stackrel{q}{\leftarrow} 
\mathscr{P}^{\lambda \geq 0} \stackrel{p}{\to} 
\mathscr{P},
\end{align*}
where $q$ is quasi-smooth
and $p$ is proper. 
Consider the functor, which is a generalization of the categorical Hall product for preprojective algebras \cite{PoSa, VaVa}:
\begin{align}\label{cathall:P}
   m_{\lambda}= p_{\ast}q^{\ast}
    \colon D^b(\mathscr{P}^{\lambda}) \to 
    D^b(\mathscr{P}).     
\end{align}
We have the following compatibility 
of categorical Hall products under 
Koszul equivalence. 
For the proof, 
see~\cite[Proposition~3.1]{P2}
or \cite[Lemma~2.4.4, 2.4.7]{T}. 
\begin{prop}\label{prop:compati:hall}
The following diagram commutes: 
\begin{align*}
    \xymatrix{
    D^b(\mathscr{P}^{\lambda}) \ar[r]^-{m_{\lambda}}
    \ar[d]^-{\Theta'}_-{\simeq} & D^b(\mathscr{P}) 
    \ar[d]^-{\Theta}_-{\simeq} \\
    \mathrm{MF}^{\rm{gr}}(\X^{\lambda}, f)
    \ar[r]^-{m_{\lambda}} & 
    \mathrm{MF}^{\rm{gr}}(\X, f). 
    }
\end{align*}
    The horizontal arrows 
    are categorical Hall products for $\mathscr{P}$ and $\X$. We denote by $\Theta$ the Koszul equivalence for both $\X$ and $\X^\lambda$, and the left vertical map is the functor $\Theta'(-):=\Theta(-)\otimes \det (\mathcal{V}^{\lambda>0})^{\vee}[\mathrm{rank}\, \mathcal{V}^{\lambda>0}]$.
    % and the left vertical arrow is the 
    % composition of the Koszul equivalence 
    % in Theorem~\ref{thm:Koszul} 
    % for $\mathscr{P}^{\lambda})$
    % and the tensor 
    % product with 
    % $\det (\mathcal{V}^{\lambda>0})^{\vee}[\mathrm{rank} \mathcal{V}^{\lambda>0}]$.
\end{prop}

\subsection{Polytopes and categories}\label{polycat}
Let $Q=(I,E)$ be a symmetric quiver and let $d=(d^a)_{a\in I}\in\mathbb{N}^I$ be a dimension vector. 
Consider the multisets of $T(d)$-weights 
\begin{align}\label{def:poly}
    \mathscr{A}&:=\{\beta^a_i-\beta^b_j \mid a,b\in I, 
    (a \to b) \in E, 1\leq i\leq d^a, 1\leq j\leq d^b\},\\
    \notag \mathscr{B}&:=\{\beta^a_i\mid a\in I, 1\leq i\leq d^a\},\\
    \notag \mathscr{C}&:=\mathscr{A}\sqcup \mathscr{B}.
\end{align}  
Here, $\mathscr{A}$ is the set of $T(d)$-weights of $R(d)$ and $\mathscr{C}$ is the set of $T(d)$-weights of $R^f(d)$.
Define the polytopes
\begin{align}\label{polyt:W}
    \mathbf{W}(d)&:=\frac{1}{2}\mathrm{sum}_{\beta\in \mathscr{A}}[0, \beta]\subset M(d)_{0,\mathbb{R}}\subset M(d)_{\mathbb{R}},\\
    \notag \mathbf{V}(d)&:=\frac{1}{2}\mathrm{sum}_{\beta\in \mathscr{C}}[0,\beta]\subset M(d)_{\mathbb{R}},
\end{align}
where the sums above are Minkowski sums in the space of weights $M(d)_{\mathbb{R}}$.

%Define the sets of weights
%\begin{align}\label{def:alambda}
%    \mathscr{A}_\lambda&:=\{\beta\in \mathscr{A} \mid \langle \lambda, \beta\rangle>0\},\\
%   \notag \mathfrak{g}_\lambda&:=\{\beta^a_i-\beta^a_j \mid 
% a\in I, 1\leq i,j\leq d^a, \langle \lambda, \beta^a_i-\beta^a_j\rangle>0\}.
%\end{align}
%We abuse notation and write
%\begin{align*}
%   R(d)^{\lambda>0}=\sum_{\beta\in\mathscr{A}_\lambda}\beta,\
%    \mathfrak{g}(d)^{\lambda>0}=\sum_{\beta\in\mathfrak{g}_\lambda}\beta,
%\end{align*} and use analogous notations for other determinants of various representations. 
%Define the weights $v_a\in\frac{1}{2}\mathbb{Z}$ by the formula \begin{equation}\label{defv} \sum_{a=1}^k v_a\tau_{d_a}=\sum_{a=1}^k w_a\tau_{d_a}-\frac{1}{2}\left(\sum_{\beta\in A_\lambda}\beta-\sum_{\beta\in\mathfrak{g}_\lambda}\beta\right). \end{equation}
%Explicitly, the weights $v_a$ are given by \[v_a=w_a-\frac{1}{2}\left(A\cdot\dd_a\left(-\sum_{b>a}\dd_b+\sum_{b<a}\dd_b\right)-\sum_{i\in I}d^i_a\left(-\sum_{b>a}d^i_b+\sum_{b<a}d^i_b\right)\right)\] for $1\leq a\leq k$.
\begin{defn}\label{def:defmdw}
For a weight $\delta_d\in M(d)_{\mathbb{R}}^{W_d}$,
%\mu\in \mathbb{R}$, 
define \begin{equation}\label{defmdw}
    \mathbb{M}(d; \delta_d)\subset D^b(\X(d))
\end{equation}
to be the full subcategory of $D^b(\X(d))$ generated by vector bundles $\mathcal{O}_{\X(d)}\otimes\Gamma_{G(d)}(\chi)$, where 
$\chi \in M(d)$ is a dominant weight such that
\begin{equation}\label{chiwd}
\chi+\rho-\delta_d\in \mathbf{W}(d).
\end{equation}
\end{defn}
The category (\ref{defmdw}) is a particular case of the noncommutative resolutions of quotient singularities introduced by \v{S}penko--Van den Bergh \cite{SVdB}. 
We may call \eqref{defmdw} a ``magic category" following \cite{hls}. 
For $\lambda$ a cocharacter of $T(d)$, define
    \begin{align}\label{nlambdadef}
n_{\lambda}=\left\langle \lambda, \det\left(\mathbb{L}_{\X(d)}|_{0}^{\lambda>0} \right)\right\rangle 
=\left\langle \lambda, \det\left(R(d)^{\vee}\right)^{\lambda\geq 0}\right\rangle-\left\langle\lambda, \det\left(\mathfrak{g}(d)^{\vee}\right)^{\lambda>0} \right\rangle. 
    \end{align}
The category \eqref{defmdw} has also the following alternative description. 
\begin{lemma}\emph{(\cite[Lemma~2.9]{hls})}\label{lemma:alt}
    The subcategory $\mathbb{M}(d; \delta_d)$ of $D^b(\X(d))$
    is generated by vector bundles $\mathcal{O}_{\X(d)} \otimes \Gamma$
    for a $G(d)$-representation $\Gamma$
    such that, for any $T(d)$-weight $\chi$ of $\Gamma$ and any cocharacter $\lambda$ of $T(d)$, we have
    that: 
    \begin{align}\label{cond:n}
     \langle \lambda, \chi-\delta_d\rangle
     \in \left[-\frac{1}{2}n_{\lambda}, \frac{1}{2}n_{\lambda}  \right].
    \end{align}
\end{lemma}
\begin{remark}\label{rmk:Mweight}
    The subcategory (\ref{defmdw}) 
    is contained in $D^b(\X(d))_w$ for $w=\langle 1_d, \delta_d \rangle$. 
    In particular, if (\ref{defmdw}) is non-zero, 
    then $\langle 1_d, \delta_d \rangle \in \mathbb{Z}$.
\end{remark}

We also define a larger subcategory corresponding 
to the polytope $\mathbf{V}(d)$. 
Let 
\begin{align}\label{def:dd}
\mathbb{D}(d; \delta_d) \subset D^b(\X(d))
\end{align}
be generated by vector bundles $\mathcal{O}_{\X(d)}\otimes\Gamma_{G(d)}(\chi)$, where $\chi$ is a dominant weight of $G(d)$ such that
\[\chi+\rho-\delta_d\in \mathbf{V}(d).\]

The following definition will be used later
in decomposing $\mathbb{D}(d; \delta_d)$ into 
categorical Hall products of $\mathbb{M}(d; \delta_d)$. 
\begin{defn}\label{def:generic2} A weight $\delta_d\in M(d)^{W_d}_\mathbb{R}$ is called \textit{good} if for all dominant cocharacters $\lambda$ such that $\langle \lambda, \beta^i_a\rangle\in\{-1, 0\}$ for all $i\in I$ and $1\leq a\leq d^i$, one has that $\langle \lambda, 2\delta_d\rangle\notin \mathbb{Z}$.\end{defn}

Next, for a quiver with potential $(Q, W)$, we 
define quasi-BPS categories:
\begin{defn}
Define the quasi-BPS category to be 
\begin{equation}\label{defsdw}
\mathbb{S}(d; \delta_d):=\mathrm{MF}(\mathbb{M}(d; \delta_d), \mathrm{Tr}\,W)\subset \mathrm{MF}(\X(d), \Tr W).
\end{equation}
\end{defn}
Suppose that $(Q, W)$ is a tripled quiver 
of a quiver $Q^{\circ}=(I, E^{\circ})$. 
In this case, 
the graded version is similarly defined 
\begin{equation}\label{defsdwgr}
\mathbb{S}^{\mathrm{gr}}(d; \delta_d):=\mathrm{MF}^{\mathrm{gr}}(\mathbb{M}(d; \delta_d), \Tr W)
\subset \mathrm{MF}^{\rm{gr}}(\X(d), \Tr W). 
\end{equation}

Next, we define quasi-BPS categories for preprojective algebras. Let $Q^\circ$ be a quiver and consider the moduli stack $\mathscr{P}(d)$ of dimension $d$ representations of the preprojective algebra of $Q^\circ$ and recall the
closed immersion $j \colon \mathscr{P}(d) \hookrightarrow 
\mathscr{Y}(d)$. 
The quasi-BPS category for $\mathscr{P}(d)$ is defined 
as follows: 
\begin{defn}\label{def:qbps:double}
Let $\widetilde{\mathbb{T}}(d; \delta_d)\subset D^b(\mathscr{Y}(d))$ be the subcategory generated by vector bundles $\mathcal{O}_{\mathscr{Y}(d)} \otimes \Gamma_{G(d)}(\chi)$, where $\chi$ is a dominant weight of $G(d)$ satisfying:
\[\chi+\rho-\delta_d\in \frac{1}{2}\text{sum}_{\beta\in \mathscr{A}}[0,\beta],\] where $\mathscr{A}$ is the set of $T(d)$-weights of $\overline{R}(d)\oplus \mathfrak{g}(d)$.
Define the preprojective quasi-BPS category (also called quasi-BPS category of preprojective algebra):
\begin{align*}
\mathbb{T}(d; \delta_d) \subset D^b(\mathscr{P}(d))
\end{align*} as the full subcategory of $D^b(\mathscr{P}(d))$ 
with objects $\mathcal{E}$ such
 that $j_{\ast}\mathcal{E}\in \widetilde{\mathbb{T}}(d; \delta_d)$. 
%is generated by the vector bundles
% $\mathcal{O}_{\mathscr{Y}(d)} \otimes \Gamma_{G(d)}(\chi)$
% for dominant weights $\chi \in M(d)$ 
% satisfying (\ref{chiwd}). 
% Here $\mathbf{W}(d)$ is defined by (\ref{polyt:W}) 
% with respect to the weights in $R(d)=\overline{R}(d) 
% \oplus \mathfrak{g}(d)$. 
\end{defn}
By Lemma~\ref{lem:genJ}, the Koszul equivalence (\ref{Kosz}) 
restricts to the equivalence 
\begin{align}\label{Koszul:qbps}
\Theta \colon \mathbb{T}(d; \delta_d) 
\stackrel{\sim}{\to} \mathbb{S}^{\rm{gr}}(d; \delta_d).
\end{align}
For $v\in \mathbb{Z}$ and $\bullet\in\{\emptyset, \mathrm{gr}\}$, we will use the following shorthand notation:
\[\mathbb{M}(d)_v:=\mathbb{M}(d; \delta_d),\, \mathbb{S}^\bullet(d)_v:=\mathbb{S}^\bullet(d; v\tau_d),\, \mathbb{T}(d)_v:=\mathbb{T}(d; v\tau_d).\] 

%Consider a subset $E^\circ\subset E$. Consider the grading on $E$ where edges on $E^\circ$ have grading $2$ and edges in $E\setminus E^\circ$ have grading zero. Assume $W$ has degree $2$ with respect to this grading. Alternatively, if we consider the action of $\mathbb{C}^*$ on $\X(d)$ which scales with weight $2$ the linear maps corresponding to edges in $E^\circ$ and fixes the other linear maps, then the regular function \eqref{def:regularfunction} has weight $2$.

\section{The categorical wall-crossing equivalence}\label{subsection:wc}
In this section, we prove a wall-crossing 
equivalence for quasi-BPS categories for symmetric quivers. We thus generalize (the particular case of moduli stacks of representations of quivers of) the theorem of Halpern-Leistner--Sam \cite[Theorem 3.2]{hls} to the situation when there is no stability condition such that the $\mathbb{C}^{\ast}$-rigidified 
stack of semistable representations is Deligne-Mumford.

\subsection{Preliminaries}
Let $Q=(I, E)$ be a symmetric quiver. 
For $d \in \mathbb{N}^I$, 
recall from Subsection \ref{polycat} that $\mathbf{W}(d) \subset M(d)_{0, \mathbb{R}}$ is the polytope 
given by the Minkowski sum 
\begin{align*}
    \mathbf{W}(d)=\frac{1}{2}\mathrm{sum}_{\beta\in\mathscr{A}}[0, \beta] \subset M(d)_{0, \mathbb{R}},
\end{align*}
where \[\mathscr{A}:=\{\beta^a_i-\beta^b_j \mid a,b\in I, 
    (a \to b) \in E, 1\leq i\leq d^a, 1\leq j\leq d^b\}.\]
 We denote by \[H_1, \ldots, H_m \subset \mathbf{W}(d)\]
the codimension one faces of $\mathbf{W}(d)$, 
and by 
\begin{equation}\label{lambdacha}
\lambda, \ldots, \lambda_m\colon \mathbb{C}^*\to ST(d)
\end{equation}
the cocharacters of $ST(d)$
such that $\lambda_i^{\perp} \subset M(d)_{0, \mathbb{R}}$ is parallel to $H_i$ for all $1\leq i\leq m$. 

Note that there is a natural pairing (we abuse notation and use the same notation as for the pairing in Subsection \ref{subsec:theweightlattice}):
\begin{align}\label{newpairing}
    \langle -, -\rangle \colon 
  \mathbb{R}^I\times  M(d)_{\mathbb{R}}^{W_d}  \to 
    \mathbb{R}
\end{align}
defined by \[\langle e^b, \det V^a\rangle=\delta^{ab},\]
where $e^b$ is the basis element of $\mathbb{R}^I$
corresponding to $b \in I$. 
Note that an element 
$\ell \in M(d)_{0, \mathbb{R}}^{W_d}$, 
is written as 
\begin{align}\label{write:l0}
    \ell=\sum_{a}\ell^{a}\det V^{a}, \ 
    \langle d, \ell\rangle=\sum_{a}d^a \ell^a=0
\end{align}
i.e. $\ell$ is a $\mathbb{R}$-character of $G(d)$
which is trivial on the diagonal torus 
$\mathbb{C}^{\ast} \subset G(d)$. 
Note the following lemma: 
\begin{lemma}\label{lem:prop}
Let $\lambda\in\{\lambda_1,\ldots,\lambda_m\}$ be an antidominant cocharacter
such that $\X(d)^\lambda=\times_{j=1}^k \X(d_j)$.
Let $\ell \in M(d)_{0, \mathbb{R}}^{W_d}$ satisfy that  
$\langle \lambda, \ell \rangle=0$. 
 Then we have $\langle d_j, \ell \rangle=0$
 for all $1 \leq j \leq k$. 
 In particular if 
$M(d)_{0, \mathbb{R}}^{W_d} \subset \lambda^{\perp}$, 
then $d_j$ is proportional to $d$ for all $1\leq j\leq k$. 
\end{lemma}
\begin{proof}
For simplicity, suppose that $\lambda$ corresponds to a 
decomposition $d=d_1+d_2$. 
Since $\lambda^{\perp}$ is spanned 
by a subset of weights in $R(d)$, 
for fixed vertices $a, b \in I$ 
we can write: 
\begin{align}\label{write:l}
\ell=\sum_{i, j, p, q}{c_{ij}^{pq}}(\beta_i^{p}-\beta_j^{q})
\end{align}
for some $c_{ij}^{pq}\in \mathbb{R}$,
where the sum on the right hand side is over 
all 
$p, q \in I$ and 
$1 \leq i \leq d_1^{p}, 1\leq j \leq d_1^{q}$
or $d_1^{p}<i \leq d^{p}, d_1^{q}<j \leq d^{q}$. 
%\textcolor{blue}{(Can we say in the above that the sum on the right hand side is after weights in $\mathscr{A}$? I think we should make it clear that the sum of the RHS is after all $a,b\in I$, not only after the fixed $a,b$ on the LHS.)}
On the other hand, as $l$ is Wely-invariant we can write 
\begin{align}\label{write:l2}
    \ell=\sum_{a \in I}\ell^a(\beta_1^a+\cdots+\beta_{d^a}^a)
\end{align}
for some $\ell^a \in \mathbb{R}$, see (\ref{write:l0}).

In the right hand side of (\ref{write:l}), 
the sum of the coefficients of 
$\beta_i^{p}$ for all $p \in I$ and $1\leq i \leq d_1^{p}$ is zero. Therefore, by taking the sum of coefficients of such $\beta_i^p$ in the right hand side of (\ref{write:l2}), 
we conclude that
\begin{align*}
    \sum_{a \in I}\ell^a d_1^a=\langle d_1, \ell \rangle=0. 
\end{align*}
If $M(d)_{0, \mathbb{R}}^{W_d} \subset \lambda^{\perp}$, 
we have $\langle d_j, \ell \rangle=0$ for all $\ell \in M(d)_{0, \mathbb{R}}^{W_d}$, 
hence $d_j$ is proportional to $d$. 
\end{proof}

%\begin{defn}\label{def:generic}
%An element $\ell \in M(d)_{0, \mathbb{R}}^{W_d}$ is \textit{a
%generic weight}
%if the following conditions hold:
%\begin{itemize}
%    \item if $\lambda\in \{\lambda_1,\ldots,\lambda_m\}$ such that $\ell\in \lambda^\perp$, then %$M(d)_{0, \mathbb{R}}^{W_d} \subset
%\lambda^{\perp}$, and 
%\item if $d'\in\mathbb{N}^I$ is a summand of a partition of $d$ such that $d'$ is not %proportional to $d$, then $\langle d', \ell\rangle\neq 0$.
%\end{itemize}
% $\ell \notin \lambda_i^{\perp}$
% for any $\lambda_i$ 
% except those with $M(d)_{0, \mathbb{R}}^{W_d} \subset
% \lambda_i^{\perp}$, and 
% $\langle d_i, \ell \rangle \neq 0$
% for any $d=d_1+d_2$ such that $d_i$ is not proportional 
% to $d$. 
%    \end{defn}
%It is obvious that the set of 
%generic weights is a dense open subset in 
%$M(d)_{0, \mathbb{R}}^{W_d}$.
For $\ell\in M(d)^{W_d}_{0, \mathbb{R}}$, consider the 
open substack 
$\X(d)^{\ell\text{-ss}} \subset \X(d)$
of $\ell$-semistable points. 
By~\cite{Kin}, 
the $\ell$-semistable locus 
consists of $Q$-representations $R$ such that, 
for any subrepresentation $R' \subset R$
of dimension vector $d'$,
we have that $\langle d', \ell\rangle \geq 0$. 
Consider the good moduli space morphism: 
\begin{align*}
    \X(d)^{\ell\text{-ss}} \to X(d)^{\ell\text{-ss}}. 
\end{align*}
A closed point of 
$X(d)^{\ell\text{-ss}}$
corresponds to a direct sum 
\begin{align*}
    \bigoplus_{i=1}^k V^{(i)} \otimes R^{(i)},
    \end{align*}
where $V^{(i)}$ is a finite dimensional $\mathbb{C}$-vector space 
and $R^{(i)}$ is an $\ell$-stable $Q$-representation 
whose dimension vector $d^{(i)}$ satisfying $\langle d^{(i)}, \ell \rangle=0$ for each $1\leq i \leq k$.

\subsection{Quasi-BPS categories for semistable stacks}
For each $\delta_d \in M(d)_{\mathbb{R}}^{W_d}$, 
recall from Definition~\ref{def:defmdw} the quasi-BPS category:
\begin{align}\label{defsdw2}
    \mathbb{M}(d; \delta_d) \subset D^b(\X(d)). 
\end{align}
By Lemma~\ref{lemma:alt}, 
it is the subcategory of objects $P$ 
such that, for any cocharacter $\lambda \colon \mathbb{C}^{\ast} \to T(d)$,
we have that: 
\begin{align}\label{wt:cond}
    \mathrm{wt}_{\lambda}(P)  \subset 
    \left[-\frac{1}{2}n_{\lambda}, 
    \frac{1}{2}n_{\lambda} \right]+\langle \lambda, \delta \rangle,
\end{align}
where $n_{\lambda}:=\left\langle \lambda, 
\det\left(\mathbb{L}_{\X(d)}^{\lambda>0}|_{0}\right) \right\rangle$, see \eqref{nlambdadef}.

Consider a complex $A\in D^b(B\mathbb{C}^*)$. Write $A=\bigoplus_{w\in\mathbb{Z}}A_w$, where $\mathbb{C}^*$ acts with weight $w$ on $A_w$. Define the set of weights
\begin{equation}\label{wtsetA}
\mathrm{wt}(A):=\{w\mid A_w\neq 0\}\subset \mathbb{Z}.
\end{equation}
Define also the integers
\begin{equation}\label{wtmax}
    \mathrm{wt}^{\mathrm{max}}(A):=\mathrm{max}\left(\mathrm{wt}(A)\right),\, \mathrm{wt}^{\mathrm{min}}(A):=\mathrm{min}\left(\mathrm{wt}(A)\right).
\end{equation}

We 
define a version of quasi-BPS categories 
for semistable loci. 
Let 
\begin{align}\label{BPS:Sl}
    \mathbb{M}^{\ell}(d; \delta_d)
    \subset D^b(\X(d)^{\ell \text{-ss}})
\end{align}
be the subcategory of objects 
$P$ such that, for any map 
$\nu \colon B\mathbb{C}^{\ast} \to \X(d)^{\ell\text{-ss}}$,
we have: 
\begin{align}\label{cond:wtnuP}
    \mathrm{wt}(\nu^{\ast}P)
    \subset \left[-\frac{1}{2}n_{\nu}, \frac{1}{2}n_{\nu}\right]+\mathrm{wt}(\nu^{\ast}\delta_d),
\end{align}
where $n_{\nu}:=\mathrm{wt}\left(\det \left(\left(\nu^{\ast}\mathbb{L}_{\X(d)}\right)^{>0}\right)\right) \in \mathbb{Z}$. 
In Corollary~\ref{cor:lzero}, we show that, for $\ell=0$, the category (\ref{BPS:Sl})
is $\mathbb{M}(d; \delta_d)$. 

Consider the restriction functor: 
\begin{align}\label{funct:res}
    \mathrm{res} \colon 
    D^b(\X(d)) 
    \twoheadrightarrow D^b(\X(d)^{\ell\text{-ss}}).
\end{align}

\begin{lemma}\label{lem:funct:res}
The functor (\ref{funct:res}) restricts 
to the functor 
 \begin{align}\label{funct:res2}
    \mathrm{res} \colon 
    \mathbb{M}(d; \delta_d)
    \to \mathbb{M}^{\ell}(d; \delta_d).
    \end{align}
    \end{lemma}
    \begin{proof}
    A map $\nu \colon \mathbb{C}^{\ast} \to \X(d)$ corresponds (possibly after conjugation) to a point 
    $x \in R(d)$ together with a cocharacter 
    $\lambda \colon \mathbb{C}^{\ast} \to T(d)$ which 
    fixes $x$. Then $n_{\nu}=n_{\lambda}$, 
    so the functor (\ref{funct:res}) restricts 
    to the functor (\ref{funct:res2}).
    \end{proof}

There is an orthogonal decomposition \[D^b(\X(d)^{\ell\text{-ss}})=\bigoplus_{w\in\mathbb{Z}}D^b(\X(d)^{\ell\text{-ss}})_w\] as in \eqref{orthogonaldecomposition}. There are analogous such decompositions for $D^b(\X(d)^\lambda)$.
The following lemma is a generalization of~\cite[Proposition~3.11]{hls}, where 
a similar result was obtained when the $\mathbb{C}^{\ast}$-rigidification of $\X(d)^{\ell\text{-ss}}$ is Deligne-Mumford. 
\begin{lemma}\label{lem:modify}
Let $w\in\mathbb{Z}$.
For $\ell \in M(d)_{0, \mathbb{R}}^{W_d}$ and $\delta_d \in M(d)_{\mathbb{R}}^{W_d}$ with $\langle 1_d, \delta_d\rangle=w$, 
the category $D^b(\X(d)^{\ell\text{-ss}})_w$ 
is generated by $\mathrm{res}\left(\mathbb{M}(d; \delta_d)\right)$
and objects of the form  
$m_{\lambda}(P)$,
where $P \in D^b(\X(d)^{\lambda})_w$ is generated by 
$\Gamma_{G^{\lambda}(d)}(\chi'') \otimes \mathcal{O}_{\X^{\lambda}}(d)$
such that 
\begin{align}\label{chi:ineq}
    \langle \lambda, \chi'' \rangle 
    <-\frac{1}{2}n_{\lambda}+\langle \lambda, \delta_d \rangle,
\end{align}
and 
$\lambda \in \{\lambda_1, \ldots, \lambda_m\}$ is an antidominant cocharacter 
of $T(d)$ such that $\langle \lambda, \ell \rangle=0$. 
The functor $m_{\lambda}$ is the categorical 
Hall product for the cocharacter $\lambda$, see \eqref{def:Hallprod}.
\end{lemma}
\begin{proof}
We explain how to modify the proof of \cite[Proposition~3.11]{hls} to obtain the desired conclusion.
   The vector bundles \begin{equation}\label{vbchi}
   \Gamma_{G(d)}(\chi) \otimes \mathcal{O}_{\X(d)^{\ell\text{-ss}}}
   \end{equation}
   for $\chi$ a dominant weight of $G(d)$ such that $\langle 1_d, \chi\rangle=w$
   generate the category $D^b\left(\X(d)^{\ell\text{-ss}}\right)_w$.
We show that \eqref{vbchi}
is generated by the objects in the statement
using induction on the pair 
$(r_{\chi}, p_{\chi})$ (with respect to the lexicographic order), 
where
\begin{align*}
    r_{\chi}:=\mathrm{min}\left\{r \geq 0 \mid 
    \chi+\rho-\delta_d \in r\mathbf{W}(d) \right\}
\end{align*}
and $p_{\chi}$ is the smallest possible number of 
$a_\beta$ equal to $-r_{\chi}$ among 
all ways of writing 
\[\chi+\rho-\delta_d=\sum_{\beta\in\mathscr{A}}a_\beta \beta,\]
with $a_\beta \in [-r_{\chi}, 0]$. 
If $r_{\chi} \le 1/2$, 
then $\Gamma_{G(d)}(\chi) \otimes 
\mathcal{O}_{\X(d)^{\ell\text{-ss}}}$
is an object in $\mathrm{res}\left(\mathbb{M}(d; \delta_d)\right)$, 
so we may assume that $r_{\chi}>1/2$. 

As in the argument of~\cite[Proposition~3.11]{hls},
there is an antidominant cocharacter $\lambda$ of $ST(d)$
such that:
\begin{itemize}
    \item $\langle \lambda, \chi \rangle \leq \langle \lambda, \mu \rangle$
for any $\mu \in -\rho+r\mathbf{W}(d)+\delta_d$, and
\item $\lambda^{\perp}$ is parallel to a face in 
$\mathbf{W}(d)$. 
\end{itemize}

Suppose first that $\langle \lambda, \ell \rangle \neq 0$. 
Then as in the proof of~\cite[Proposition~3.11]{hls}, 
there is a complex of vector bundles
consisting of 
$\Gamma_{G(d)}(\chi) \otimes \mathcal{O}_{\X(d)}$
(which appears once)
and $\Gamma_{G(d)}(\chi') 
\otimes \mathcal{O}_{\X(d)}$
such that $(r_{\chi'}, p_{\chi'})$ is smaller 
than $(r_{\chi}, p_{\chi})$, and supported 
on the $\ell$-unstable locus. The conclusion then follows by induction. 

Suppose next that 
$\langle \lambda, \ell \rangle=0$ (observe that this case did not occur in \cite[Proposition~3.11]{hls}). 
The object 
\[m_{\lambda}(\Gamma_{G(d)^{\lambda}}(\chi) \otimes \mathcal{O}_{\X^{\lambda}(d)})\]
is quasi-isomorphic to a complex of vector bundles (see \cite[Proposition 2.1]{PTzero}) 
consisting of 
$\Gamma_{G(d)}(\chi) \otimes \mathcal{O}_{\X(d)}$
(which appears once)
and $\Gamma_{G(d)}(\chi') \otimes \mathcal{O}_{\X(d)}$
such that $(r_{\chi'}, p_{\chi'})$ is smaller 
than $(r_{\chi}, p_{\chi})$, 
see the last part 
of the proof of~\cite[Proposition~4.1]{P}. 
Since we have 
\begin{align*}
    r_{\chi}=-\frac{\langle \lambda, \chi+\rho-\delta_d\rangle}{\langle \lambda, R^{\lambda>0}(d) \rangle}>\frac{1}{2}
\end{align*}
and $\lambda$ is antidominant, the inequality 
(\ref{chi:ineq}) holds for $\chi''=\chi$. 
The conclusion then follows using induction on $(r_{\chi}, p_{\chi})$. 
\end{proof}

\subsection{The categorical wall-crossing equivalence for symmetric quivers}
In this subsection, 
we give some sufficient conditions for the 
     functor (\ref{funct:res2}) to be an equivalence. As a corollary, we obtain the categorical 
wall-crossing equivalence for quasi-BPS categories of symmetric quivers (and potential zero). 
The argument is similar to~\cite[Theorem~3.2]{hls}, 
with a modification due to the existence of faces of 
$\mathbf{W}(d)$ parallel 
to hyperplanes which contain $M(d)_{0, \mathbb{R}}^{W_d}$. 

For a stability condition $\ell \in M(d)_{0, \mathbb{R}}^{W_d}$, 
let 
\begin{align}\label{KN:X}
    \X(d)=\mathscr{S}_1 \sqcup \cdots 
    \sqcup \mathscr{S}_N \sqcup 
    \X(d)^{\ell\text{-ss}}
\end{align}
be the Kempf-Ness stratification of $\X(d)$
with center $\mathscr{Z}_i \subset \mathscr{S}_i$ for $1\leq i\leq N$. Consider the associated one parameter subgroups 
\begin{equation}\label{mucha}
    \mu_1,\ldots,\mu_N\colon \mathbb{C}^{\ast} \to T(d),
\end{equation} see~\cite[Section~2]{halp} for a review of Kempf-Ness stratification.
\begin{prop}\label{lem:incW}
Suppose that $2\langle \mu_i, \delta_d \rangle \notin 
\mathbb{Z}$ for all $1\leq i\leq N$. 
Then  
the functor (\ref{funct:res2}) 
is fully-faithful: 
\[\mathrm{res} \colon \mathbb{M}(d; \delta_d) \hookrightarrow 
\mathbb{M}^{\ell}(d; \delta_d).\] 
\end{prop}
\begin{proof}
  For a choice $k_\bullet=(k_i)_{i=1}^N \in \mathbb{R}^N$, 
there is a ``window category"
$\mathbb{W}_{k_{\bullet}}^{\ell} \subset 
D^b(\X(d))$
such that the 
functor (\ref{funct:res}) restricts 
to the equivalence (see~\cite{halp, MR3895631}):
\begin{align*}
\mathrm{res} \colon 
    \mathbb{W}_{k_{\bullet}}^{\ell} \stackrel{\sim}{\to} 
    D^b(\X(d)^{\ell\text{-ss}}). 
\end{align*}
The subcategory $\mathbb{W}_{k_{\bullet}}^{\ell}$
consists of objects $P$ such that $P|_{\mathscr{Z}_i}$
has $\mu_i$-weights contained in the interval
$[k_i, k_i+n_{\mu_i})$, where the width $n_{\mu_i}$ is defined 
by (\ref{nlambdadef}) for $\lambda=\mu_i$. 
For $1\leq i\leq N$, let \[k_i:=-n_{\mu_i}/2+\langle \mu_i, \delta_d \rangle.\]
By the assumption of $\delta_d$, 
we have that
$\mathbb{M}(d; \delta_d) \subset \mathbb{W}_{k_{\bullet}}^{\ell}$. 
It thus follows that the functor (\ref{funct:res2}) 
is fully-faithful.
\end{proof}

\begin{prop}\label{prop:esssurj}
%Suppose that $\ell \in M(d)_{0, \mathbb{R}}^{W_d}$ satisfies the following condition: 
%for any $\lambda \in \{\lambda_1, \ldots, \lambda_m\}$ 
%with $\langle \lambda, \ell \rangle=0$ and associated partition $(d_i)_{i=1}^k$ of $d$, we have %that $\langle d_i, \ell \rangle=0$ for any $1\leq i\leq m$.
% and $d=d_1+\cdots+d_k$ is a decomposition corresponding to 
% $\lambda$, 
% then $\langle \ell, d_i \rangle=0$ for all $1\leq i\leq k$. 
For any $\delta_d\in M(d)^{W_d}_{\mathbb{R}}$, the functor (\ref{funct:res2}) is essentially 
surjective: 
\[\mathrm{res} \colon \mathbb{M}(d; \delta_d) \twoheadrightarrow \mathbb{M}^{\ell}(d; \delta_d).\] 
\end{prop}

\begin{proof}
We will use Lemma \ref{lem:modify}.
 We first explain that it suffices to show that 
\begin{align}\label{Hom:vanish}
    \Hom\left(\mathbb{M}^{\ell}(d; \delta_d), 
   \mathrm{res}\left(m_{\lambda}(P)\right)\right)=0,
\end{align}
where $m_\lambda$ is the categorical Hall product and $\lambda$ and $P$ are given as in Lemma~\ref{lem:modify}. 
If the above vanishing holds, then 
\begin{align*}
 \Hom\left(\mathrm{res}\left(\mathbb{M}(d; \delta_d)\right), \mathrm{res}\left(m_{\lambda}(P)\right)\right)=0,
 \end{align*}
hence by Lemma~\ref{lem:modify} 
there is a semiorthogonal decomposition 
\begin{align*}
D^b(\X(d)^{\ell\text{-ss}})
=\left\langle \mathcal{C}, \mathrm{res}\left(\mathbb{M}(d; \delta_d)\right) \right\rangle,
\end{align*}
where $\mathcal{C}$ is the subcategory of $D^b(\X(d)^{\ell\text{-ss}})$ generated by objects $\mathrm{res}\left(m_{\lambda}(P)\right)$
as above (or as in Lemma \ref{lem:modify}). 
Observe that $\mathbb{M}^{\ell}(d; \delta_d)$ is in the left complement of $\mathcal{C}$ in $D^b(\X(d)^{\ell\text{-ss}})$.
Thus the vanishing (\ref{Hom:vanish}) implies that indeed \eqref{funct:res2} is essentially surjective.

Below we show the vanishing (\ref{Hom:vanish}).
The stack $\X^{\lambda \geq 0}(d)$
is the stack of filtrations 
\begin{align}\label{filt:R}
    0=R_0 \subset R_1 \subset R_2 \subset \cdots \subset R_k=R
    \end{align}
of $Q$-representations such that $R_i/R_{i-1}$
has dimension vector $d_i$. 
By Lemma~\ref{lem:prop}, we have that
$\langle \ell, d_i\rangle=0$ for $1\leq i\leq k$. It follows that 
$R$ is $\ell$-semistable if and only if 
every $R_i/R_{i-1}$ is $\ell$-semistable. 
Therefore there are Cartesian squares: 
\begin{equation}\label{C:diagram}
    \begin{tikzcd}
        \times_{i=1}^k \X(d_i)^{\ell\text{-ss}}
\arrow[d, hook, "j"]  & \mathcal{N} \arrow[l, "q^{\ell}_{\lambda}"'] \arrow[r, "p^{\ell}_{\lambda}"] \arrow[d, hook, "j"]\arrow[dr, phantom, "\square"] \arrow[dl, phantom, "\square"]
& 
\X(d)^{\ell\text{-ss}} \arrow[d, hook, "j"] \\
\X(d)^{\lambda} & \arrow[l, "q_{\lambda}"'] \X(d)^{\lambda \geq 0} \arrow[r, "p_{\lambda}"] & \X(d), 
    \end{tikzcd}
\end{equation}
% \begin{align}\label{C:diagram}
%     \xymatrix{
% \times_{i=1}^k \X(d_i)^{\ell\text{-ss}}
% \dinclusion_-{j}  & \mathcal{N} \ar[l]_-{q^{\ell}_{\lambda}} 
% \ar[r]^-{p^{\ell}_{\lambda}} \dinclusion_-{j} & 
% \X(d)^{\ell\text{-ss}} \dinclusion_-{j} \\
% \X(d)^{\lambda} & \ar[l]_-{q_{\lambda}} \X(d)^{\lambda \geq 0} \ar[r]^-{p_{\lambda}} & \X(d), 
%     }
% \end{align}
where each vertical arrow is an open immersion. 
The stack $\mathcal{N}$ is the moduli stack 
of filtrations (\ref{filt:R}) such that each 
$R_i/R_{i-1}$ is $\ell$-semistable with 
dimension vector $d_i$. 
Using proper base change and adjunction, the vanishing (\ref{Hom:vanish}) is 
equivalent to the vanishing of 
\begin{align}\label{vanish:l}
    \Hom(E, p_{\lambda\ast}^{\ell}q_{\lambda}^{\ell\ast}
    j^{\ast}P)
    =\Hom(p_{\lambda}^{\ell\ast}E, q_{\lambda}^{\ell\ast}j^{\ast}P)
\end{align}
for any $E \in \mathbb{M}^{\ell}(d; \delta_d)$. 
Let 
$\pi$ be the following composition: 
\begin{align*}\pi \colon \mathcal{N} \to \X(d)^{\ell\text{-ss}} \to 
X(d)^{\ell\text{-ss}}.
\end{align*}
By the local-to-global Hom spectral sequence, it is enough to show
the vanishing of
\begin{align}\label{vanish:p}
\pi_{\ast}\mathcal{H}om(p_{\lambda}^{\ell\ast}E, q_{\lambda}^{\ell\ast}j^{\ast}P)=0. 
\end{align}
Further, it suffices to show the vanishing (\ref{vanish:p})
formally locally 
at each point $p \in X(d)^{\ell\text{-ss}}$. 
We abuse notation and also denote by $p$ the unique closed point 
in the fiber of $\X(d)^{\ell\text{-ss}} \to X(d)^{\ell\text{-ss}}$
and we may assume that $\lambda$ is a
cocharacter of $G_p:=\mathrm{Aut}(p)$. 
By Lemma~\ref{lem:form:A} below, 
the top diagram of (\ref{C:diagram})
is, formally locally over $p\in X(d)^{\ell\text{-ss}}$, 
of the form 
\begin{align}\label{form:A}
 A^{\lambda}/G_p^{\lambda} \leftarrow A^{\lambda \geq 0}/G_p^{\lambda \geq 0} \to A/G_p,
 \end{align}
 where $A$ is a smooth affine scheme (of finite type over a complete local ring) with a $G_p$-action and good moduli space map $\pi_A\colon A/G_p\to A\ssslash G_p$. %and $A_0$ is the formal completion of $A$ along $\pi_A^{-1}(0)$ for a point $0\in A\sslash G_p$.  
Then the vanishing (\ref{vanish:p}) at $p$ holds 
since the $\lambda$-weights of $p_{\lambda}^{\ell\ast}E$
are strictly larger than those of $q_{\lambda}^{\ell\ast}j^{\ast}P$
by the definition of $\mathbb{M}^{\ell}(d; \delta_d)$
and the inequality (\ref{chi:ineq}), see~\cite[Corollary~3.17, Amplification~3.18]{halp}, \cite[Proposition 4.2]{P}. 
\end{proof}
We have used the following lemma: 
\begin{lemma}\label{lem:form:A}
Let $X(d)_p^{\ell\text{-ss}}=\Spec \widehat{\mathcal{O}}_{X(d)^{\ell \text{-ss}}, p}$. 
The top diagram of (\ref{C:diagram})
pulled back via $X(d)_p^{\ell\text{-ss}} \to X(d)^{\ell \text{-ss}}$
is of the form (\ref{form:A}). 
\end{lemma}
\begin{proof}
Consider the stack $\Theta=\mathbb{A}^1/\mathbb{C}^{\ast}$ .
Since $\mathcal{N}$ is the moduli
stack of filtrations of $\ell$-semistable 
objects (\ref{filt:R}), the top diagram of (\ref{C:diagram}) is 
a component of the diagram 
\begin{align}\label{theta:map}
    \mathrm{Map}(B\mathbb{C}^{\ast}, \X(d)^{\ell\text{-ss}}) \leftarrow 
    \mathrm{Map}(\Theta, \X(d)^{\ell\text{-ss}}) \to
    \X(d)^{\ell\text{-ss}},
\end{align}
where the 
horizontal arrows are evaluation maps 
for mapping stacks from $B\mathbb{C}^{\ast}$
or $\Theta$, see~\cite{halpinstab}. 
Namely, $\mathcal{N}$ is an open and closed
substack of $\mathrm{Map}(\Theta, \X(d)^{\ell\text{-ss}})$, 
and the top diagram in (\ref{C:diagram}) is the restriction 
of (\ref{theta:map}) to $\mathcal{N}$. 
By Luna \'etale slice theorem, 
the pull-back of $\X(d)^{\ell\text{-ss}} \to X(d)^{\ell\text{-ss}}$
via $X(d)_p^{\ell \text{-ss}} \to X(d)^{\ell\text{-ss}}$
is of the form 
$A/G_p$, where $A$ is a smooth affine scheme 
of finite type over $X(d)_p^{\ell\text{-ss}}$
with an action of $G_p$ and good moduli space map $\pi_A\colon A/G_p\to A\ssslash G_p$
%and $A_0$ is the formal completion of $A$ along $\pi_A^{-1}(0)$ for a point $0\in A\sslash G_p$. 
Since the mapping stacks from $B\mathbb{C}^{\ast}$
or 
$\Theta$ commute with 
pull-backs of maps to good moduli spaces, see~\cite[Corollary~1.30.1]{halpinstab}, 
the pull-back of the top diagram in (\ref{C:diagram}) via 
$X(d)_p^{\ell \text{-ss}} \to X(d)^{\ell\text{-ss}}$
consists of compatible connected components of the stacks: 
\begin{align}\label{Map2}
\mathrm{Map}(B\mathbb{C}^{\ast}, A/G_p) \leftarrow 
\mathrm{Map}(\Theta, A/G_p) \to A/G_p. 
\end{align}
Such connected components are of the form (\ref{form:A}) for some cocharacter $\lambda$ of $G_p$, see~\cite[Theorem~1.37]{halpinstab}, and thus the conclusion follows.
% is a component of (\ref{Map2}), 
% which corresponds to the formal fibers of the top 
% diagram in (\ref{C:diagram}). 
\end{proof}

Recall the cocharacters $\{\lambda_1,\ldots,\lambda_m\}$ from \eqref{lambdacha} and the cocharacters $\{\mu_1,\ldots,\mu_N\}$ from \eqref{mucha}.
\begin{defn}\label{def:ul}
    For $\ell\in M(d)_{0,\mathbb{R}}^{W_d}$, let $U_\ell\subset M(d)^{W_d}_\mathbb{R}$ be the dense open subset of weights $\delta_d$ such that $2\langle \mu_i, \delta_d \rangle \notin \mathbb{Z}$. 
\end{defn}

 Propositions~\ref{lem:incW} and \ref{prop:esssurj} imply the following: 
\begin{thm}\label{thm:catwall}
 Suppose that the pair $(\ell, \delta_d)\in M(d)_{0,\mathbb{R}}^{W_d}\times M(d)^{W_d}_\mathbb{R}$ satisfies $\delta_d \in U_{\ell}$. 
 Then the functor (\ref{funct:res2}) is an equivalence: 
        \[\mathrm{res} \colon \mathbb{M}(d; \delta_d)
        \stackrel{\sim}{\to} \mathbb{M}^{\ell}(d; \delta_d).\]
   \end{thm}

We mention two corollaries of Theorem \ref{thm:catwall}: 
\begin{cor}\label{cor:lgen}
  %  Let $\ell \in M(d)_{0, \mathbb{R}}^{W_d}$ be a generic weight. 
  %  Then 
  %  $U_{\ell} \subset M(d)_{\mathbb{R}}^{W_d}$ is a dense open subset and for $\delta_d \in %U_{\ell}$ there is an equivalence: 
  %\[\mathrm{res} \colon \mathbb{M}(d; \delta_d)
  %      \stackrel{\sim}{\to} \mathbb{M}^{\ell}(d; \delta_d).\]  
   For weights $\ell, \ell' \in M(d)_{0, \mathbb{R}}^{W_d}$
        and $\delta_d \in U_{\ell} \cap U_{\ell'}$, there is an equivalence: 
        \begin{align*}
            \mathbb{M}^{\ell}(d; \delta_d) \simeq \mathbb{M}^{\ell'}(d; \delta_d). 
        \end{align*}
\end{cor}
%\begin{proof}
%    If $\ell$ is generic, by Lemma~\ref{lem:prop} 
%    the condition (2) in Theorem~\ref{thm:catwall} is satisfied 
%    as $\langle d, \ell\rangle=0$ and each term $d_i$ is proportional to $d$. 
%    The subset of $\delta \in M(d)_{\mathbb{R}}^{W_d}$ satisfying (1) 
%    in Theorem~\ref{thm:catwall} is a dense open subset. 
%    The conclusion then follows. 
%\end{proof}

\begin{cor}\label{cor:lzero}
For any $\delta_d \in M(d)_{\mathbb{R}}^{W_d}$, we have
$\mathbb{M}(d; \delta_d)=\mathbb{M}^{\ell=0}(d; \delta_d)$.    
\end{cor}
\begin{proof}
    For $\ell=0$, there are no Kempf-Ness loci, so the condition in Theorem~\ref{thm:catwall} is automatic. 
    %Further, we have that $\langle d', \ell\rangle=0$ for any $d'\in\mathbb{N}^I$, so condition %(2) also holds. Therefore we obtain the corollary. 
\end{proof}

\begin{remark}\label{rmk:Ul}
The condition in Theorem~\ref{thm:catwall} is satisfied 
for $\delta_d=\varepsilon \cdot \ell$ for $0<\varepsilon \ll 1$
since $\langle \ell, \mu_i \rangle \in \mathbb{Z}\setminus \{0\}$. 
Similarly, in Corollary~\ref{cor:lgen}, 
the weight $\delta_d=\varepsilon \cdot \ell +\varepsilon' \cdot \ell'$
satisfies $\delta_d \in U_{\ell} \cap U_{\ell'}$
if $0<\varepsilon, \varepsilon \ll 1$ and $(\varepsilon, \varepsilon')$ are linearly 
independent over $\mathbb{Q}$. 
\end{remark}

\subsection{The categorical wall-crossing equivalence for symmetric quivers with potential}\label{subsec:catpot}
Let $(Q, W)$ be a symmetric quiver with potential. 
Similarly to (\ref{BPS:Sl}), we define a category
\begin{align}\label{defuqias}
\mathbb{S}^{\ell}(d; \delta_d) \subset \mathrm{MF}(\X(d)^{\ell\text{-ss}}, \Tr W).
\end{align}
First, for an object $A\in \mathrm{MF}(B\mathbb{C}^*, 0)$, write $A=\bigoplus_{w\in\mathbb{Z}} A_w$ with $A_w\in \mathrm{MF}(B\mathbb{C}^*, 0)_w$. Consider the set of weights: 
\[\mathrm{wt}(A):=\{w\mid A_w\neq 0\}\subset \mathbb{Z}.\]
Then \eqref{defuqias} is the subcategory of $\mathrm{MF}(\X(d)^{\ell\text{-ss}}, \Tr W)$ containing objects $P$ such that, for any 
map $\nu \colon B\mathbb{C}^{\ast} \to \X(d)^{\ell\text{-ss}}$
with $\nu^{\ast} \Tr W=0$, 
the set $\mathrm{wt}\left(\nu^{\ast}P\right)$ satisfies the condition (\ref{cond:wtnuP}). 
Note that 
if $\nu^{\ast} \Tr W \neq 0$, then 
$\mathrm{MF}(B\mathbb{C}^{\ast}, \nu^{\ast} \Tr W)=0$
so that the condition of weights in $\nu^{\ast}P$ is 
vacuous in this case. 

In the graded case (of a the tripled quiver), 
the subcategory 
\begin{align*}
   \mathbb{S}^{\mathrm{gr}, \ell}(d; \delta_d) \subset \mathrm{MF}^{\mathrm{gr}}(\X(d)^{\ell\text{-ss}}, \Tr W) 
\end{align*}
is defined to be the pull-back 
of $\mathbb{S}^{\ell}(d; \delta_d)$ by the forget-the-grading functor 
\begin{align}\label{def:forg}
    \mathrm{forg} \colon \mathrm{MF}^{\rm{gr}}(\X(d)^{\ell \text{-ss}}, \Tr W)
    \to \mathrm{MF}(\X(d)^{\ell \text{-ss}}, \Tr W).
\end{align}
Let $\bullet \in \{\emptyset, \mathrm{gr}\}$. 
The following results are proved as
Theorem~\ref{thm:catwall}, Corollary~\ref{cor:lgen}, and Corollary~\ref{cor:lzero}. 

\begin{thm}\label{thm:catwall2}
 Suppose that the pair $(\ell, \delta_d)\in M(d)_{0,\mathbb{R}}^{W_d}\times M(d)^{W_d}_\mathbb{R}$ satisfies $\delta_d \in U_{\ell}$. 
           Then the restriction functor induces an equivalence: 
        \[\mathrm{res} \colon \mathbb{S}^{\bullet}(d; \delta_d)
        \stackrel{\sim}{\to} \mathbb{S}^{\bullet, \ell}(d; \delta_d).\]
   \end{thm}
   
\begin{cor}\label{cor:lgen2}
    %Let $\ell \in M(d)_{0, \mathbb{R}}^{W_d}$ be a generic weight. 
    %Then, for $\delta \in U_{\ell}$, there is an equivalence 
  %$\mathrm{res} \colon \mathbb{S}^{\bullet}(d; \delta_d)
   %     \stackrel{\sim}{\to} \mathbb{S}^{\bullet, \ell}(d; \delta_d)$.  
    For weights $\ell, \ell' \in M(d)_{0, \mathbb{R}}^{W_d}$, 
        and $\delta_d \in U_{\ell} \cap U_{\ell'}$, there is an equivalence 
        \begin{align*}
            \mathbb{S}^{\bullet, \ell}(d; \delta_d) \simeq \mathbb{S}^{\bullet, \ell'}(d; \delta_d). 
        \end{align*}
\end{cor}

\begin{cor}\label{cor:lzero2}
For any $\delta_d \in M(d)_{\mathbb{R}}^{W_d}$, we have 
$\mathbb{S}^{\bullet}(d; \delta_d)=\mathbb{S}^{\bullet, \ell=0}(d; \delta_d)$.    
\end{cor}

\subsection{The categorical wall-crossing for preprojective algebras}\label{subsec:catdoub}

In this subsection we will use the notations from Subsections~\ref{subsec22} and \ref{subsub:triple}. 
Consider a quiver $Q^{\circ}=(I, E^{\circ})$. Let $Q^{\circ, d}=(I, E^{\circ, d})$ be its 
doubled quiver.
For $\ell \in M(d)_{0, \mathbb{R}}^{W_d}$, the moment map $\mu\colon T^*R^\circ(d)\to \mathfrak{g}(d)$ induces a map $\mu^{\ell \text{-ss}}\colon \left(T^*R^\circ(d)\right)^{\ell \text{-ss}}\to \mathfrak{g}(d)$.
Let \[\mathscr{P}(d)^{\ell \text{-ss}}:=\big(\mu^{\ell \text{-ss}}\big)^{-1}(0)\big/G(d) \subset \mathscr{P}(d)\]
be the derived open substack of $\ell$-semistable representations of the preprojective algebra of $Q^\circ$.
Consider the restriction functor 
\begin{align}\label{rest:2}
    \mathrm{res} \colon D^b(\mathscr{P}(d)) \twoheadrightarrow D^b(\mathscr{P}(d)^{\ell\text{-ss}}). 
\end{align}
The closed immersion (\ref{def:Pj}) restricts to the 
closed immersion 
\begin{align}\label{def:Pj2}
j \colon \mathscr{P}(d)^{\ell\text{-ss}} \hookrightarrow 
\mathscr{Y}(d)^{\ell\text{-ss}}. 
\end{align}
For $\delta_d \in M(d)_{\mathbb{R}}^{W_d}$, 
define the subcategory
\begin{align}\label{def:int}
    \mathbb{T}^{\ell}(d; \delta_d) \subset D^b(\mathscr{P}(d)^{\ell\text{-ss}})
\end{align}
with objects $\mathcal{E}$ such that, 
for any map $\nu \colon B\mathbb{C}^{\ast} \to \mathscr{P}(d)^{\ell \text{-ss}}$,
we have 
\begin{align}\label{wt:cond:nu}
    \mathrm{wt}(\nu^{\ast}j^{\ast}j_{\ast}\mathcal{E})
    \subset \left[-\frac{1}{2}n_{\nu}, \frac{1}{2}n_{\nu} \right] +\mathrm{wt}(\nu^{\ast}\delta_d). 
\end{align}
Here, we let $n_{\nu}:=\mathrm{wt}(\det (\nu^{\ast}\mathbb{L}_{\X(d)}|_{\mathscr{P}(d)})^{\nu>0})$, 
where $\X(d)$ is the moduli stack of representations (\ref{stack:triple}) of the tripled quiver $Q$ of $Q^{\circ}$. 

\begin{remark}
The subcategory (\ref{def:int}) is the 
intrinsic window subcategory for the quasi-smooth stack
$\mathscr{P}(d)^{\ell \text{-ss}}$
defined in~\cite[Definition~5.2.13]{T}. 
\end{remark}

\begin{prop}\label{prop:koszul:P}
The Koszul equivalence (\ref{Kosz}) descends to an equivalence: 
\begin{align}\label{equiv:lss1}
    \Theta \colon D^b(\mathscr{P}(d)^{\ell \text{-ss}}) \stackrel{\sim}{\to}
    \mathrm{MF}^{\rm{gr}}(\X(d)^{\ell \text{-ss}}, \Tr W),
\end{align}
which restricts to an equivalence: 
\begin{align}\label{equiv:lss2}
    \Theta \colon \mathbb{T}^{\ell}(d; \delta_d) \stackrel{\sim}{\to} 
    \mathbb{S}^{\mathrm{gr}, \ell}(d; \delta_d). 
\end{align}
    \end{prop}
\begin{proof}
    Let $\eta \colon \X(d) \to \mathscr{Y}(d)$ be the projection. 
    Then we have 
    \begin{align}\label{inclu:ss}
    \mathrm{Crit}(\Tr W) \cap \X(d)^{\ell \text{-ss}} \subset 
    \eta^{-1}(\mathscr{Y}(d)^{\ell \text{-ss}})
    \subset \X(d)^{\ell \text{-ss}},     
    \end{align}
    where the first inclusion is proved in~\cite[Lemma~4.3.22]{halpK32} and the 
    second inclusion is immediate from the definition of $\ell$-stability. 
    We obtain equivalences 
    \begin{align*}
        D^b(\mathscr{P}(d)^{\ell \text{-ss}}) \stackrel{\sim}{\to} 
        \mathrm{MF}^{\rm{gr}}(\eta^{-1}(\mathscr{Y}(d)^{\ell \text{-ss}}), \Tr W) 
        \stackrel{\sim}{\leftarrow} 
        \mathrm{MF}^{\rm{gr}}(\X(d)^{\ell \text{-ss}}, \Tr W), 
    \end{align*}
    where the first equivalence is the Koszul equivalence in Theorem~\ref{thm:Koszul} and 
    the second equivalence follows from (\ref{inclu:ss}) together with the fact that 
    matrix factorizations are supported on critical locus, see (\ref{rest:MFU}). 
    Therefore we obtain the equivalence (\ref{equiv:lss1}). 

    For an object $\mathcal{E} \in D^b(\mathscr{P}(d)^{\ell \text{-ss}})$, 
    the object $P=\Theta(\mathcal{E})$
    is in 
$\mathbb{S}^{\rm{gr}, \ell}(d; \delta_d)$ if and only 
if, for any map $\nu \colon B\mathbb{C}^{\ast} \to \X(d)^{\ell \text{-ss}}$
with $\nu^{\ast}\Tr W=0$, the set $\mathrm{wt}\left(\nu^{\ast}\mathrm{forg}(P)\right)$ satisfies the weight condition (\ref{cond:wtnuP}).
As $P$ is supported on $\mathrm{Crit}(\Tr W)$, 
we may assume that 
the image of $\nu$ is contained in $\mathrm{Crit}(\Tr W) \cap \X(d)^{\ell \text{-ss}}$. 
By the $\mathbb{C}^{\ast}$-equivariance of $P$ 
for the fiberwise weight $2$-action on $\eta\colon \X(d)\to\Y(d)$ and upper semicontinuity,  
we have that 
\[\mathrm{wt}\left(\nu^{\ast}\mathrm{forg}(P)\right)\subset \mathrm{wt}\left(\nu'^{\ast}\mathrm{forg}(P)\right),\]
 where $\nu'$ is the composition 
\begin{align*}
\nu' \colon 
B\mathbb{C}^{\ast} \stackrel{\nu}{\to} \X(d) \stackrel{\eta}{\to} \mathscr{Y}(d)
\stackrel{0}{\hookrightarrow} \X(d).
\end{align*}
The image of $\nu'$ lies in $\mathscr{P}(d)^{\ell \text{-ss}} \hookrightarrow \Y(d)^{\ell \text{-ss}}\stackrel{0}{\hookrightarrow} \X(d)^{\ell \text{-ss}}$. 
Therefore we may assume that the image of $\nu$ is contained in 
$\mathscr{P}(d)^{\ell \text{-ss}}$. 
Since   
    the object $P$ is represented by 
    \begin{align*}
    P=\Theta(\mathcal{E})=
    (\mathcal{E} \otimes_{\mathcal{O}_{\mathscr{P}(d)}}\mathcal{O}_{\X(d)})|_{\X(d)^{\ell \text{-ss}}}
    =(j_{\ast}\mathcal{E} \otimes_{\mathcal{O}_{\mathscr{Y}(d)}} \mathcal{O}_{\X(d)})|_{\X(d)^{\ell \text{-ss}}},
    \end{align*}
    it follows that 
    $\mathrm{wt}\left(\nu^{\ast}\mathrm{forg}(P)\right)$ satisfies the condition (\ref{cond:wtnuP}) if and only $\mathrm{wt}\left(\nu^{\ast}(j_{\ast}\mathcal{E})\right)$
    satisfies the condition (\ref{wt:cond:nu}).
    Therefore $\Theta(\mathcal{E})$ is in $\mathbb{S}^{\rm{gr}, \ell}(d; \delta_d)$
    if and only if $\mathcal{E}$ is in $\mathbb{T}^{\ell}(d; \delta_d)$. 
\end{proof}

By combining Proposition~\ref{prop:koszul:P} with Theorem~\ref{thm:catwall2}, Corollary~\ref{cor:lgen2}
and Corollary~\ref{cor:lzero2}, we obtain the following: 

\begin{thm}\label{thm:catwall3}
 Suppose that the pair $(\ell, \delta_d)\in M(d)_{0,\mathbb{R}}^{W_d}\times M(d)^{W_d}_\mathbb{R}$ satisfies $\delta_d \in U_{\ell}$. 
           Then the restriction functor (\ref{rest:2}) induces an equivalence: 
        \[\mathrm{res} \colon \mathbb{T}(d; \delta_d)
        \stackrel{\sim}{\to} \mathbb{T}^{\ell}(d; \delta_d).\]
   \end{thm}
   
\begin{cor}\label{cor:lgen3}
  %  Let $\ell \in M(d)_{0, \mathbb{R}}^{W_d}$ be generic. 
  %  Then, for $\delta_d \in U_{\ell}$, there is an equivalence 
  %$\mathrm{res} \colon \mathbb{T}(d; \delta_d)
  %      \stackrel{\sim}{\to} \mathbb{T}^{\ell}(d; \delta_d)$.  
   For $\ell, \ell' \in M(d)_{0, \mathbb{R}}^{W_d}$ 
        and $\delta_d \in U_{\ell} \cap U_{\ell'}$, there is an equivalence 
        \begin{align*}
            \mathbb{T}^{\ell}(d; \delta_d) \simeq \mathbb{T}^{\ell'}(d; \delta_d). 
        \end{align*}
\end{cor}

\begin{cor}\label{cor:lzero3}
For any $\delta_d \in M(d)_{\mathbb{R}}^{W_d}$, we have 
$\mathbb{T}(d; \delta_d)=\mathbb{T}^{\ell=0}(d; \delta_d)$.    
\end{cor}

\subsection{Quasi-BPS categories under Kn\"orrer periodicity}

In this subsection, we apply Corollaries \ref{cor:lzero} and \ref{cor:lzero2} to obtain an equivalence of quasi-BPS categories under Kn\"orrer periodicity, which is a particular case of the Koszul equivalence. Other than the use of Corollaries \ref{cor:lzero} and \ref{cor:lzero2}, the current subsection is independent of the other results and constructions discussed in Section \ref{subsection:wc}. 
We will use the results of this subsection in Section \ref{sec:SOD} and in \cite{PTtop}.

We will use the notations from Subsection \ref{subsec22two}.
For a symmetric quiver $Q$ and 
$d \in \mathbb{N}^I$, 
let $U$ be a $G(d)$-representation. 
Consider the closed immersion 
\begin{align*}
    j \colon \X(d):=R(d)/G(d) \hookrightarrow 
    \mathscr{Y}=(R(d) \oplus U)/G(d)
\end{align*}
into $(R(d)\oplus \{0\})/G(d)$. 
We consider the quotient stack $\X^\gimel$ and regular function $f^\gimel$: 
\begin{align*}
    \X^\gimel=(R(d) \oplus U \oplus U^{\vee})/G(d)
    \stackrel{f}{\to} \mathbb{C}, \ 
    f(x, u, u')=\langle u, u'\rangle. 
\end{align*}
We consider the following Cartesian diagram 
\begin{align*}
    \xymatrix{
    \X^\gimel \times_{\mathscr{Y}}\X(d)
    \inclusion^-{s} \ar[d]_-{v} & \X^\gimel \ar[d]^-{\eta} \\
    \X(d) \inclusion^-{j} & \mathscr{Y}, 
    }
\end{align*}
where $\eta$ is the natural projection. 
Let $\mathbb{C}^{\ast}$ act with weight $0$ on $R(d)\oplus U$ and with weight $2$ on $U^\vee$. 
In this case, the Koszul equivalence in Theorem~\ref{thm:Koszul}
is given by, see~\cite[Remark~2.3.5]{T}:
\begin{align}\label{equiv:thetaX}
    \Theta=s_{\ast}v^{\ast}
    \colon D^b(\X(d)) \stackrel{\sim}{\to} 
    \mathrm{MF}^{\rm{gr}}(\X^\gimel, f). 
\end{align}
Such an equivalence is also called Kn\"orrer periodicity \cite{OrLG, Hirano}.  
For a weight $\delta^\gimel_d \in M(d)_{\mathbb{R}}^{W_d}$, define the subcategory 
\begin{align}\label{qbps:X'}
    \mathbb{S}^{\rm{gr}}(d; \delta^\gimel_d) \subset 
    \mathrm{MF}^{\rm{gr}}(\X^\gimel, f)
\end{align}
in a way similar to (\ref{defsdwgr}). 
By Lemma~\ref{lemma:alt},
it consists of matrix factorizations 
whose factors are direct sums of vector bundles
$\mathcal{O}_{\X^\gimel}\otimes \Gamma$, where $\Gamma$
is a $G(d)$-representations
such that, for any weight $\chi'$ of $\Gamma$ and any cocharacter $\lambda$ of $T(d)$, we have
\begin{align}\label{cond:n'}
\langle \lambda, \chi'-\delta^\gimel_d \rangle 
\in \left[-\frac{1}{2}n^\gimel_{\lambda}, \frac{1}{2}n^\gimel_{\lambda}   \right]. 
\end{align} 
Here, we define $n^\gimel_{\lambda}$ by: 
\begin{align}\label{id:n'}
    n^\gimel_{\lambda}=\langle \lambda, 
    \mathbb{L}_{\X^\gimel}^{\lambda>0} \rangle=
    n_{\lambda}+\langle \lambda, U^{\lambda>0} \rangle+
    \langle \lambda, (U^{\vee})^{\lambda>0} \rangle,
\end{align}
where recall the definition of $n_{\lambda}$ from (\ref{nlambdadef}). The following is the main result we prove in this subsection:
\begin{prop}\label{prop:period}
Let $\delta^\gimel_d=\delta_d-\frac{1}{2}\det U\in M(d)^{W_d}_\mathbb{R}$. 
The equivalence (\ref{equiv:thetaX}) 
restricts to the equivalence: 
\begin{align}\label{Thetaprop320}
    \Theta \colon \mathbb{M}(d; \delta_d)
    \stackrel{\sim}{\to} \mathbb{S}^{\rm{gr}}(d; \delta^\gimel_d).
\end{align}
\end{prop}
\begin{proof}
We first note that, by Lemma~\ref{lem:genJ}, 
an object $\mathcal{E} \in D^b(\X(d))$
satisfies $\Theta(\mathcal{E}) \in \mathbf{S}^{\rm{gr}}(d; \delta^\gimel_d)$ 
if and only if 
$j_{\ast}\mathcal{E}$ is generated by vector bundles 
$\mathcal{O}_{\mathscr{Y}} \otimes \Gamma'$,
where any weight $\chi'$ of $\Gamma'$ satisfies 
(\ref{cond:n'}). 

    By Lemma~\ref{lemma:alt},
    the category $\mathbb{M}(d; \delta_d)$
    is generated by vector bundles $\mathcal{O}_{\X(d)}\otimes\Gamma$
    such that any weight $\chi$ of $\Gamma$
    satisfies (\ref{cond:n}) for any $\lambda$. 
   Consider the Koszul resolution 
    \begin{align}\label{Kos:resol}
        j_{\ast}(\Gamma \otimes \mathcal{O}_{\X(d)}) 
        =\Gamma \otimes \mathrm{Sym}_{\mathscr{Y}}
        (\mathcal{U}^{\vee}[1]),
    \end{align}
    where $\mathcal{U} \to \mathscr{Y}$ is the vector 
    bundle associated with the $G(d)$-representation $U$. 
    Therefore, the category 
$j_{\ast}\mathbb{M}(d; \delta)$
is generated by vector bundles $\mathcal{O}_{\mathscr{Y}} \otimes \Gamma'$
    such that any weight $\chi'$ of $\Gamma'$
    satisfies 
    \begin{align}\label{cond:n''}
        \langle \lambda, \chi'-\delta_d \rangle 
        \in \left[-\frac{n_{\lambda}}{2}+
        \langle \lambda, (U^{\vee})^{\lambda<0}), 
        \frac{n_{\lambda}}{2}+
        \langle \lambda, (U^{\vee})^{\lambda>0})
        \right]
    \end{align}
    for any $\lambda$. 
    By (\ref{id:n'}), we have 
    \begin{align}\label{n:equality}
        \frac{n^\gimel_{\lambda}}{2}
        =\frac{n_{\lambda}}{2}+\langle \lambda, (U^{\vee})^{\lambda>0}\rangle 
        +\frac{1}{2}\langle \lambda, U \rangle 
                =\frac{n_{\lambda}}{2}-
        \langle \lambda, 
        (U^{\vee})^{\lambda<0}\rangle 
        -\frac{1}{2}\langle \lambda, U \rangle. 
    \end{align}
    Therefore (\ref{cond:n''}) implies (\ref{cond:n'})
    for $\delta^\gimel_d=\delta_d-\frac{1}{2} \det U$, 
    hence the functor (\ref{equiv:thetaX})
    sends 
$\mathbb{M}(d; \delta)$ to $\mathbb{S}^{\rm{gr}}(d; \delta^\gimel_d)$, which shows the fully-faithfullness of \eqref{Thetaprop320}.

       To show essential surjectivity of \eqref{Thetaprop320}, let $\mathcal{E} \in D^b(\X(d))$ be
    such that $j_{\ast}\mathcal{E}$ is generated by the vector bundles
$\mathcal{O}_{\mathscr{Y}} \otimes \Gamma'$,
where any weight $\chi'$ of $\Gamma'$ satisfies 
(\ref{cond:n'}). 
We will show that $\mathcal{E}\in \mathbb{M}^{\ell=0}(d; \delta_d)$, and thus that $\mathcal{E}\in \mathbb{M}(d; \delta_d)$ by Corollary \ref{cor:lzero}.

Let $\nu \colon B\mathbb{C}^{\ast} \to \X(d)$
be a map, which corresponds to a point 
$x \in R(d)$ and a cocharacter $\lambda \colon 
\mathbb{C}^{\ast} \to T(d)$ which fixes $x$. 
By the condition (\ref{cond:n'}) for weights of $\Gamma'$, we have 
\begin{align*}
    \mathrm{wt}^{\mathrm{max}}(\nu^{\ast}j^{\ast}j_{\ast}\mathcal{E}) \leq 
    \frac{1}{2} n_{\lambda}'
    +\langle \lambda, \delta_d' \rangle,
\end{align*}
see \eqref{wtmax} for the definition of $\mathrm{wt}^{\mathrm{max}}$.
% Here for an object $A \in D^b(B\mathbb{C}^{\ast})$, 
% by writing $A=\oplus_{m\in \mathbb{Z}} A_m$ where $A_m$ is of 
% $\mathbb{C}^{\ast}$-weight $m$, 
% we denoted by $\mathrm{wt}^{\mathrm{max}}(A)$ the 
% maximum $m \in \mathbb{Z}$ such that $A_m \neq 0$. 
On the other hand, by the Koszul resolution (\ref{Kos:resol}),
we have 
\begin{align*}
   \mathrm{wt}^{\mathrm{max}}(\nu^{\ast}j^{\ast}j_{\ast}\mathcal{E})
   = \mathrm{wt}^{\mathrm{max}}(\nu^{\ast}\mathcal{E})
   -\langle\lambda, (U^{\vee})^{\lambda>0} \rangle. 
\end{align*}
Therefore we have 
\begin{align*}
   \mathrm{wt}^{\mathrm{max}}(\nu^{\ast}\mathcal{E}) 
   \leq \frac{n_{\lambda}'}{2}+\langle \lambda, \delta_d'
   \rangle -\langle \lambda, (U^{\vee})^{\lambda>0}
   \rangle =\frac{n_{\lambda}}{2}+\langle \lambda, \delta_d\rangle,
\end{align*}
where the last equality follows from (\ref{n:equality}).
The lower bound 
\[\mathrm{wt}^{\mathrm{min}}(\nu^{\ast}\mathcal{E})
\geq -\frac{n_{\lambda}}{2}+\langle \lambda, \delta_d\rangle\] is proved similarly. 
We then have that: 
\[\mathrm{wt}\left(\nu^{\ast}\mathcal{E}\right)\subset\left[-\frac{n_{\lambda}}{2}+\langle \lambda, \delta_d\rangle, \frac{n_{\lambda}}{2}+\langle \lambda, \delta_d\rangle\right].\]
Thus $\mathcal{E} \in \mathbb{M}(d; \delta_d)^{\ell=0}$, and then $\mathcal{E} \in \mathbb{M}(d; \delta_d)$
by Corollary~\ref{cor:lzero}. 
\end{proof}
% \begin{remark}
%  We have used Corollary~\ref{cor:lzero} which will 
%  be proved in Section~\ref{subsection:wc}. 
%  Note that in the proof of Corollary~\ref{cor:lzero}
%  we will not use Proposition~\ref{prop:period} nor any 
%  result in Section~\ref{sec:SOD}, so it is not 
%  a circular reasoning. 
% \end{remark}

Let $W$ be a potential of $Q$. By abuse of notation, 
we denote by $\Tr W \colon \X^\gimel \to \mathbb{C}$ 
the pull-back of $\Tr W \colon \X(d) \to \mathbb{C}$ 
by the natural projection $\X^\gimel \to \X(d)$. 
There is an equivalence similar to (\ref{equiv:thetaX}), also called 
Kn\"orrer periodicity, see~\cite[Theorem~4.2]{Hirano}, \cite{OrLG}:
\begin{align}\label{theta:period}
    \Theta=s_{\ast}v^{\ast} \colon \mathrm{MF}(\X(d), \Tr W) \stackrel{\sim}{\to} 
    \mathrm{MF}(\X^\gimel, \Tr W+f). 
\end{align}
The subcategory 
\begin{align*}
    \mathbb{S}(d; \delta_d) \subset \mathrm{MF}(\X^\gimel, \Tr W+f)
\end{align*}
is defined similarly to (\ref{qbps:X'}). 
The following proposition is proved as
Proposition~\ref{prop:period}, using Corollary~\ref{cor:lzero2}
instead of Corollary~\ref{cor:lzero}. 
\begin{prop}\label{prop:period2}
    Let $\delta^\gimel_d=\delta_d-\frac{1}{2}\det U\in M(d)^{W_d}_\mathbb{R}$. 
    The equivalence (\ref{theta:period})
    restricts to the equivalence: 
    \begin{align*}
        \Theta \colon \mathbb{S}(d; \delta_d) \stackrel{\sim}{\to}
        \mathbb{S}(d; \delta^\gimel_d). 
    \end{align*}
\end{prop}

\section{The semiorthogonal decompositions of DT categories}
\label{sec:SOD}
In this section, we construct semiorthogonal decompositions for the moduli of (framed or unframed)
representations of certain symmetric quivers (see Subsection \ref{symmquivers}) in terms of quasi-BPS categories, see Theorem \ref{thmsodC}, \ref{sodfullstackB}, \ref{sodfullstackBW} and Corollary \ref{thmsodD}. The results generalize the decomposition of DT categories of points on $\mathbb{C}^3$ from \cite[Theorem 1.1]{PTzero} and the decomposition of the Hall algebra of $\mathbb{C}^3$ (equivalently, of the Porta--Sala Hall algebra of $\mathbb{C}^2$) from \cite[Theorem 1.1]{P}, \cite[Theorem 1.1]{P2}.

\subsection{Semiorthogonal decompositions}
The following 
is the main result in this section, 
which provides a semi-orthogonal decomposition of $D^b\left(\X^f(d)^{\text{ss}}\right)$ in products of quasi-BPS categories of $Q$. Recall
the definition of a good weight from Definition \ref{def:generic2}, the category $\mathbb{M}(d;\delta)$ from Definition \eqref{def:defmdw}, and the Weyl-invariant real weights $\tau_d, \sigma_d$ from Subsection \ref{subsec:theweightlattice}. Recall also the convention about the product of categories of matrix factorizations from Subsection \ref{subsection:matrixfactorizations}.

%\textcolor{red}{(I think Thm 2.3 and 2.4 are special cases of Thm 2.10 and 2.11. I suggest stating only Thm 2.10, 2.11 (removing Thm 2.3, 2.4) and creating a subsection proving Thm 2.10, 2.11 in the very symmetric case. Also the definition of $\mathbf{W}(d)$, etc can be applied for any symmetric case, so better to remove the very symmetric assumption in this subsection.)}\textcolor{blue}{(I think Thm 2.5and 2.6 (in the current version, previously 2.3 and 2.4) have a slighly stonger statement because of $\delta_d$. But I'll try to write a more general statement for them as well.)}

\begin{thm}\label{thmsodC}
Let $Q$ be a symmetric quiver such that the number of loops at each vertex $i\in I$ has the same parity. Let $d\in \mathbb{N}^I$,
   let $\delta_d\in M(d)^{W_d}_\mathbb{R}$, and let $\mu\in \mathbb{R}$ such that $\delta_d+\mu\sigma_d$ is a good weight. 
   For a partition $(d_i)_{i=1}^k$ of $d$, let $\lambda$ be an associated antidominant cocharacter and define the 
       weights $\delta_{d_i}\in M(d_i)_\mathbb{R}^{W_{d_i}}$, 
       $\theta_i \in \frac{1}{2} M(d_i)^{W_{d_i}}$ for $1\leq i\leq k$ by: 
\begin{equation}\label{def:deltai}
        \sum_{i=1}^k \delta_{d_i}=\delta_d, \
        \sum_{i=1}^k \theta_i=-\frac{1}{2}R(d)^{\lambda>0}+\frac{1}{2}\mathfrak{g}(d)^{\lambda>0}. 
    \end{equation}
   There is a semiorthogonal decomposition
    \begin{equation}\label{sodA}
        D^b\left(\X^f(d)^{\text{ss}}\right)=\left\langle \bigotimes_{i=1}^k \mathbb{M}(d_i; \theta_i+\delta_{d_i}+v_i\tau_{d_i}): 
        \mu\leq \frac{v_1}{\dd_1}<\cdots<\frac{v_k}{\dd_k}<1+\mu
        \right\rangle,
    \end{equation}
%     Here $d-d_1+\cdots+d_k$ is a partition of $d$, 
%     and the 
%        weights $\delta_{d_i}\in M(d_i)_\mathbb{R}^{W_{d_i}}$, 
%        $\theta_i \in \frac{1}{2} M(d_i)^{W_{d_i}}$ are defined by 
% \begin{equation}\label{def:deltai}
%         \sum_{i=1}^k \delta_{d_i}=\delta_d, \
%         \sum_{i=1}^k \theta_i=-\frac{1}{2}R(d)^{\lambda>0}+\frac{1}{2}\mathfrak{g}(d)^{\lambda>0} 
%     \end{equation}
%     where $\lambda$ is an antidominant cocharacter with associated partition $(d_i)_{i=1}^k$. 
    %and define $\theta_{i}\in \frac{1}{2}M(d_i)^{W_{d_i}}$
    %such that 
    %\begin{equation}\label{def:deltai}
    %\sum_{i=1}^k \theta_i=-\frac{1}{2}R(d)^{\lambda>0}+\frac{1}{2}\mathfrak{g}(d)^{\lambda>0}.
    %\end{equation}
    where the right hand side in \eqref{sodA} is after all partitions $(d_i)_{i=1}^k$ of $d$ and
    real numbers 
    $(v_i)_{i=1}^k\in\mathbb{R}^k$ such that the sum of coefficients of $\theta_{i}+\delta_{d_i}+v_i\tau_{d_i}$ is an integer for all $1\leq i\leq k$.
    The order on the summands is as in Subsection \ref{comppartitions}. 

    The functor from a summand on the right hand side to $D^b\left(\X^f(d)^{\text{ss}}\right)$ is the composition of the Hall product with the pullback along the forget-the-framing map $\X^f(d)^{\mathrm{ss}}\to \X(d)$. 
Further, the decomposition \eqref{sodA} is $X(d)$-linear for the map $\pi_{f,d}\colon \X^f(d)^{\mathrm{ss}}\to\X(d)\xrightarrow{\pi_{X,d}} X(d)$. 
\end{thm}

The same argument also applies to obtain the following similar semiorthogonal decomposition using Theorem \ref{sodfullstackA}
for unframed moduli stacks.
Note that this decomposition is different from the one discussed in \cite[Theorem 1.1]{P}, which we could not use to obtain a decomposition of $D^b\left(\X^f(d)^{\mathrm{ss}}\right)$ as in Theorem \ref{thmsodC}.

\begin{thm}\label{sodfullstackB}
Let $Q$ be a symmetric quiver such that the number of loops at each vertex $i\in I$ has the same parity.
Let $d\in\mathbb{N}^I$ and let $\delta_d\in M(d)^{W_d}_\mathbb{R}$. 
   For a partition $(d_i)_{i=1}^k$ of $d$, define the 
       weights $\delta_{d_i}\in M(d_i)_\mathbb{R}^{W_{d_i}}$, 
       $\theta_i \in \frac{1}{2} M(d_i)^{W_{d_i}}$ for $1\leq i\leq k$  as in \eqref{def:deltai}. 
% \begin{equation}\label{def:deltai}
%         \sum_{i=1}^k \delta_{d_i}=\delta_d, \
%         \sum_{i=1}^k \theta_i=-\frac{1}{2}R(d)^{\lambda>0}+\frac{1}{2}\mathfrak{g}(d)^{\lambda>0}. 
%     \end{equation}
   There is a semiorthogonal decomposition
    \begin{equation}\label{sod:sodfullstackB}
D^b\left(\X(d)\right)=\left\langle \bigotimes_{i=1}^k \mathbb{M}(d_i; \theta_i+\delta_{d_i}+v_i\tau_{d_i}) : 
\frac{v_1}{\dd_1}<\cdots<\frac{v_k}{\dd_k}
\right\rangle,
   \end{equation}
   where the right hand side in \eqref{sod:sodfullstackB} is after all partitions $(d_i)_{i=1}^k$ of $d$ and 
    real numbers $(v_i)_{i=1}^k\in\mathbb{R}^k$ such that the sum of coefficients of $\theta_{i}+\delta_{d_i}+v_i\tau_{d_i}$ is an integer for all $1\leq i\leq k$.
    The order on the summands is as in Subsection \ref{comppartitions}. 

    The functor from a summand on the right hand side to $D^b\left(\X^f(d)^{\text{ss}}\right)$ is given by the Hall product. 
The decomposition \eqref{sod:sodfullstackB} is $X(d)$-linear.
\end{thm}

The plan of proof for both of Theorem \ref{thmsodC} and \ref{sodfullstackB}
is as follows: we first prove them in a particular case, that of very symmetric quivers; then we use Knörrer periodicity to obtain the more general statements stated above.

Using \cite[Proposition 2.3]{PTzero}, \cite[Proposition 2.1]{P0}, one obtains versions for quivers with potentials of Theorems \ref{thmsodC} and \ref{sodfullstackB}, see Subsection \ref{subsection:moresod}.
Note that both theorems apply to tripled quivers, and thus to tripled quivers with potential. 

\subsection{Very symmetric quivers}

We introduce a class of quivers for which we can prove Theorems \ref{thmsodC} and \ref{sodfullstackB} using the arguments employed for the quiver with one vertex and three loops in \cite{PTzero, P2}.

\begin{defn}
    A quiver $Q=(I,E)$ is a \textit{very symmetric quiver} if there exists an integer $A \in \mathbb{Z}_{\geq 1}$ 
such that, for any vertices $a, b\in I$, the number of edges from $a$ to $b$ is $A$. 
\end{defn}

\begin{remark}
We will use that the categories of matrix factorizations for a symmetric quiver as in Theorem \ref{thmsodC} (i.e. a symmetric quiver with the number of loops at each vertex with the same parity) are equivalent, under Knörrer periodicity, to the categories of matrix factorizations for a very symmetric quiver. Thus we will study in detail the categories of matrix factorizations for very symmetric quivers, for which one can use the same combinatorial techniques as in \cite{P2, PTzero}, see Remark \ref{remark46}.
\end{remark}

The first step in proving Theorem \ref{thmsodC} is:

\begin{thm}\label{thmsodA}
Let $Q$ be a very symmetric quiver. Then
Theorem \ref{thmsodC} holds for $Q$.
\end{thm}

Until Subsection \ref{symmquivers}, we assume the quiver $Q$ is very symmetric.

The proof of Theorem \ref{thmsodA} follows closely the proof of \cite[Theorem 3.2]{PTzero}. 
As in loc. cit., the claim follows from the following semiorthogonal decomposition for subcategories of $D^b(\X(d))$ using ``window categories". Recall the definition of the categories $\mathbb{D}(d; \delta)$ from \eqref{def:dd}.

\begin{thm}\label{thmsodB}
    Let $\delta_d\in M(d)^{W_d}_\mathbb{R}$ and let $\mu\in \mathbb{R}$ such that $\delta_d+\mu\sigma_d$ is a good weight.
    Then there is a semiorthogonal decomposition
    \begin{equation}\label{SODD}
    \mathbb{D}(d; \delta_d+\mu\sigma_d)
    =\left\langle \bigotimes_{i=1}^k \mathbb{M}(d_i; \theta_i+\delta_{d_i}+v_i\tau_{d_i}) \right\rangle,
    \end{equation}
    where the right hand side is as in Theorem \ref{thmsodC}.
\end{thm}

%It is convenient to introduce the following definition:\begin{defn} A weight $\delta_d\in M(d)^{W_d}_\mathbb{R}$ is called \textit{generic} if for all dominant cocharacters $\lambda$ such that $\langle \lambda, \beta^i_a\rangle\in\{-1, 0\}$ for all $i\in I$ and $1\leq a\leq d^i$, one has that $\langle \lambda, \delta_d\rangle\notin\frac{1}{2}\mathbb{Z}$.\end{defn}

\begin{proof}[Proof of Theorem \ref{thmsodA} assuming Theorem~\ref{thmsodB}]
 We briefly explain why Theorem \ref{thmsodB} implies Theorem \ref{thmsodA}, for full details see \cite[Proof of Theorem 3.2]{PTzero}.
    Consider the 
    morphisms 
    \begin{equation}\label{def:pullback}
    \X^f(d)^{\text{ss}} \stackrel{j}{\hookrightarrow} \X^f(d)\stackrel{\pi}{\twoheadrightarrow} \X(d).\end{equation}
    where $j$ is an open immersion and $\pi$ is the natural projection. 
    Let 
    \begin{align*}
     \mathbb{E}(d; \delta_d+\mu\sigma_d) \subset D^b(\X^f(d))
     \end{align*}
     be the subcategory 
     generated by the complexes $\pi^*(D)$ for $D\in \mathbb{D}(d; \delta_d+\mu\sigma_d)$. 
    If $\delta_d+\mu\sigma_d$ is a good weight, then there is an equivalence of categories   
    \[j^*\colon \mathbb{E}(d; \delta_d+\mu\sigma_d)\xrightarrow{\sim} D^b\left(\X^f(d)^{\text{ss}}\right),\] 
    see \cite[Proof of Proposition 3.16]{PTzero}.
    The equivalence follows from the theory of ``window categories" of Halpern-Leistner \cite{halp}, Ballard--Favero--Katzarkov \cite{MR3895631}, and the description of ``window categories" via explicit generators (due to Halpern-Leistner--Sam \cite{hls}) for the self-dual representation of $G(d)$:
    \begin{align*}
    R^f(d) \oplus V(d)^{\vee}=R(d)\oplus V(d)\oplus V(d)^{\vee}.\end{align*}  
    Further,
    there is a semiorthogonal decomposition 
    \begin{align*}
        D^b(\X^f(d))= \langle \pi^{\ast} D^b(\X(d))_w : w \in \mathbb{Z}   \rangle 
    \end{align*}
    and equivalences 
$\pi^{\ast} \colon D^b(\X(d))_w \stackrel{\sim}{\to} \pi^{\ast}D^b(\X(d))_w$ for each $w\in\mathbb{Z}$. 
Therefore, by Theorem~\ref{thmsodB}, there is a 
    semiorthogonal decomposition 
    \begin{align*}
    \mathbb{E}(d; \delta_d+\mu\sigma_d)=
        \left\langle \bigotimes_{i=1}^k \mathbb{M}(d_i; \theta_i+\delta_{d_i}+v_i\tau_{d_i})\right\rangle,
    \end{align*}
    where the right hand side is as in 
    Theorem~\ref{thmsodA}. 
    Therefore the theorem holds.
\end{proof}

\subsection{Decompositions of weights}\label{subsub421}

The proof of Theorem \ref{thmsodB} follows closely the proof of \cite[Proposition 3.12]{PTzero}. 

\begin{remark}\label{remark46}
 The assumption that the quiver $Q$ is very symmetric is used only to obtain a (convenient) decomposition \eqref{decompchi2} of the weight $\chi$, which allows us to perform a similar combinatorial analysis as in \cite{P2, PTzero}. Note that the decomposition \eqref{decompchi2} is a corollary of Proposition \ref{prop:tau}, where it is essential that $Q$ is very symmetric. 
 \end{remark}

%In this subsection, we explain that the methods used in loc. cit. can be also applied for the case of very symmetric quivers.
The proof of Theorem \ref{thmsodB} uses the decomposition of categorical Hall algebras for quivers with potential in quasi-BPS categories from \cite{P}. The summands in this semiorthogonal decomposition are indexed by decompositions of weights of $T(d)$, which we now briefly review.

Before stating it, we introduce some notations. For $d_a$ a summand of a partition of $d$, denote by $M(d_a)\subset M(d)$ the subspace as in the decomposition from Subsection~\ref{id}.
% and let $A\subset \{1,\ldots, d\}$ be the set of indices of weights of standard representation corresponding to $M(d_a)\subset M(d)$. 
Assume that $\ell$ is a partition of a dimension $d_a\in \mathbb{N}^I$, alternatively, $\ell$ is an edge of the tree $\mathcal{T}$ introduced in Subsection \ref{tree}.
Let $\lambda_{\ell}$ be the corresponding 
antidominant cocharacter of $T(d_a)$. 
Recall the set \[\mathscr{A}=\{(\beta^a_i-\beta^b_j)^{\times A}\mid a,b\in I, 1\leq i\leq i\leq d^a, 1\leq j\leq d^b\}\] from \eqref{def:poly}.
Let \[\mathscr{A}_{\ell}\subset M(d_a)\cap \mathscr{A}\] be the multiset of weights in $M(d_a)\cap \mathscr{A}$
such that $\langle \lambda_{\ell}, \beta\rangle>0$.
Define
$N_{\ell}$ by 
 \begin{align*}
 N_\ell:=\sum_{\beta\in \mathscr{A}_\ell}\beta.
 \end{align*}
% Define \[\mathfrak{g}_\ell:=\sum_{\beta\in\mathcal{W}_{\ell}}\beta.\]

\begin{prop}\label{prop:decompchi}
Let $\chi$ be a dominant weight in $M(d)_\mathbb{R}$,
let $\delta_d\in M(d)^{W_d}_\mathbb{R}$, and let $w=\langle 1_{d}, 
\chi-\delta_d\rangle\in\mathbb{R}$.
%such that $\chi+\rho+\delta\in \textbf{V}(d)$. We use the results in Subsection \ref{decompchitree} for the weight $\chi+\rho+\delta-(w+d\mu)\tau_d=\chi+\rho-w\tau_d$.
There exists:
\begin{enumerate}
    \item a path of partitions $T$, see Subsection \ref{tree}, with decomposition 
$(d_i)_{i=1}^k$ at the end vertex,
\item coefficients $r_\ell$ for $\ell\in T$ such that $r_{\ell}>1/2$ 
if $\ell$ corresponds to a partition with length $>1$, 
and $r_{\ell}=0$ otherwise; further, if $\ell, \ell'\in T$ are vertices 
corresponding to partitions with length $>1$, and 
with a path from $\ell$ to $\ell'$, then $r_{\ell}> r_{\ell'}> \frac{1}{2}$, and
\item dominant weights $\psi_i\in\mathbf{W}(d_i)$ for $1\leq i\leq k$
such that:
\end{enumerate}
\begin{equation}\label{decompchi}
    \chi+\rho-\delta_d=-\sum_{\ell\in T}r_\ell N_\ell+\sum_{i=1}^k\psi_i+w\tau_d.
\end{equation}
\end{prop}

% Let $\chi$ be a weight of $T(d)$, 
% let $\delta_d\in M(d)^{W_d}_\mathbb{R}$, and let $w=\langle 1_{d}, 
% \chi-\delta_d\rangle\in\mathbb{R}$.
% By \cite[Subsection 3.2.8]{P}, there exists a path of partitions (alternatively, of antidominant cocharacters) $T$ (see \cite[Subsection 2.2.9]{PTzero}) with associated decomposition $(d_i)_{i=1}^k$ such that 
% \begin{equation}\label{decompchi}
% \chi+\rho-\delta_d=-\sum_{\ell\in T}r_\ell N_\ell+\sum_{i=1}^k \psi_i+w\tau_d.
% \end{equation}
% The coefficients $r_\ell>\frac{1}{2}$ are such that if there exists an edge from $\ell$ to $\ell'$ in $T$, then $r_{\ell}>r_{\ell'}$. The weights $\psi_i$ are in $\mathbf{W}(d_i)$. The weight $N_{\ell}$ is defined as follows: $\ell$ corresponds to a partition of a summand $d_a$ of some partition of $d$. Let $\lambda$ be the antidominant partition corresponding to the partition $\ell$ of $d_a$. Let 
% $\mathscr{A}_{\ell} \subset \mathscr{A}$ be the subset of weights $\beta$ in $R(d_a)$
% such that $\langle \lambda, \beta \rangle>0$. 
% %\[\mathscr{A}_\ell=\{\beta\in \mathscr{A}\mid\beta\text{ wt in }R(d_a)\text{ and }\langle\lambda, \beta\rangle>0\}.\]
% Then 
% $N_{\ell}$ is defined by 
% \begin{align*}
% N_\ell:=\sum_{\beta\in \mathscr{A}_\ell}\beta.
% \end{align*}

\begin{proof}
The above is proved in \cite[Subsection 3.1.2]{P0}, see also
\cite[Subsection 3.2.8]{P2}.

We briefly explain the process of obtain the decomposition \eqref{decompchi}.
Choose $r$ such that $\chi+\rho-\delta_d-w\tau_d$ is on the boundary of $2r\mathbf{W}(d)$ (that is, let $r$ be the $r$-invariant of $\chi+\rho-w\tau_d$). The first partition $\ell_1$ in $T$ corresponds to the face of $2r\mathbf{W}(d)$ which contains $\chi+\rho-\delta_d-w\tau_d$ in its interior. Assume $\ell_1$ corresponds to a partition $(e_i)_{i=1}^s$.
Then there exist weights $\chi'_i\in M(e_i)_0$ for $1\leq i\leq s$ such that \[\chi+\rho-\delta_d-w\tau_d+r_{\ell_1}N_{\ell_1}=\sum_{i=1}^s \chi'_i.\] By the choice of $r$ and $\ell_1$, the weights $\chi'_i$ are inside the polytopes $2r\mathbf{W}(e_i)$. One repeats the process above to decompose further the weights $\chi'_i$ until the decomposition \eqref{decompchi} is obtained.
\end{proof}

Let $w\in\mathbb{Z}$ and assume that $\langle 1_d, \delta_d\rangle=w$.
Let $M(d)^+_w$ be the subset of $M(d)$ of integral dominant weights $\chi$ with $\langle 1_d, \chi\rangle=w$. 
We denote by $L^d_{\delta_d}$ the set of all paths of partitions $T$ with coefficients $r_\ell$ for $\ell\in T$ satisfying (2) from the statement of Proposition \ref{prop:decompchi} for a dominant integral weight $\chi\in M(d)^+_w$. By Proposition \ref{prop:decompchi}, there is a map:
\[\Upsilon\colon M(d)^+_w\to L^d_{\delta_d}, \,\, \Upsilon(\chi)=(T, r_\ell).\]

The set $L^d_{\delta_d}$ was used in \cite{P} to index summands in semiorthogonal decompositions of $D^b(\X(d))_w$.
We will show in the next subsection that $L^d_{\delta_d}$ has an explicit description.

\subsection{Partitions associated to dominant weights}

We continue with the notation from Subsection \ref{subsub421}.
Using the following proposition, the weight $-\sum_{\ell\in T}r_\ell N_\ell$ from \eqref{decompchi} is a linear combinations of $\tau_{d_i}$ for $1\leq i\leq k$.
%For $\lambda$ a cocharacter of $T(d)$, we use the notation\begin{align*}N^{\lambda>0}&=\sum_{\beta\in A_\lambda}\beta,\\\mathfrak{g}^{\lambda>0}&=\sum_{\beta\in\mathfrak{g}_\lambda}\beta.\end{align*}

\begin{prop}\label{prop:tau}
    Let $\lambda$ be an antidominant cocharacter associated to the partition $(d_i)_{i=1}^k$ of $d$. Then $R(d)^{\lambda>0}$ is a linear combination of the weights $\tau_{d_i}$ (alternatively, of $\sigma_{d_i}$) for $1
    \leq i\leq k$.
\end{prop}

\begin{proof}
    This follows from a direct computation, for example when $k=2$, one computes directly that 
    \[R(d)^{\lambda>0}=A
    \sum_{a, b \in I}\sum_{d_1^a<i \leq d^a, 1 \leq j\leq d_1^b}(\beta^b_j-\beta^a_i)=A\left(\dd_2\sigma_{d_1}-\dd_1\sigma_{d_2}\right).\]
\end{proof}
\begin{remark}
The conclusion of Proposition~\ref{prop:tau} is not true for a general symmetric quiver. 
\end{remark}

Consider the decomposition \eqref{decompchi} and let $\lambda$ be the antidominant cocharacter corresponding to $(d_i)_{i=1}^k$. 
We define $v_i \in \mathbb{R}$
for $1\leq i \leq k$ by 
\begin{align}\label{decompchi24}
    \sum_{i=1}^k v_i \tau_{d_i}=-\sum_{\ell \in T}\left(r_{\ell}-\frac{1}{2}\right)N_{\ell}+w\tau_d=-\sum_{\ell\in T} r_{\ell} N_{\ell}
    +\frac{1}{2}R(d)^{\lambda>0}+w\tau_d.
\end{align}
Here the right hand side is a linear combination 
of $\tau_{d_i}$ by Proposition~\ref{prop:tau}, so $v_i \in \mathbb{R}$ is well-defined. 
We
define the 
weights $\theta_i \in M(d_i)^{W_{d_i}}_{\mathbb{R}}$ by 
\begin{align}\label{defthetai}
    \sum_{i=1}^k \theta_i=-\frac{1}{2}R(d)^{\lambda>0}+\frac{1}{2}\mathfrak{g}(d)^{\lambda>0}.
\end{align}
Then we rewrite \eqref{decompchi} as
\begin{align}\label{decompchi2}
    \chi
   =\sum_{i=1}^k \theta_i+\sum_{i=1}^k v_i\tau_{d_i}+\sum_{i=1}^k(\psi_i-\rho_i+\delta_{d_i}),
\end{align}
where $\rho_i$ is half the sum of positive roots of $\mathfrak{g}(d_i)$. The next proposition follows as in \cite[Proposition 3.5]{PTzero}:
%\textcolor{red}{(If I understand correctly, you define $v_i$ by the identity $\sum v_i \tau_{d_i}=-\sum_{l\in T}(r_l-1/2)N_l$ as in~\cite[(3.15)]{PTzero}. Then I agree that Prop 2.4 holds as in the argument of~\cite[Prop 3.4]{PTzero}. But then why the first identity $\chi=\sum_{i=1}^k w_i \tau_{d_i}+\sum_{i=1}^k(\psi_i-\rho_i)$ in (2.5) holds? In fact the two identities in (2.5) and Prop 2.3 imply that $\mathfrak{g}^{\lambda>0}$ is a linear combination of $\tau_{d_i}$, but I think this is not necessary trueif $\lvert I \rvert >1$.)}
\begin{prop}\label{ineqslopes}
%Let $\mu\in \mathbb{R}\setminus\mathbb{Q}$ and let $\delta_d\in M(d)_\mathbb{Q}$.
    Let $\chi$ be a dominant weight in $M(d)$ and consider the weights $(v_i)_{i=1}^k$ from \eqref{decompchi24}. Then 
    \begin{equation}\label{ineqslopes2}
    \frac{v_1}{\dd_1}<\ldots<\frac{v_k}{\dd_k}.
    \end{equation}
\end{prop}

Let $T^d_w$ be the set of tuples $A=(d_i, v_i)_{i=1}^k$ of $(d,w)$ such that $(d_i)_{i=1}^k$ is a partition of $d$ and the real numbers $(v_i)_{i=1}^k\in\mathbb{R}^k$ are such that:
\begin{itemize}
    \item $\sum_{i=1}^k v_i=w$,
    \item the inequality \eqref{ineqslopes2} holds,
    \item for each $1\leq i\leq k$, the sum of coefficients of $\theta_i+v_i\tau_{d_i}+\delta_{d_i}$ is an integer.
\end{itemize}
By Proposition \ref{ineqslopes}, there is a map 
\begin{equation}\label{varphi}
\varphi\colon L^d_{\delta_d}\to T^d_w.
\end{equation}

\begin{prop}\label{prop410}
    The map \eqref{varphi} is a bijection.
\end{prop}

The above follows as \cite[Proposition 3.8]{PTzero}. The proof goes by constructing an inverse $\varphi'$. To construct $\varphi'$, one chooses $\psi_i\in \textbf{W}(d_i)$ such that $\chi$ defined in \eqref{decompchi2} is an integral weight.

Thus the summands appearing in the semiorthogonal decomposition of $D^b(\X(d))_w$ from \cite[Theorem 1.1]{P} are labeled by elements of $T^d_w$.

\subsection{Partitions for framed quivers}

The following proposition is the analogue of \cite[Proposition 3.7]{PTzero}.

\begin{prop}\label{prop55}
Let $\delta_d\in M(d)^{W_d}_\mathbb{R}$ and let $\chi \in M(d)$ be a dominant weight. 
Consider the decomposition (\ref{decompchi2})
with associated partition $(d_i)_{i=1}^k$ and weights $(v_i)_{i=1}^k\in \mathbb{R}^k$. 
Let $\mu\in \mathbb{R}$ and assume that \begin{align}\label{chi:Vd}
\chi+\rho-\delta_d-\mu\sigma_d\in\mathbf{V}(d).\end{align}
Then 
    \begin{equation}\label{ineqslopesmu}
    \mu\leq \frac{v_1}{\dd_1}<\ldots<\frac{v_k}{\dd_k}\leq 1+\mu.
     \end{equation}
\end{prop}
\begin{proof}
    Using the decomposition \eqref{decompchi2}, we have that
    \begin{align}\label{decompC}
    \chi+\rho-\delta_d-\mu\sigma_d 
    =-\frac{1}{2}R(d)^{\lambda>0}+\sum_{i=1}^k (v_i-\mu\dd_i)\tau_{d_i}+\sum_{i=1}^k\psi_i. 
   \end{align}
   Let $\alpha_k$ be the (dominant) cocharacter of $T(d)$ which acts with weight $1$ on $\beta^a_i$ for $a\in I$ and $d^a-d^a_k<i\leq d^a$ and with weight $0$ on $\beta^a_i$ for $a\in I$ and $d^a-d^a_k\geq i$. 
 By (\ref{decompC}), we have 
that \begin{align}\label{id:lambdak}
\langle \alpha_k, \chi+\rho-\delta_d-\mu\sigma_d\rangle=\left\langle\alpha_k, -\frac{1}{2}R(d)^{\lambda>0}\right\rangle+v_k-\mu\dd_k.
    \end{align}
    On the other hand, we have 
\begin{align*}
    \mathbf{V}(d)=\frac{1}{2} \mathrm{sum}_{\beta \in \mathscr{C}}[0, \beta],
\end{align*}
       where recall that \[\mathscr{C}=\{(\beta^a_i-\beta^b_j)^{\times A}, \beta^a_i\mid a,b\in I, 1\leq i\leq d^a, 1\leq j\leq d^b\}.\]
   Then $-R(d)^{\lambda>0}/2+\dd_k \tau_{d_k}$ has maximum $\alpha_k$-weight among weights in 
    $\mathbf{V}(d)$. 
    Therefore 
    from (\ref{chi:Vd}) we obtain that 
    \begin{align}\label{chi:Vd2}\langle \alpha_k, \chi+\rho-\delta_d-\mu\sigma_d\rangle\leq \left\langle\alpha_k, -\frac{1}{2}R(d)^{\lambda>0}\right\rangle+\dd_k.\end{align}
    By comparing (\ref{id:lambdak}) and (\ref{chi:Vd2}), we
    conclude that  
    $v_k-\mu\dd_k\leq \dd_k$, so  
    $v_k/\dd_k \leq 1+\mu$. 
    A similar argument also shows the lower bound.
\end{proof}

%In view of Proposition \ref{prop410}, the summands of the semiorthogonal decomposition of $\mathbb{D}(d; \delta_d)$ will be labelled by $A\in T^d_w$ such that \eqref{ineqslopesmu} holds.

%Recall the definition of the set $\mathscr{A}_\lambda$ from \eqref{def:alambda}.

\begin{cor}\label{cor56}
In the setting of Proposition~\ref{prop55}, 
consider the decomposition (\ref{decompC}). 
    %Let $\chi$ be a dominant weight such that $\chi+\rho-\delta_d-\mu\sigma_d\in\mathbf{V}(d)$ with associated partition $(d_i)_{i=1}^k$ and weights $(v_i)_{i=1}^k\in\mathbb{R}^k$. 
    %Then the decomposition \eqref{decompC} can be written as
    %\[\chi+\rho-\delta_d-\mu\sigma_d=-\frac{1}{2}R(d)^{\lambda>0}+\sum_{i=1}^k (v_i-\mu\dd_i)\tau_{d_i}+\sum_{i=1}^k \psi_i\] such that
  Recall the sets $\mathscr{A}$ and $\mathscr{B}$ from (\ref{def:poly}) and define $\mathscr{A}_{\lambda}:=\{\beta \in \mathscr{A} \mid \langle \lambda, \beta \rangle >0\}$.
    Then there are $\psi_i\in \mathbf{W}(d_i)$ for $1\leq i\leq k$ such that: 
    \begin{align*}
        -\frac{1}{2}R(d)^{\lambda>0}&\in \frac{1}{2}\mathrm{sum}_{\beta\in\mathscr{A}_\lambda}[0, -\beta],\
        \sum_{i=1}^k (v_i-\mu\dd_i)\tau_{d_i}\in \mathrm{sum}_{\beta\in\mathscr{B}}[0,\beta].
    \end{align*}
    % Here $\mathscr{A}_{\lambda}:=\{\beta \in \mathscr{A} \mid \langle \lambda, \beta \rangle >0\}$, and 
    % $\mathscr{A}$ is defined in (\ref{def:poly}). 
\end{cor}
 
\begin{proof}
    The inclusion 
    \begin{align*}\sum_{i=1}^k (v_i-\mu\dd_i)\tau_{d_i}\in \mathrm{sum}_{\beta\in\mathscr{B}}[0,\beta]
    \end{align*}
    follows from Proposition \ref{prop55}, see \cite[Proposition 3.8]{PTzero}. The rest of the decomposition is immediate.
\end{proof}

\subsection{Comparison of partitions}\label{comppartitions}

In this subsection, we explain the order used in the semiorthogonal decomposition from Theorem \ref{thmsodA}, see also the discussion in Subsection \ref{notation}.

Fix $d\in\mathbb{N}^I$ and a weight $\delta^\circ\in M(d)^{W_d}_{0,\mathbb{R}}$. Define $L^d_{\delta^\circ, w}:=L^{d}_{\delta^\circ+w\tau_d}$ and  $L^d_{\delta^\circ}:=\bigcup_{w\in\mathbb{Z}}L^d_{\delta^\circ, w}$.
We define a set 
\[O\subset L^d_{\delta^\circ}\times L^d_{\delta^\circ}\]
which is used to compare summands of semiorthogonal decompositions, see Subsection \ref{notation}.
 
For $w>w'$, let $O_{w, w'}:=L^d_{\delta^\circ, w}\times L^d_{\delta^\circ, w'}$. For $w<w'$, let $O_{w, w'}$ be the empty set.

We now define $O_{w,w}\subset L^d_{\delta^\circ, w}\times L^d_{\delta^\circ, w}$.
The general procedure for defining such a set, equivalently for comparing two partitions for an arbitrary symmetric quiver is described in \cite[Subsection 3.3.4]{P}. 
Consider the path of partitions $T_A$ with coefficients $r_{\ell, A}$ as in \eqref{decompchi} corresponding to $A\in L^d_{\delta^\circ, w}$. 
Order the coefficients $r_{\ell, A}$ in decreasing order $r'_{1,A}>r'_{2,A}>\cdots>r'_{f(A),A}.$ Each $r'_{i, A}$ for $1\leq i\leq f(A)$ corresponds to a partition $\pi_{i,A}$.  
Similarly, consider the path of partitions $T_B$ with coefficients $r_{\ell, B}$ corresponding to $B\in L^d_{\delta^\circ, w}$. Define similarly $r'_{1,B}>\cdots>r'_{f(B),B}$ and $\pi_{i,B}$ for $1\leq i\leq f(B)$. 

Define the set $R\subset L^d_{\delta^\circ, w}\times L^d_{\delta^\circ, w}$ which contains pairs $(A,B)$ such that

\begin{itemize}
    \item there exists $n\geq 1$ such that $r'_{n, A}>r'_{n, B}$ and $r'_{i, A}=r'_{i, B}$ for $i<n$, or
    \item there exists $n\geq 1$ such that $r'_{i, A}=r'_{i, B}$ for $i\leq n$, $\pi_{i, A}=\pi_{i, B}$ for $i<n$, and $\pi_{n, B}\geq \pi_{n, A}$, see Subsection \ref{compa}, or
    \item are of the form $(A, A)$.
\end{itemize}
We then let $O_{w,w}:=L^d_{\delta^\circ, w}\times L^d_{\delta^\circ, w}\setminus R$ and 
\begin{equation}\label{def:setO}
O:=\bigcup_{w,w'\in\mathbb{Z}}O_{w, w'}.\end{equation}

We will only use that such an order exists in the current paper. 

In order to make the above process more accessible, we explain how to compute $r'_{1,A}$ and $\pi_{1,A}$.
From Proposition \ref{prop410}, there is an isomorphism of sets $L^d_{\delta^\circ, w}\cong T^d_w$.

%The $r$-invariant of $\chi+\rho$ is attained for an antidominant cocharacter $\lambda$. 
For a dominant weight $\theta\in M(d)^+_{\mathbb{R}}$ with $\langle 1_d, \theta\rangle=w$, its $r$-invariant is the smallest $r$ such that $
\theta-w\tau_d\in 2r\textbf{W}$.
Equivalently, the $r$-invariant of a dominant weight $\theta\in M(d)^+_\mathbb{R}$ is the maximum after all dominant cocharacters $\mu$ of $ST(d)$:
\begin{equation}\label{def:rtheta}
r(\theta)=\text{max}_\mu \frac{\langle \mu, \theta\rangle}{\Big\langle \mu, R(d)^{\mu>0}\Big\rangle},
\end{equation}
see \cite[Subsection 3.1.1]{P}. 
Assume that $A=(d_i, v_i)_{i=1}^k\in T^d_w$; we also denote by $A$ the corresponding element of $T^d_{w}$.
Then one can show that 
\[r'_{1,A}=r(\chi_A+\rho),\]
where $\chi_A:=\sum_{i=1}^k v_i\tau_{d_i}+\sum_{i=1}^k\theta_i\in M(d)_\mathbb{R}$, see \eqref{defthetai} for the definition of $\theta_i$. We let $\lambda$ be the antidominant cocharacter corresponding to the partition $(d_i)_{i=1}^k$.
By \eqref{decompchi24} and Proposition \ref{prop:tau},
write \[\chi_A+\frac{1}{2}\mathfrak{g}(d)^{\lambda<0}=\sum_{i=1}^k w_i\tau_{d_i}.\]
There is a transformation:
\[(d_i, v_i)_{i=1}^k\mapsto (d_i, w_i)_{i=1}^k.\]
We will compute $r'_{1,A}$ in terms of $(w_i)_{i=1}^k$.
% Alternatively, one can 
 Let $\mu$ be a dominant cocharacter attaining the maximum above and
 assume the associated partition of $\mu$ is $(e_i)_{i=1}^s$. Then the maximum in \eqref{def:rtheta} is also attained for the cocharacter $\mu$ with associated partition $\left(\sum_{i\leq b} e_i, \sum_{i>b} e_i\right)$, see \cite[Proposition 3.2]{P}. 
We have that $(d_i)_{i=1}^k\geq (e_i)_{i=1}^s$, see Subsection \ref{compa} for the notation, and Proposition \ref{prop:decompchi}.

Let $\mu_a$ be a cocharacter which acts with weight $\sum_{i\leq a} d_i$ on $\beta_j$ for $j>a$ and weight $-\sum_{i>a} d_i$ on $\beta_j$ for $j\leq a$. 
We have that
\[r(\chi_A+\rho)=\text{max}_a \frac{\langle \mu_a, \chi_A+\rho\rangle}{\Big\langle \mu_a, R(d)^{\mu_a>0}\Big\rangle},\] where the maximum is taken after all $1\leq a<k$.
We compute
\[\Big\langle \mu_a, R(d)^{\mu_a>0}\Big\rangle=
A\underline{d}\left(\sum_{i>a} \underline{d}_i\right)\left(\sum_{i\leq a}\underline{d}_i\right).
\]
Then
\[r(\chi_A+\rho)=r'_{1,A}=\frac{1}{A}\text{max}_a\left(\frac{\sum_{i>a} w_i}{\sum_{i>a} \dd_i}-\frac{\sum_{i\leq a} w_i}{\sum_{i\leq a} \dd_i}\right),\]
where the maximum is taken after $1\leq a\leq k$.
The partition $\pi_{1,A}$ can be reconstructed from all $1\leq a<k$ for which $\mu_a$ attains the maximum of \eqref{def:rtheta}.
Assume the set of all such $1\leq a < k$ is $1\leq a_2<\cdots<a_s< k$. Then $\pi_{1,A}=(e_i)_{i=1}^s$ is the partition of $d$ with terms:
\[(d_1+\ldots+d_{a_2},\ldots, d_{a_s+1}+\ldots+d_k).\]

\subsection{Semiorthogonal decompositions for very symmetric quivers}

We now prove Theorem \ref{thmsodB}, and thus Theorem \ref{thmsodA}.
 
  \begin{proof}[Proof of Theorem \ref{thmsodB}]
     The same argument used to prove \cite[Proposition 3.9]{PTzero} applies here. The argument in loc. cit. was organized in three steps:
     \begin{enumerate}
         \item for $(d_i)_{i=1}^k$ a partition of $d$ and $(v_i)_{i=1}^k\in
         \mathbb{R}^k$ as in the statement of Theorem \ref{thmsodB},
         the categorical Hall product \[\boxtimes_{i=1}^k \mathbb{M}(d_i; \theta_i+\delta_{d_i}+v_i\tau_{d_i})\to D^b(\X(d))\] has image in $\mathbb{D}(d; \delta_d+\mu\sigma_d)$,

         \item the categories on the left hand side of \eqref{SODD} are semiorthogonal for the ordering introduced in Subsection \ref{comppartitions}, see \cite[Subsection 3.3]{P}, \cite[Subsection 3.4]{PTzero},

         \item the categories on the left hand side of \eqref{SODD} generate $\mathbb{D}(d; \delta_d+\mu\sigma_d)$.
     \end{enumerate}
     The proofs of (2) and (3) are exactly as in loc. cit. and follow from the semiorthogonal decomposition of the categorical Hall algebra for a quiver from \cite{P}. 

     We explain the shifts used in defining the categories on the right hand side of \eqref{SODD}. Consider the weights $\chi_i\in M(d_i)$ such that $\sum_{i=1}^k\chi_i=\chi$.
     The decomposition \eqref{decompchi2} can be rewritten as
     \[\sum_{i=1}^k \left(\chi_i+\rho_i-\delta_{d_i}\right)=\sum_{i=1}^k \left(\theta_i+v_i\tau_{d_i}+\psi_i\right),\] and so
     \[\chi_i+\rho_i-\left(\theta_i+\delta_{d_i}+v_i\tau_{d_i}\right)=\psi_i\in \mathbf{W}(d_i).\]
     The proof of (1) follows from Corollary \eqref{cor56} and an explicit resolution by vector bundles of the Hall product of generators of categories $\mathbb{M}(d_i; \theta_i+\delta_{d_i}+v_i\tau_{d_i})$ for $1\leq i\leq k$, see Step 1 in the proof of \cite[Proposition 3.9]{PTzero} and \cite[Proposition 2.1]{PTzero}. 
\end{proof}

 In Theorem \ref{thmsodB}, we obtained a semiorthogonal decomposition after decomposing the polytope $\mathbf{V}(d)$ in translates of direct sums of the polytopes $\mathbf{W}(e)$ for dimension vectors $e$ which are parts of partitions of $d$. 
 We can also decompose the full weight lattice $M(d)$ to prove Theorem \ref{sodfullstackA}, which follows \cite[Theorem 1.1]{P} for the quiver $Q$ and Proposition \ref{ineqslopes}, see also the argument in \cite[Corollary 3.3]{P2}.

\begin{thm}\label{sodfullstackA}
 Let $Q$ be a very symmetric quiver.
     Then Theorem \ref{sodfullstackB} holds for $Q$.
 \end{thm}

Note that Theorem \ref{sodfullstackA} follows from the analogues (for the full weight lattice) of Steps (2) and (3) from the proof of Theorem \ref{thmsodB}.

\subsection{A class of symmetric quivers}\label{symmquivers}

In this section, we discuss some preliminaries before proving Theorems \ref{thmsodC} and \ref{sodfullstackB}. We stop assuming $Q$ is a very symmetric quiver.

Let $Q=(I, E)$ be a symmetric quiver such that the number of loops at each vertex $a\in I$ has the same parity $\varepsilon\in \mathbb{Z}/2\mathbb{Z}$. 
We construct a very symmetric quiver $Q^\gimel=(I, E^\gimel)$ with potential as follows. For $a,b\in I$, let $e_{ab}$ be the number of edges from $a$ to $b$ in $Q$.
Choose $A \in \mathbb{Z}_{\geq 1}$ such that 
\[A\geq \text{max}\{e_{ab}\mid a,b\in I\}\text{ and }A\equiv \varepsilon \,(\text{mod }2).\]
For $a\in I$, let $c_a\in\mathbb{N}$ be defined by 
\[c_{a}:=\frac{A-e_{aa}}{2}.\]
Add loops $\{\omega_k\mid 1\leq k\leq c_a\}$ and their opposites $\{\overline{\omega}_k\mid 1\leq k\leq c_a\}$ at $a$ and define the potential \[W_a:=\sum_{k=1}^{c_a}\omega_k \overline{\omega}_k.\] 
Fix a total ordering on $I$.
For two different vertices $a<b$ in $I$, let $c_{ab}:=A-e_{ab}$.
Add edges $\{e_k\mid 1\leq k\leq c_{ab}\}$ from $a$ to $b$ and their opposites $\{\overline{e}_k\mid 1\leq k\leq c_{ab}\}$ from $b$ to $a$. Let 
\[W_{ab}:=\sum_{k=1}^{c_{ab}}e_k\overline{e}_k.\]
Consider the potential
$W^\gimel$ of $Q^\gimel$
\[W^\gimel:=\sum_{a\in I}W_a+\sum_{a, b \in I, a<b}W_{ab}\] of the quiver $Q^\gimel$.
For $d\in\mathbb{N}^I$, let $U(d)$ be the affine space of linear maps corresponding to the edges \[\bigsqcup_{a\in I}\{\omega_k\mid 1\leq k\leq c_a\}\sqcup \bigsqcup_{a<b}\{e_k\mid 1\leq k\leq c_{ab}\}.\]
The stack of representations of dimension $d$ of $Q^\gimel$ is 
\[\X^\gimel(d):=R^\gimel(d)/G(d):=
\left(R(d)\oplus U(d)\oplus U(d)^{\vee}\right)/G(d).\]
Consider the action of $\mathbb{C}^*$ on 
\begin{align*}
  R^\gimel(d):=R(d)\oplus U(d)\oplus U(d)^{\vee}
  \end{align*}
  of weight $(0, 0, 2)$. 
  Let $s,v$ be the maps
\[\X(d)\xleftarrow{v}\left(R(d)\oplus U(d)^{\vee}\right)/G(d)\xrightarrow{s}\X'(d)\]
where $s$ is the inclusion and $v$ is the projection. 
The Koszul equivalence (\ref{equiv:thetaX}) gives an equivalence (Kn\"orrer periodicity):
\begin{equation}\label{Koszulqprime}
\Theta_d=s_{\ast}v^{\ast}\colon D^b(\X(d))\xrightarrow{\sim}\mathrm{MF}^{\mathrm{gr}}(\X^\gimel(d), \mathrm{Tr}\,W^\gimel).
\end{equation}
We may write $\Theta$ instead of $\Theta_d$ when the dimension vector $d$ is clear from the context.
%Then
%\begin{equation}\label{descriptionHall}
%\Theta_d:=p_*q^*,
%\end{equation}
%where $p$ and $q$ are the maps

%Consider a partition $(d_a)_{a=1}^k$ of $d$ and consider weights $(w_a)_{a=1}^k\in\mathbb{Z}^k$. Let $\lambda$ be a dominant character with associated partition $(d_a)_{a=1}^k$. For a representation $V$ of $T(d)$, we abuse notation and denote by $V^{\lambda>0}$ the $T(d)$-weight of the one dimensional $T(d)$-representation $\det(V^{\lambda>0})$. 
%Define the weights $(v_a)_{a=1}^k\in\left(\frac{1}{2}\mathbb{Z}\right)^k$ via the equality: \begin{align}\label{defv2} \sum_{a=1}^k v_a\tau_{d_a} &:=\sum_{a=1}^k w_a\tau_{d_a}-\frac{1}{2}\left(R'(d)^{\lambda>0}-\mathfrak{g}^{\lambda>0}\right)+U(d)^{\lambda>0}\\\notag &:=\sum_{a=1}^k w_a\tau_{d_a}-\frac{1}{2}\left(R(d)^{\lambda>0}+(U(d)^{\vee})^{\lambda>0}-U(d)^{\lambda>0}-\mathfrak{g}^{\lambda>0}\right).\end{align}
%\textcolor{red}{($U(d)^{\lambda >0}$ is not very symmetric, so the right hand side of (2.7) does not seem to be written as a linear combination of $\tau_{d_a}$.)}

Denote by $\X^f(d)^{\text{ss}}$ and $\X^{\gimel f}f(d)^{\text{ss}}$ the varieties of stable framed representations of $Q$ and $Q^\gimel$, respectively.

%We use Theorem \ref{thmsodA} and the Koszul equivalence \eqref{Koszulqprime} to construct a semiorthogonal decomposition of $D^b\left(\X^f(d)^{\text{ss}}\right)$ in products of quasi-BPS categories of $Q$.

%\begin{thm}\label{thmsodC}Let $Q$ be a symmetric quiver such that the number of loops at each vertex $i\in I$ has the same parity.Let $\mu\in \mathbb{R}\setminus\mathbb{Q}$. There is a semiorthogonal decomposition \[D^b\left(\X^f(d)^{\text{ss}}\right)=\left\langle \bigotimes_{i=1}^k \mathbb{M}(d_i; \theta_i+v_i\tau_{d_i})\right\rangle.\] The right hand side is after all partitions $(d_i)_{i=1}^k$ of $d$ and all weights $(v_i)_{i=1}^k\in\left(\frac{1}{2}\mathbb{Z}\right)^k$ such that the sum of coefficients of $\theta_i+v_i\tau_{d_i}$ is an integer for all $1\leq i\leq k$ and\[\mu\leq \frac{v_1}{\dd_1}<\cdots<\frac{v_k}{\dd_k}<1+\mu.\] The weights $\theta_i$ are defined as follows. Let $\lambda$ be an antidominant cocharacter with associated partition $(d_i)_{i=1}^k$. Then \begin{equation}\label{def:deltai} \sum_{i=1}^k \theta_i=-\frac{1}{2}R(d)^{\lambda>0}+\frac{1}{2}\mathfrak{g}(d)^{\lambda>0}.\end{equation}\end{thm}

We will prove Theorem \ref{thmsodC} in the next section. The same argument also applies to obtain Theorem \ref{sodfullstackB} using Theorem \ref{sodfullstackA}.

%\begin{thm}\label{sodfullstackB}Let $Q$ be a symmetric quiver such that the number of loops at each vertex $i\in I$ has the same parity.Let $\mu\in \mathbb{R}$. There is a semiorthogonal decomposition \[D^b\left(\X(d)\right)=\left\langle \bigotimes_{i=1}^k \mathbb{M}(d_i; \theta_i+v_i\tau_{d_i})\right\rangle.\] The right hand side is after all partitions $(d_i)_{i=1}^k$ of $d$ and all weights $(v_i)_{i=1}^k\in\left(\frac{1}{2}\mathbb{Z}\right)^k$ such that the sum of coefficients of $\theta_i+v_i\tau_{d_i}$ is an integer for all $1\leq i\leq k$ and \[\frac{v_1}{\dd_1}<\cdots<\frac{v_k}{\dd_k}.\] The weights $\theta_i$ are defined as follows. Let $\lambda$ be an antidominant cocharacter with associated partition $(d_i)_{i=1}^k$. Then \[\sum_{i=1}^k \theta_i=-\frac{1}{2}R(d)^{\lambda>0}+\frac{1}{2}\mathfrak{g}(d)^{\lambda>0}.\] %The weights $\theta_i^\circ$ are in $M(d_i)_0$ for all $1\leq i\leq k$.\end{thm}

\subsection{Proof of Theorem \ref{thmsodC}}\label{subsec:proof}
The order of summands in the semiorthogonal decomposition is induced from the order in Subsection \ref{comppartitions} for the quiver $Q^\gimel$. Note that $Q^\gimel$ depends on the choice of a certain integer $A$, but we do not discuss the dependence of the order on this choice and only claim that such an order exists.

For a partition $(d_i)_{i=1}^k$ of $d$ and $\lambda$ an associated antidominant cocharacter, 
 we define the weights $\theta^\gimel_i\in M(d_i)_{\mathbb{R}}$ such that
 \[\sum_{i=1}^k \theta^\gimel_i=-\frac{1}{2}R^\gimel(d)^{\lambda>0}+\frac{1}{2}\mathfrak{g}(d)^{\lambda>0}.\]
 For $\delta_d\in M(d)_\mathbb{R}^{W_d}$, denote by $\mathbb{M}^\gimel(d; \delta_d) \subset D^b(\X^\gimel(d))$ the magic categories \eqref{defmdw} for the quiver $Q^\gimel$. Consider the quasi-BPS categories: 
 \begin{align*}
\mathbb{S}^{\gimel\rm{gr}}(d; \delta_d) :=\mathrm{MF}^{\rm{gr}}(\mathbb{M}^\gimel(d; \delta), \Tr W')
\subset \mathrm{MF}^{\rm{gr}}(\X^\gimel(d), \Tr W^\gimel). 
 \end{align*}
 Define \[\delta^\circ_d:=-\frac{1}{2}\det U(d)=-\frac{1}{2} U(d),\] where above, as elsewhere, we abuse notation and denote by $U(d)$ the sum of weights of $U(d)$.
Further, define 
\[\delta^\gimel_d=\delta_d+\delta^\circ_d.\] We will use the notations $\delta^\gimel_{d_i}, \delta_{d_i}$ from \eqref{def:deltai}. Note that $\delta^\gimel_d+\mu\sigma_d$ is a good weight if and only if $\delta_d+\mu\sigma_d$ is a good weight because $\langle \lambda, \delta^\circ_d\rangle\in\frac{1}{2}\mathbb{Z}$ for all $\lambda$ as in Definition \ref{def:generic2}.

\begin{step}
    There is a semiorthogonal decomposition
    \[\mathrm{MF}^{\mathrm{gr}}\left(\X^{\gimel f}(d)^{\text{ss}}, \mathrm{Tr}\,W^\gimel\right)=\left\langle \bigotimes_{i=1}^k \mathbb{S}^{\gimel\mathrm{gr}}(d_i; \theta^\gimel_i+\delta^\gimel_{d_i}+v_i \tau_{d_i}) : \mu \leq \frac{v_1}{\dd_1}<\cdots<\frac{v_k}{\dd_k}<1+\mu \right \rangle,\] where the right hand side is as in Theorem \ref{thmsodA}.
\end{step}

 \begin{proof}
 The claim follows
 by applying matrix factorizations \cite[Proposition 2.5]{PTzero}, \cite[Proposition 2.1]{P0} for the potential $W^\gimel$ to the semiorthogonal decomposition of Theorem \ref{thmsodA} for the quiver $Q^\gimel$.
 \end{proof}

\begin{step}
    There is an equivalence: \[\Theta_d^f\colon D^b\left(\X^f(d)^{\text{ss}}\right)\xrightarrow{\sim}\mathrm{MF}^{\mathrm{gr}}\left(\X^{\gimel f}(d)^{\text{ss}}, \mathrm{Tr}\,W^\gimel\right).\]
\end{step}
     
    \begin{proof}
    Consider the natural projection map
    $\pi^\gimel\colon \X^{\gimel f}(d)\to \X^f(d).$ Then \[(\pi^{\gimel})^{-1}\left(\X^f(d)^{\text{ss}}\right)\subset \X^{f \gimel}(d)^{\text{ss}}\] is an inclusion of open sets and \begin{align}\label{Crit:open}(\pi^{\gimel})^{-1}\left(\X^f(d)^{\text{ss}}\right)\cap \mathrm{Crit}(\mathrm{Tr}\,W^\gimel)=\X^{\gimel f}(d)^{\text{ss}}\cap \mathrm{Crit}(\mathrm{Tr}\,W^\gimel).\end{align}
   We have equivalences
    \[\mathrm{MF}^{\mathrm{gr}}\left(\X^{\gimel f}(d)^{\text{ss}}, \mathrm{Tr}\,W^\gimel\right)\stackrel{\sim}{\to} \mathrm{MF}^{\mathrm{gr}}\left(\pi^{\gimel -1}\left(\X^f(d)^{\text{ss}}\right), \mathrm{Tr}\,W^\gimel\right)\stackrel{\sim}{\leftarrow}D^b\left(\X^f(d)^{\text{ss}}\right).\]
    Here the first equivalence follows from (\ref{Crit:open})
    and (\ref{rest:MFU}), and the second equivalence 
    is an instance of the Koszul equivalence from Theorem~\ref{thm:Koszul}. 
    \end{proof}
    %For $1\leq i\leq k$, let $\theta^\circ_i:=-\det P(d_i)=-P(d_i)\in M(d_i)_\mathbb{R}$.

 The claim of Theorem \ref{thmsodC} then follows from:

 \begin{step}\label{step3}
 The equivalence (\ref{Koszulqprime})
 restricts to the equivalence 
 \begin{align*}
\Theta_d\colon \bigotimes_{i=1}^k\mathbb{M}(d_i; \theta_i+\delta_{d_i}+v_i\tau_{d_i})\stackrel{\sim}{\to} \bigotimes_{i=1}^k\mathbb{S}^{\gimel \mathrm{gr}}(d_i; \theta^\gimel_i+\delta^\gimel_{d_i}+v_i\tau_{d_i}),
\end{align*}
where the tensor products on the left and right hand sides are 
embedded by categorical Hall products into 
$D^b(\X(d))$ and $\mathrm{MF}^{\rm{gr}}(\X^\gimel(d), \Tr W^\gimel)$, respectively.  
  %   \begin{equation}\label{Hallequiv}
%\bigotimes_{i=1}^k\mathrm{MF}^{\mathrm{gr}}\left(\mathbb{M}'(d_i; \theta'_i+\delta'_{d_i}+v_i\tau_{d_i}), \mathrm{Tr}\,W'\right)\cong \bigotimes_{i=1}^k\mathbb{M}(d_i; \theta_i+\delta_{d_i}+v_i\tau_{d_i}).
%\end{equation} 
 \end{step}
    \begin{proof}
By the compatibility of the Koszul equivalence with 
categorical Hall products in Proposition~\ref{prop:compati:hall} and 
with quasi-BPS categories in Proposition~\ref{prop:period}, 
it is enough to check that 
\[\sum_{i=1}^k \left(\theta_i+v_i\tau_{d_i}+\delta_{d_i}\right)-U(d)^{\lambda>0}=\sum_{i=1}^k\left(\theta^\gimel_i+\delta^\gimel_{d_i}+v_i\tau_{d_i}+\frac{1}{2}U(d_i)\right).\]
Recall that \begin{align*}
    &\sum_{i=1}^k U(d_i)=U(d)^\lambda,\
  \sum_{i=1}^k \delta^\circ_{d_i}=-\frac{1}{2}U(d),\\
   & \sum_{i=1}^k (\theta^\gimel_i-\theta_i)=-\frac{1}{2}\left(R^\gimel(d)^{\lambda>0}-R(d)^{\lambda>0}\right).
\end{align*}
It thus suffices to show that:
\[-\frac{1}{2}\left(R^\gimel(d)^{\lambda>0}-R(d)^{\lambda>0}\right)-\frac{1}{2}U(d)+\frac{1}{2}U(d)^\lambda+U(d)^{\lambda>0}=0,\]
which can be verified by a direct computation.
    \end{proof}

\subsection{More classes of quivers}
Note that Theorem \ref{thmsodC} applies for any tripled quiver.
The semiorthogonal decomposition in Theorem \ref{thmsodC} is particularly simple for quivers satisfying the following assumption:
\begin{assum}\label{assum11}
The quiver $Q=(I,E)$ is symmetric and:
\begin{itemize}
    \item for all $a, b \in I$ different, the number of edges from $a$ to $b$ is even, and
    \item for all $a\in I$, the number of loops at $a$ is odd.
\end{itemize}
\end{assum}

Examples of quivers satisfying Assumption \ref{assum11} are 
tripled quivers $Q$ of quivers $Q^\circ=(I, E^\circ)$ satisfying the following assumption, where recall $\alpha_{a, b}$ from (\ref{alphaab}):

\begin{assum}\label{assum1}
For all $a, b \in I$, 
we have $\alpha_{a, b} \in 2\mathbb{Z}$. 
%we have that \begin{align}\label{cond:ab}
%	\sharp(a \to b\text{ in }E^\circ) -\sharp(b \to a\text{ in }E^\circ) \in 2\mathbb{Z}.
%	\end{align} 
\end{assum}

For example, Assumption \ref{assum1} is satisfied if $Q^{\circ}$ is symmetric. Further, the moduli stack of semistable sheaves on a K3 surface is locally described by the stack of representations of a preprojective algebra of a quiver satisfying Assumption \eqref{assum1}, see \cite{PTK3}.

We discuss the particular case of Theorem \ref{thmsodC} for quivers satisfying Assumption \ref{assum11}.

\begin{cor}\label{thmsodD}
Let $Q$ be a quiver satisfying Assumption \ref{assum11}.
    Let $\mu\in \mathbb{R}$ be such that $\mu\sigma_d$ is a good weight. Then there is a $X(d)$-linear semiorthogonal decomposition:
    \begin{equation}\label{cor29}
D^b\left(\X^f(d)^{\text{ss}}\right)=\left\langle \bigotimes_{i=1}^k \mathbb{M}(d_i)_{v_i} : 
\mu\leq \frac{v_1}{\dd_1}<\cdots<\frac{v_k}{\dd_k}<1+\mu
\right\rangle,
     \end{equation}
    where the right hand side is after all partitions $(d_i)_{i=1}^k$ of $d$ and integers $(v_i)_{i=1}^k\in\mathbb{Z}^k$ satisfying the above inequality.
    %The weights $\theta_i\in M(d_i)$ are defined in \eqref{def:deltai}.
\end{cor}

\begin{proof}
We set $\delta_d=0$ in Theorem~\ref{thmsodC}. 
%The claim to be proved is that the weights $v_i$ are integers.
Fix $d\in\mathbb{N}^I$. For $a\in I$, let $V^a$ be a $\mathbb{C}$-vector space of dimension $d^a$.
For each $a, b\in I$, let $V^{ab}:=\text{Hom}\left(V^a, V^b\right)$, and $e^{ab}$ denotes the number of edges from $a$ to $b$.
Then 
\[\frac{1}{2}R(d)^{\lambda>0}-\frac{1}{2}\mathfrak{g}(d)^{\lambda>0}=\sum_{a\in I}\frac{e^{aa}-1}{2}\mathfrak{g}(d)^{\lambda>0}+\sum_{a\neq b\in I}\frac{e^{ab}}{2}(V^{ab})^{\lambda>0}\in M(d).\]
Thus $\theta_i\in M(d_i)^{W_{d_i}}$, so the weights $v_i$ are integers for $1\leq i\leq k$. 
Moreover there is an equivalence 
$\mathbb{M}(d_i)_{v_i}=\mathbb{M}(d_i; v_i \tau_{d_i})\stackrel{\sim}{\to} \mathbb{M}(d_i; \theta_i+v_i \tau_{d_i})$
by taking the tensor product with $\theta_i$. 
    The claim then follows from Theorem \ref{thmsodC}.  
\end{proof}

\subsection{More framed quivers}\label{framedquiversnew}
The semiorthogonal decompositions in Theorems \ref{thmsodA} and Corollary \ref{thmsodD} also hold for spaces of semistable representations of the quivers $Q^{\alpha f}$, where $Q=(I,E)$ is as in the statements of these theorems, $\alpha\in\mathbb{Z}_{\geq 1}$, and $Q^{\alpha f}$ has set of vertices $I\sqcup\{\infty\}$ and its set of edges is $E$ together with $\alpha$ edges from $\infty$ to any vertex of $I$. For future reference, we state the version of Corollary \ref{thmsodD} for the space of semistable representations $\X^{\alpha f}(d)^{\text{ss}}$ of the quiver $Q^{\alpha f}$:

\begin{cor}\label{thmsodE}
Let $Q$ be a quiver satisfying Assumption \ref{assum1}.
    Let $\mu\in \mathbb{R}$ such that $\mu\sigma_d$ is a good weight and let $\alpha\in\mathbb{N}$. Then there is a $X(d)$-linear semiorthogonal decomposition
    \[D^b\left(\X^{\alpha f}(d)^{\text{ss}}\right)=\left\langle \bigotimes_{i=1}^k \mathbb{M}(d_i)_{v_i} :
    \mu\leq \frac{v_1}{\dd_1}<\cdots<\frac{v_k}{\dd_k}<\alpha+\mu
    \right\rangle,\] where the rights hand side is after all partitions $(d_i)_{i=1}^k$ of $d$ and 
    integers 
    $(v_i)_{i=1}^k\in\mathbb{Z}^k$ satisfying the above inequality.
    %The weights $\theta_i\in M(d_i)$ are defined in \eqref{def:deltai}.
\end{cor}

\subsection{Semiorthogonal decompositions for general potentials}\label{subsection:moresod}

Let $W$ be a potential of $Q$.
By \cite[Proposition 2.5]{PTzero}, there are analogous semiorthogonal decompositions to those in Theorems \ref{thmsodC} and \ref{sodfullstackB} (and also to those in Corollaries \ref{thmsodD} and \ref{thmsodE}) for categories of matrix factorizations. 
Recall the definition of (graded or not) quasi-BPS categories from \eqref{defsdw} and \eqref{defsdwgr}.
We first state the version for Corollary \ref{thmsodD}, which we use in \cite{PTtop}:

\begin{thm}\label{sodfullstackBW0}
Let $Q$ be a quiver satisfying Assumption \ref{assum1} and let $W$ be a potential of $Q$ (and possibly a grading as in Section \ref{polycat}). Let $\mu \in \mathbb{R}$ such that $\mu\sigma_d$ is a
good weight and let $\alpha \in \mathbb{Z}_{\geq 1}$. There is a semiorthogonal 
decomposition 
    \[
\mathrm{MF}^{\bullet}\left(\X^{\alpha f}(d)^{\text{ss}}, \mathrm{Tr}\,W\right)=\left\langle \bigotimes_{i=1}^k \mathbb{S}^{\bullet}(d_i)_{v_i}: 
\mu \leq \frac{v_1}{\dd_1}<\ldots<\frac{v_k}{\dd_k}<\mu+\alpha \right\rangle,\]
   where the right hand side is
   after all partitions $(d_i)_{i=1}^k$ of $d$ 
   and integers $(v_i)_{i=1}^k \in \mathbb{Z}^k$ satisfying the above inequality, 
   and where $\bullet\in\{\emptyset, \mathrm{gr}\}$.
\end{thm}

A version for Theorem~\ref{sodfullstackB} (for quivers satisfying Assumption \ref{assum1}) is the following: 
\begin{thm}\label{sodfullstackBW}
Let $Q$ be a quiver satisfying Assumption \ref{assum1} and let $W$ be a potential of $Q$ (and possibly consider a grading as in Section \ref{polycat}). %Let $\delta_d\in M(d)^{W_d}_\mathbb{R}$. %and let $\mu\in \mathbb{R}$ such that $\delta_d+\mu\sigma_d$ is generic. 
There is a semiorthogonal decomposition
    \[
\mathrm{MF}^{\bullet}\left(\X(d), \mathrm{Tr}\,W\right)=\left\langle \bigotimes_{i=1}^k \mathbb{S}^{\bullet}(d_i; v_i\tau_{d_i}): 
\frac{v_1}{\dd_1}<\ldots<\frac{v_k}{\dd_k}\right\rangle,\]
   where the right hand side is
   after all partitions $(d_i)_{i=1}^k$ of $d$ 
   and integers $(v_i)_{i=1}^k \in \mathbb{Z}^k$ satisfying the above inequality, 
   and where $\bullet\in\{\emptyset, \mathrm{gr}\}$.
\end{thm}

%We have the following corollaries:

%\begin{cor}\label{thmsodDW}
 %   Let $Q^\circ$ be a quiver satisfying Assumption~\ref{assum1} and let $(Q,W)$ be its tripled quiver with potential. There is a semiorthogonal decomposition
 %   \begin{equation}\notag
%\mathrm{MF}^{\bullet}\left(\X(d), \mathrm{Tr}\,W\right)=\left\langle \bigotimes_{i=1}^k \mathbb{S}^{\bullet}\left(d_i; \theta_i+v_i\tau_{d_i}\right) : 
%\frac{v_1}{\dd_1}<\ldots<\frac{v_k}{\dd_k}
%%\right\rangle,
%\end{equation}
 %  where the right hand side is after all partitions $(d_i)_{i=1}^k$ of $d$ and integers $(v_i)_{i=1}^k\in\mathbb{Z}^k$. 
 %  The weights $\theta_i\in M(d_i)$ are defined in %\eqref{def:deltai}.
%\end{cor}

%\begin{remark}
%In the above corollary, note that $\theta_i\in %M(d_i)^{W_{d_i}}$, and thus the 
%semiorthogonal decomposition (\ref{SODprop216}) 
%is simply written as 
%   \begin{equation}\label{SODprop216}
%\mathrm{MF}^{\bullet}\left(\X(d), \mathrm{Tr}\,W\right)=\left\langle \bigotimes_{i=1}^k \mathbb{S}^{\bullet}\left(d_i; v_i\tau_{d_i}\right) : 
%\frac{v_1}{\dd_1}<\ldots<\frac{v_k}{\dd_k}
%\right\rangle,
%\end{equation}
%However the description (\ref{SODprop216}) is 
%often useful as it keeps track of the fact 
%that each summand is the image of the categorical 
%Hall product. 
%The semiorthogonal decomposition \eqref{SODprop216} is particularly convenient because all the weights $\delta,\delta_i$ appearing in \eqref{SODprop216} are multiples of $\tau_d, \tau_{d_i}$. %We will use Corollary \ref{thmsodDW} in Section \ref{s4} to prove the categorical support lemma, see Theorem \ref{lem:support}.
%\end{remark}

In the case of doubled quiver, by combining 
Theorems~\ref{sodfullstackB} and~\ref{sodfullstackBW} with the Koszul equivalence
in Theorem~\ref{thm:Koszul}, and the compatibility 
of Koszul equivalence with categorical 
Hall product in Proposition~\ref{prop:compati:hall}, 
we obtain the following:
\begin{thm}\label{cor:double}
Let $Q^{\circ}$ be a quiver and let $(Q^{\circ, d}, \mathscr{I})$ be its doubled 
quiver relation. For a partition $(d_i)_{i=1}^k$ of $d$, let $\lambda$ be an associated antidominant cocharacter, and define $\theta_i \in \frac{1}{2}M(d_i)^{W_{d_i}}$
  by:
   \[
    \sum_{i=1}^k \theta_i=-\frac{1}{2}\overline{R}(d)^{\lambda>0}+\mathfrak{g}(d)^{\lambda>0}.
   \] 
There is a semiorthogonal decomposition: 
\begin{align}\label{sod:Pd}
 D^b(\mathscr{P}(d))=\left\langle 
 \bigotimes_{i=1}^k \mathbb{T}(d_i, \theta_i+
 v_i \tau_{d_i})\right\rangle
\end{align}
  where the right hand side is after all partitions $(d_i)_{i=1}^k$ of $d$ and tuplets $(v_i)_{i=1}^k\in\mathbb{R}^k$ such that the sum of coefficients of $\theta_i+v_i\tau_{d_i}$ is an integer for all $1\leq i\leq k$ and satisfying the inequality:
  \begin{equation}\label{ineuality}
  \frac{v_1}{\dd_1}<\ldots<\frac{v_k}{\dd_k}.
   \end{equation}
   Moreover, each summand is given by 
   the image of categorical Hall product (\ref{cathall:Pd}). 

   If furthermore $Q^{\circ}$ satisfies Assumption~\ref{assum1}, then $v_i \in \mathbb{Z}$
   and $\theta_i \in M(d_i)^{W_{d_i}}$, so the right hand side of \eqref{sod:Pd} is after all partitions $(d_i)_{i=1}^k$ of $d$ and all tuplets $(v_i)_{i=1}^k\in\mathbb{Z}^k$ satisfying \eqref{ineuality}.
  \end{thm}
The following example will be used in~\cite{PTK3}: 
\begin{example}\label{exam:gloop}
Let $Q^{\circ}$ be the quiver with 
one vertex and $g\geq 1$ loops. 
Then $d \in \mathbb{N}$ and 
the semiorthogonal decomposition (\ref{sod:Pd}) 
is 
\begin{align*}
D^b(\mathscr{P}(d))=\left\langle 
 \bigotimes_{i=1}^k \mathbb{T}(d_i)_{w_i} :\frac{v_1}{d_1}<\ldots<\frac{v_k}{d_k} \right\rangle.
\end{align*}
Here, $w_i\in\mathbb{Z}$ for $1\leq i\leq k$ is given by 
\begin{align*}
w_i=v_i+(g-1)d_i\left(\sum_{i>j}d_j-\sum_{i<j}d_j  \right). 
\end{align*}    Note that $\mathbb{T}(d_i)_{w_i}\cong \mathbb{T}(d_i)_{v_i}$ for all $1\leq i\leq k$.
\end{example}

\subsection{Strong generation of quasi-BPS categories}
We use Theorem \ref{thmsodC} to prove the strong generation of the (graded or not) quasi-BPS 
categories  \[\mathbb{S}^\bullet(d; \delta_d)\text{ for }\bullet\in\{\emptyset, \mathrm{gr}\},\] where the grading is as in Subsection \ref{notation}. 
We first recall some terminology. 

Let $\mathcal{D}$ be a pre-triangulated dg-category. 
For a set of objects $\mathcal{S} \subset \mathcal{D}$, 
we denote by $\langle \mathcal{S} \rangle$
the smallest subcategory which contains $S$ 
and closed under shifts, the finite direct sums and 
direct summands. 
For subcategories $\mathcal{C}_1, \mathcal{C}_2 \subset 
\mathcal{D}$, we denote by 
$\mathcal{C}_1 \star \mathcal{C}_2 \subset \mathcal{D}$
the smallest subcategory which contains 
objects $E$ which fit into distinguished triangles
$A_1 \to E \to A_2\to A_1[1]$ for $A_i \in \mathcal{C}_i$
and closed under shifts, finite direct sums, and direct 
summands. 

We say that $\mathcal{D}$ is \textit{strongly generated by 
$C \in \mathcal{D}$} if
$\mathcal{D}=\langle C \rangle^{\star n}$
for some $n\geq 1$. 
A dg-category $\mathcal{D}$ is called \textit{regular} 
if it 
has a strong generator. 
It is called \textit{smooth} if the diagonal dg-module of $\mathcal{D}$ is perfect. It is proved in~\cite[Lemma~3.5, 3.6]{VL}
that if $\mathcal{D}$ is smooth, then $\mathcal{D}$ is regular. 

\begin{prop}\label{thm:stronggenerator}
Let $Q$ be a symmetric quiver such that the number of loops at each vertex $i\in I$ has the same parity, let $W$ be any potential of $Q$, let $d\in\mathbb{N}^I$, and let $\delta_d\in M(d)_\mathbb{R}^{W_d}$. 
    The category $\mathbb{S}^\bullet(d; \delta_d)$ has a strong generator, thus it is regular. 
\end{prop}

\begin{proof}
    The category $\mathbb{S}^\bullet(d; \delta_d)$ is admissible in $\mathrm{MF}^\bullet\left(\X^f(d),\mathrm{Tr}\,W\right)$ by the variant for an arbitrary potential (see \cite[Proposition 2.5]{PTzero}) of Theorem \ref{thmsodC}. Let \[\Phi\colon \mathrm{MF}^\bullet\left(\X^f(d),\mathrm{Tr}\,W\right)\to \mathbb{S}^\bullet(d; \delta_d)\] be the adjoint of the inclusion. The category $\mathrm{MF}^\bullet\left(\X^f(d),\mathrm{Tr}\,W\right)$ is smooth, see~\cite[Lemma~2.11]{FavTy}. Therefore it is regular, so it has a strong generator $C$. 
    Then $\mathbb{S}^\bullet(d; \delta_d)$ has the strong generator $\Phi(C)$.
\end{proof}

\begin{remark}
    We do not discuss smoothness of the category $\mathbb{S}^\bullet(d; \delta_d)$ in this paper. Note that smoothness is equivalent to regularity if the category is proper \cite[Theorem 3.18]{Orsmooth}, but this is not the case for the categories $\mathbb{S}^\bullet(d; \delta_d)$.
However it should be possible to prove the smoothness of $\mathbb{S}^{\bullet}(d; \delta_d)$ using non-commutative matrix factorizations 
as in~\cite[Proposition~3.29]{PThiggs}. The details may appear elsewhere. 
\end{remark}

\section{Quasi-BPS categories for tripled quivers}\label{s4}

In this section, we prove a categorical analogue of Davison's support lemma \cite[Lemma 4.1]{Dav} for tripled quivers with potential of quivers $Q^\circ$ satisfying Assumption \ref{assum1}, see Theorem \ref{lem:support}. We then use Theorem \ref{lem:support} to construct reduced quasi-BPS categories $\mathbb{T}$ for preprojective algebras, which are proper over the good moduli space $P$ of representations of the preprojective algebra, and regular. When the reduced stack of representations of the preprojective algebra is classical, we show that the relative Serre functor of $\mathbb{T}$ over $P$ is trivial, and further that the category $\mathbb{T}$ is indecomposable.

Throughout this section, we consider tripled quivers with potential or preprojective algebra for quivers $Q^\circ$ satisfying Assumption \ref{assum1}.

\subsection{The categorical support lemma}
%\textcolor{red}{(Unfortunately at this moment, I have to impose the assumption (\ref{cond:ab}) for the categorical support lemma, and can only give the result for $\mathbb{S}(d; v\tau_d)$ with $(v, \dd)$ coprime. I will think more.)}%\textcolor{blue}{ (Is it possible to replace $v\tau_d$ with a different weight with sum of coefficients $v$?)}
%\textcolor{red}{(The issue is that, if $\delta$ is a weight with total weight $v$ and $\delta=\delta'+\delta''$ is the decomposition corresponding to the decomposition $V^a=V^{'a} \oplus V^{''a}$, it is not necessary true that $\langle 1_d, \delta\rangle/\dd=\langle 1_{d'}, \delta'/\dd'=\langle 1_{d''}, \delta'' \rangle/\dd''$. I think this is only true for $\delta=v\tau_d$.)}
Let $Q^{\circ}=(I, E^{\circ})$ be a quiver satisfying Assumption \ref{assum1}, and consider 
its tripled quiver $Q=(I, E)$ with potential $W$, see Subsection~\ref{subsec22}. 
We will use the notations from Subsection \ref{subsec22}.
Recall that 
\begin{align*}
\X(d)=R(d)/G(d), \ 
	R(d)=\overline{R}(d) \oplus \mathfrak{g}(d),
\end{align*}
where $\overline{R}(d)$ is the representation space of the doubled 
quiver of $Q^{\circ}$. 
There is thus a projection map onto the second summand (which records the linear maps corresponding to the loops in the tripled quiver not in the doubled quiver):
\[\tau\colon \X(d)\to \mathfrak{g}(d)/G(d).\]
We consider the good moduli space morphism 
\begin{align}\label{map:eigen}
	\pi_{\mathfrak{g}}\colon \mathfrak{g}(d)/G(d) \to \mathfrak{g}(d)\ssslash G(d)=\prod_{a \in I}\mathrm{Sym}^{d^a}(\mathbb{C}). 
\end{align}
The above map sends $z \in \mathfrak{g}(d)=\bigoplus_{a\in I}\mathrm{End}(V^a)$ to its generalized eigenvalues.  
Let $\Delta$ be the diagonal 
\begin{align*}
	\Delta \colon \mathbb{C} \hookrightarrow \prod_{a\in I}\mathrm{Sym}^{d^a}(\mathbb{C}), \ x\mapsto 
	\prod_{a\in I} (\overbrace{x,\ldots, x}^{d^a})=(\overbrace{x,\ldots, x}^{\dd}). 
\end{align*}
Consider the composition 
\begin{align}\label{compose:pi}
	\pi \colon \mathrm{Crit}(\Tr W) \hookrightarrow \X(d) 
	\xrightarrow{\tau} \mathfrak{g}(d)/G(d)
	\xrightarrow{\pi_{\mathfrak{g}}} \mathfrak{g}(d)\ssslash G(d). 
\end{align}
%where $\X(d) \to \mathfrak{g}(d)/G(d)$ is the projection  and the last map is given by (\ref{map:eigen}). 
The following is the main result of this section, and a generalization of \cite[Theorem~3.1]{PT1} which discusses the case of the tripled quiver with potential of the Jordan quiver.

\begin{thm}\label{lem:support}
	Let 
	$v  \in \mathbb{Z}$ such that $\gcd(v, \dd)=1$.
	Then any object in $\mathbb{S}(d)_v$ is supported on 
	$\pi^{-1}(\Delta)$. 
\end{thm}
Before the proof of Theorem~\ref{lem:support}, we 
introduce some notations related to formal completions of fibers over $\mathfrak{g}(d)\sslash G(d)$.
For $p \in \mathfrak{g}(d)\ssslash G(d)$, 
we denote by 
$\X_p(d)$ the pull-back of 
the morphism 
\[\pi_{\mathfrak{g}} \circ \tau \colon 
\X(d) \to \mathfrak{g}(d)\ssslash G(d)\]
by $\Spec \widehat{\mathcal{O}}_{\mathfrak{g}(d)\ssslash G(d), p} \to 
\mathfrak{g}(d)\ssslash G(d)$. 
We write
an element $p \in \prod_{a\in I}\mathrm{Sym}^{d^a}(\mathbb{C})$ as
\begin{align}\label{psum}
 p=\left(\sum_{j=1}^{l^a} d^{a, (j)} x^{a, (j)}\right)_{a\in I}
\end{align}
where $x^{a, (j)} \in \mathbb{C}$
with $x^{a, (j)} \neq x^{a, (j')}$ for $1\leq j\neq j'\leq l^a$ and 
$d^{a, (j)} \in \mathbb{Z}_{\geq 1}$ for $1\leq j\leq l^a$ are such that $\sum_{j=1}^{l^a}d^{a,(j)}=d$.
There is an isomorphism 
\begin{align}\label{isom:Xp}
	\X_p(d)\cong\left(\overline{R}(d) \times \prod_{a, j} \widehat{\mathfrak{g}}^{a, (j)}\right)/G_p
\end{align}
where 
$V^a=\bigoplus_{j} V^{a, (j)}$ is the decomposition into generalized eigenspaces corresponding to $p$, $G_p:=\prod_{a, j} GL(V^{a, (j)})$, and
$\mathfrak{g}^{a, (j)}:=\mathrm{End}(V^{a, (j)})$. 
\begin{remark}\label{rmk:Qp}
For a point $p$ as in (\ref{psum}), 
let $J$ be the set of pairs $(a, j)$ such that $a \in I$ and $1\leq j \leq l^a$. 
The support of $p$ 
is defined to be 
\[\mathrm{supp}(p):=\{x^{a,(j)} \mid (a, j) \in J \} \subset \mathbb{C}.\]
Let $Q_p^{\circ}$ be a quiver with vertex set $J$ 
and the number of arrows from $(a, j)$ to $(b, j')$ is the
number of arrows from $a$ to $b$ in $Q^{\circ}$. 
Since we have  
\begin{align}\label{RJ}
    &\overline{R}(d) \oplus \bigoplus_{(a, j)\in J}\mathfrak{g}^{a, (j)} \\
   \notag &=\bigoplus_{(a\to b) \in E^{\circ}, j, j'}
\Hom(V^{a, (j)}, V^{b, (j')}) \oplus 
\Hom(V^{b, (j')}, V^{a, (j)}) \oplus 
\bigoplus_{(a, j)\in J}\mathrm{End}(V^{a, (j)}),     
    \end{align}
the space (\ref{RJ}) is the representation space of 
the tripled quiver 
$Q_p$ of $Q_p^{\circ}$
with dimension vector $d=\left(d^{a, (j)}\right)_{(a, j)\in J}$.
Note that $Q_p^{\circ}$ satisfies Assumption~\ref{assum1} 
since $Q^{\circ}$ satisfies Assumption~\ref{assum1}. 
%Let $\mathscr{Y}_p(d)$ be the stack of representations of dimension $d$ of $Q_p$.
%The quiver $Q_p$ has an even number of edges between two different vertices and an odd number of loops at each vertex. 
%Indeed, it suffices to check that the codimension of the map 
%\[\mathscr{Y}_p(d)^{\lambda\geq 0}\to \mathscr{Y}_p(d)\] is even for $\lambda$ a cocharacter of $T(d)$. This is true because the codimension %of the map
%\[\X(d)^{\lambda\geq 0}\to \X(d)\] is even for any cocharacter $\lambda$ of $T(d)$, which in turn is true because $Q$ has an even number of edges between two different vertices and an odd number of loops at each vertex.
There is a correspondence from dimension vectors of $Q_p$ to dimension vectors of $Q$:
\begin{align*}
	\left(d^{a, (j)}\right)_{(a, j)\in J} \mapsto \left(d^{a}=\sum_{j}d^{a, (j)}  \right)_{a\in I}.
	\end{align*}
\end{remark}

\begin{proof}[Proof of Theorem \ref{lem:support}]
	The proof is similar to the proof of  
	~\cite[Theorem~3.1]{PT1}, but simpler by the use of Proposition~\ref{prop:period2}. 
	Consider an object $\mathcal{F} \in \mathbb{S}(d)_v$. 
	Let $p\in \prod_{a\in I}\mathrm{Sym}^{d^a}(\mathbb{C})$ be a point which decomposes as $p=p'+p''$ such that 
	the supports of $p'$ and $p''$ are disjoint. Write $p$ as in \eqref{psum}.
 Assume the support of $\mathcal{F}$
	intersects $\pi^{-1}(p)$. We will reach a contradiction with the assumption $\gcd(v, \dd)=1$. 
	Define 
	\begin{align}\label{subcat:Sp}
		\mathbb{S}_p(d)_v \subset \mathrm{MF}\left(\X_p(d), \Tr W\right)
	\end{align}
	to be the full subcategory generated by matrix factorizations 
	whose factors are direct summands of the vector bundles 
 $\mathcal{O}_{\X_p(d)} \otimes \Gamma_{G_p}(\chi)$, 
	where $\chi$ is a $G_p$-dominant $T(d)$-weight
	satisfying 
	\begin{align}\label{wei:chip}
		\chi+\rho_p -v\tau_d\in \mathbf{W}_p(d):=\frac{1}{2}\mathrm{sum}[0, \beta]. 
	\end{align}
	Here, $\rho_p$ is half the sum of positive roots of $G_p$
	and the Minkowski sum for $\mathbf{W}_p(d)$ 
 is after all weights $\beta$ of the $T(d)$-representation \eqref{RJ}. %$\overline{R}(d) \oplus \bigoplus_{i, a} \mathfrak{g}^{a, (i)}$. 
 Define $n_{\lambda, p}$ by 
     \begin{align}\label{identity:n}
         n_{\lambda, p}:=\left\langle \lambda, \det\left(\mathbb{L}_{\X_p(d)}^{\lambda>0}|_{0}\right)\right\rangle=\left\langle \lambda, \det\left(\mathbb{L}_{\X(d)}^{\lambda>0}|_{0}\right)\right\rangle. 
     \end{align}
 As in Lemma~\ref{lemma:alt}, 
 the subcategory (\ref{subcat:Sp})
 is generated by matrix factorizations
 whose factors are
 of the form $\Gamma \otimes \mathcal{O}_{\X_p(d)}$, where $\Gamma$ is a $G_p$-representation
 such that any $T(d)$-weight $\chi$ of $\Gamma$
 satisfies 
 \begin{align*}
\langle \lambda, \chi-v\tau_d \rangle
\in \left[-\frac{1}{2} n_{\lambda, p}, \frac{1}{2} n_{\lambda, p}   \right]. 
     \end{align*}
     In particular by the identity (\ref{identity:n}), the restriction along 
     with the natural morphism $\iota_p \colon \X_p(d) \to \X(d)$
     restricts to the functor 
     \begin{align*}
     \iota_p^{\ast} \colon 
         \mathbb{S}(d)_v \to \mathbb{S}_p(d)_v. 
     \end{align*}
     Therefore, by the assumption that the support of 
	$\mathcal{F}$ intersect $\pi^{-1}(p)$, we have $0\neq \iota_p^{\ast}\mathcal{F} \in \mathbb{S}_p(d)_v$, and
	in particular $\mathbb{S}_p(d)_v \neq 0$. We show that, in this case, $v$ is not coprime with $\dd$.
	
	The decomposition $p=p'+p''$ corresponds to 
	decompositions $V^{a}=V'^{a} \oplus V''^{a}$
	\begin{align*}
		V'^{a}=\bigoplus_{x^{a, (j)} \in \mathrm{supp}(p')} V^{a, (j)}, \ 
		V''^{a}=\bigoplus_{x^{a, (j)} \in \mathrm{supp}(p'')} V^{a, (j)} 
	\end{align*} 
 for all $a\in I$.
	Let $d'^{a}=\dim V'^{a}$, $d''^{a}=\dim V''^{a}$, 
	$d'=(d'^{a})_{a\in I}$ and $d''=(d''^{a})_{a\in I}$, 
	so $d=d'+d''$. 
	By Lemma~\ref{sublem:1.5}, 
 after possibly replacing the isomorphism (\ref{isom:Xp}),
	the regular function $\Tr W$ restricted to $\X_p(d)$ is written as 
	\begin{align}\label{W:Xpd}
		\Tr W|_{\X_p(d)}= \Tr W'\boxplus 
		\Tr W'\boxplus f.
	\end{align}
	Here, $\Tr W'$ and $\Tr W''$ are the regular functions given by 
	$\Tr W$ on $\X(d')$ and $\X(d'')$, respectively,
	restricted to $\X_{p'}(d')$ and $\X_{p''}(d'')$,
	and 
	$f$ is a non-degenerate $G_p$-invariant 
	quadratic form on $U\oplus U^{\vee}$
	given by $f(u, v)=\langle u, v \rangle$, where $U$ is the $G_p$-representation 
	\begin{align*}
		U:=\bigoplus_{(a\to b) \in E^{\circ}}\Hom(V'^{a}, V''^{b}) \oplus
		\bigoplus_{(b\to a) \in E^{\circ}}\Hom(V'^{a}, V''^{b}). 
	\end{align*}
	Note that we have the decomposition as $G_p$-representations
	\begin{align}\label{decom:R}
		\overline{R}(d)=\overline{R}(d') \oplus \overline{R}(d'') \oplus U \oplus U^{\vee}. 
	\end{align}
	
	We have the following diagram 
	\begin{align*}
		\xymatrix{
			\mathcal{U} \ar@<-0.3ex>@{^{(}->}[r]^-{i} \ar[d]_-{q} 
			& \mathcal{U} \oplus \mathcal{U}^{\vee} & \ar[l]_-{j} \X_p(d) \\
			\X_{p'}(d') \times \X_{p''}(d''), & &	
		}
	\end{align*}
	where $\mathcal{U}$ is the vector bundle on $\X_{p'}(d') \times \X_{p''}(d'')$
	determined by the $G_p$-representation $U$, $i$ is the closed immersion $x \mapsto (x, 0)$, 
	and $j$ is the natural morphism 
	induced by the formal completion which induces the isomorphism on critical loci of $\Tr W$. 
	Consider the functor
	\begin{align}\label{Kn:eq}
		\Psi := j^{\ast}i_{\ast}q^{\ast} \colon 
		\mathrm{MF}(\X_{p'}(d'), \Tr W') &\boxtimes 
		\mathrm{MF}(\X_{p''}(d''), \Tr W'') \\
		\notag &\stackrel{\sim}{\to}\mathrm{MF}(\mathcal{U} \oplus \mathcal{U}^{\vee}, \Tr W) 
		\stackrel{j^{\ast}}{\hookrightarrow}
		\mathrm{MF}(\X_p(d), \Tr W), 
	\end{align}
	where the first arrow is an equivalence 
	by Kn\"orrer periodicity (\ref{theta:period}) 
	and the second arrow is fully-faithful with dense image, see~\cite[Lemma~6.4]{T3}. 
 Let $v', v''\in\mathbb{Q}$ be such that
	\begin{equation}\label{vdvprime}
	v\tau_d=v'\tau_{d'}+v''\tau_{d''}, \end{equation} and let $\delta'\in M(d')_\mathbb{R}$ and $\delta''\in M(d'')_\mathbb{R}$ be
	such that 
	\begin{align}\label{delta:U}
		\delta' +\delta''=\frac{1}{2}U. 
		\end{align} 
  The quiver $Q^\circ$ satisfies Assumption \ref{assum1}, and thus $\delta' \in M(d')^{W_{d'}}$ and $\delta'' \in M(d'')^{W_{d''}}$.  
	By Proposition~\ref{prop:period2}, 
 the functor \eqref{Kn:eq}
	restricts to the fully-faithful functor
 with dense image:
	\begin{align}\label{rest:M}
		\mathbb{S}_{p'}(d'; \delta'+v'\tau_{d'})\boxtimes \mathbb{S}_{p''}(d''; \delta''+v''\tau_{d''})
		\to \mathbb{S}_p(d)_v.
	\end{align}
	%Below we show that the above functor (\ref{rest:M}) has dense image. 
	Then we have 
	\begin{align*}
	    &\mathbb{S}_{p'}(d')_{v'}\simeq \mathbb{S}_{p'}(d'; \delta'+v' \tau_{d'}) \neq 0, \\
     &\mathbb{S}_{p''}(d'')_{v''}\simeq 
 \mathbb{S}_{p''}(d''; \delta'' +v'' \tau_{d''}) \neq 0.
 \end{align*}
	In particular, we have  
	$v' \in \mathbb{Z}$, $v'' \in \mathbb{Z}$ by Remark~\ref{rmk:Mweight}.
	By \eqref{vdvprime}, we further have that  
	\begin{align*}
		\frac{v'}{\dd'}
		=\frac{v''}{\dd''}=\frac{v}{\dd},
		\end{align*} which contradicts the assumption that $\gcd(v, \dd)=1$. 
	\end{proof}

We have postponed the proof of the following:

\begin{lemma}\label{sublem:1.5}
	By replacing the isomorphism (\ref{isom:Xp}) if necessary, 
	the identity (\ref{W:Xpd}) holds. 
\end{lemma}
\begin{proof}
	(cf.~\cite[Lemma~3.3]{PT1})
	For $p \in \prod_{a \in I}\mathrm{Sym}^{d^a}(\mathbb{C})$ as in (\ref{psum}), let 
	$u \in \mathfrak{g}(d)/G(d)$ be the unique closed point 
	over $p$. Note that $u$ is represented by a diagonal matrix 
	with eigenvalues $x^{a, (j)}$. 
	In particular, we can assume that 
	\begin{align*}
		u \in \bigoplus_{(a, j)\in J}\mathfrak{g}^{a, (j)}
		\subset \mathfrak{g}(d). 
		\end{align*}
	We construct 
	a map 
	\begin{align*}
		\nu \colon \left(\overline{R}(d) \oplus \bigoplus_{(a, j)\in J}\mathfrak{g}^{a, (j)}\right)/G_p \to \X(d)
		\end{align*}
	given by $(\alpha, \beta=\beta^{a, (j)}) \mapsto (\alpha, \beta+u)$. 
		The morphism $\nu$ is \'{e}tale at $\nu(0)$. 
		Indeed, the tangent complex of $\X(d)$
	at $\nu(0)$ is 
	\begin{align*}
		\mathbb{T}_{\X(d)}|_{\nu(0)}=
		\big(\mathfrak{g}(d) \to \overline{R}(d) \oplus \mathfrak{g}(d)
		   \big), \ 
		\gamma \mapsto (0, [\gamma, u]).
	\end{align*}
	The kernel of the above map is 
	$\bigoplus_{(a, j)\in J}\mathfrak{g}^{a, (j)}$ and the cokernel 
	is $\overline{R}(d) \oplus \bigoplus_{(a, j)\in J}\mathfrak{g}^{a, (j)}$, 
	so the morphism $\nu$ induces a quasi-isomorphism on tangent 
	complexes at $\nu(0)$. 
	
	For $x \in \overline{R}(d)$, 
	 a vertex $a \in I$ and and edge $e=(a \to b)$ in $E^{\circ}$, write its corresponding maps as
	$x(e) \colon V^a \to V^b$ 
	and $x(\overline{e}) \colon V^b \to V^a$. For $\theta \in \mathfrak{g}(d)$, write its corresponding map as $\theta(a) \colon V^a \to V^a$. 
	The function $\Tr W$ is given by 
	\begin{align*}
		\Tr W(x, \theta)=\Tr \left(\sum_{e \in E^{\circ}}[x(e), x(\overline{e})]
		   \right)\left(\sum_{a\in I} \theta(a)\right). 
		\end{align*}
	We set $\mathfrak{g}'$ and $\mathfrak{g}''$ to be 
\begin{align*}\mathfrak{g}'=\bigoplus_{x^{a, (j)} \in \mathrm{supp}(p')}
\mathfrak{g}^{a, (j)}, \ 
\mathfrak{g}''=\bigoplus_{x^{a, (j)} \in \mathrm{supp}(p'')}
\mathfrak{g}^{a, (j)},
	\end{align*}
and write 
an element $\gamma \in \bigoplus_{(a, j)\in J}\mathfrak{g}^{a, (j)}$ as 
$\gamma=\gamma'+\gamma''$ for $\gamma' \in \mathfrak{g}'$, $\gamma'' \in \mathfrak{g}''$. 
Note that there are isomorphisms 
\begin{align*}
	\X_{p'}(d')\cong\left(\overline{R}(d') \times \widehat{\mathfrak{g}}'\right)/G_{p'}, \ 	\X_{p''}(d'')\cong\left(\overline{R}(d'') \times \widehat{\mathfrak{g}}''\right)/G_{p''}.
	\end{align*}
For $x \in \overline{R}(d)$, we write 
\begin{align*}
	x=x'+x''+x_U+x_{U^{\vee}}
	\end{align*}	
for $x' \in \overline{R}(d')$, $x'' \in \overline{R}(d'')$, 
$x_U \in U$ and $x_{U^{\vee}} \in U^{\vee}$. 
	Then $\nu^{\ast} \Tr W$ is calculated as 
	%\begin{align*}
	%	\nu^{\ast}\Tr W(\rho, \gamma)&=
	%	\Tr \left(\sum_{e \in E^{\circ}}[\rho'(e), \rho'(\overline{e})]
	%	\right)\left(\sum_{a\in I} \gamma'(a)\right) \\
	%	&+\Tr \left(\sum_{e \in E^{\circ}}[\rho''(e), \rho''(\overline{e})]
	%%	&+\Tr \left(\sum_{e\in E^{\circ}} \rho_{U^{\vee}}(\overline{e})\left(\rho_{U}(e)\gamma'+\rho_U(e)u'-%\gamma''\rho_U(e)-u''\rho_U(e)\right)
	%	\right) \\
	%	&+\Tr \left(\sum_{e\in E^{\circ}} \rho_{U}(e)\left(\rho_{U^{\vee}}(\overline{e})\gamma''+\rho_{U^{\vee}}%(\overline{e})u''-
	%	\gamma'\rho_{U^{\vee}}(\overline{e})-u'\rho_U(\overline{e})\right)
	%	\right)
	%	\end{align*}
  %\textcolor{blue}{Should the above be the following? (only last two lines changed)
  \begin{align*}
		\nu^{\ast}\Tr W(x, \gamma)&=
		\Tr \left(\sum_{e \in E^{\circ}}[x'(e), x'(\overline{e})]
		\right)\left(\sum_{a\in I} \gamma'(a)\right) \\
		&+\Tr \left(\sum_{e \in E^{\circ}}[x''(e), x''(\overline{e})]
		\right)\left(\sum_{a\in I} \gamma''(a)\right) \\
		&+\Tr \left(\sum_{e\in E^{\circ}} x_{U^{\vee}}(e)\left(x_{U}(\overline{e})\gamma'+x_U(\overline{e})u'-\gamma''x_U(\overline{e})-u''x_{U}(\overline{e})\right)
		\right) \\
		&+\Tr \left(\sum_{e\in E^{\circ}} x_{U}(e)\left(x_{U^{\vee}}(\overline{e})\gamma''+x_{U^{\vee}}(\overline{e})u''-
		\gamma'x_{U^{\vee}}(\overline{e})-u'x_{U^\vee}(\overline{e})\right)
		\right).
		\end{align*}
	We take the following $G_p$-equivariant variable change
	%\begin{align*}
	%&\rho_U(e) \mapsto \rho_{U}(e)\gamma'+\rho_U(e)u'-\gamma''\rho_U(e)-u''\rho_U(e), \\
	%&\rho_{U^{\vee}}(\overline{e}) \mapsto 	
	%\rho_{U^{\vee}}(\overline{e})\gamma''+\rho_{U^{\vee}}(\overline{e})u''-
	%\gamma'\rho_{U^{\vee}}(\overline{e})-u'\rho_U(\overline{e}).
	%\end{align*}
 %\textcolor{blue}{Should we have instead the following?
 \begin{align*}
	&x_U(\overline{e}) \mapsto x_{U}(\overline{e})\gamma'+x_U(\overline{e})u'-\gamma''x_U(\overline{e})-u''x_{U}(\overline{e}), \\
	&x_{U^{\vee}}(\overline{e}) \mapsto 	
	x_{U^{\vee}}(\overline{e})\gamma''+x_{U^{\vee}}(\overline{e})u''-
	\gamma'x_{U^{\vee}}(\overline{e})-u'x_{U^\vee}(\overline{e}).
	\end{align*}
Since $u'$ and $u''$ are diagonal matrices with different eigenvalues, 
the above variable change determines an automorphism of $\X_p(d)$. 
Therefore 
	we obtain the desired identity (\ref{W:Xpd}). 
\end{proof}

\subsection{Reduced stacks}\label{subsec52}
We use the notations introduced in Subsections~\ref{subsec22}, \ref{subsub:triple}. 
Let 
$\mathfrak{g}(d)_0 \subset \mathfrak{g}(d)$ 
be the traceless subalgebra, i.e. the kernel of the map 
\begin{align*}
\mathfrak{g}(d)=\bigoplus_{a\in I} \Hom(V^a, V^a) \to \mathbb{C}, \ 
(g^a)_{a \in I} \mapsto \sum_{a\in I} \Tr(g^a). 
\end{align*}
The moment map $\mu$ in (\ref{def:moment}) 
factors though the map 
\begin{align*}
\mu_0 \colon \overline{R}(d) \to \mathfrak{g}(d)_0.
\end{align*}
We define the following reduced derived stack:
\begin{align}\label{def:Pd}
    \mathscr{P}(d)^{\rm{red}}:=\mu_0^{-1}(d)/G(d). 
\end{align}
Note that we have 
the commutative diagram 
\begin{align*}
    \xymatrix{
    \mathscr{P}(d)^{\rm{red, cl}} \ar@{=}[r] \ar[d]_-{\pi_{P}} & \mathscr{P}(d)^{\rm{cl}} \inclusion \ar[d]_-{\pi_{P}}
    & \mathscr{P}(d)^{\rm{red}} \inclusion 
    & \mathscr{P}(d) \inclusion & \mathscr{Y}(d) \ar[d]^-{\pi_Y} \\
    P(d) \ar@{=}[r] & P(d)  \ar@<-0.3ex>@{^{(}->}[rrr] & & & Y(d), 
    }
\end{align*}
where the horizontal arrows are closed immersions and $\pi_{P}=\pi_{P,d}$, $\pi_{Y}=\pi_{Y,d}$ are the good 
moduli space morphisms.
% $P(d)$ and $Y(d)$ are given by 
% \begin{align*}P(d):=\mu^{-1}(0)^{\rm{cl}}\ssslash G(d), \ 
% Y(d):=\overline{R}(d)\ssslash G(d),
% \end{align*}
% such that 
%\begin{align*}
%\mathscr{P}(d)^{\rm{red}, \rm{cl}}=
%    \mathscr{P}(d)^{\rm{cl}} \hookrightarrow \mathscr{P}(d)^{\rm{red}} \hookrightarrow 
%    \mathscr{P}(d) \hookrightarrow \mathscr{Y}(d). 
%\end{align*}
%Let $\pi_P$ and $\pi_Y$ be the good moduli space morphisms
%\begin{align*}
%\pi_P \colon 
%        \mathscr{P}(d)^{\rm{cl}} \to P(d) :=
%    \mu^{-1}(0)\ssslash G(d), \ \pi_Y \colon \mathscr{Y}(d) \to Y(d)=\overline{R}(d)\ssslash G(d). 
%\end{align*}
%We have the commutative diagram 
%\begin{align*}
 %   \xymatrix{
  %  \mathscr{P}(d)^{\rm{cl}} \inclusion \ar[d]_-{\pi_P} & \mathscr{Y}(d) \ar[d]^-{\pi_Y} \\
   % P(d) \inclusion & Y(d). 
   % }
%\end{align*}
%Here the horizontal arrows are closed immersions. 

Further, consider the stack  
\begin{align*}
    \X_0(d):=\left(\overline{R}(d) \oplus 
    \mathfrak{g}(d)_0\right)/G(d)
\end{align*}
and the regular function: 
\begin{align*}
    \Tr W_0 =\Tr W|_{\X_0(d)}\colon \X_0(d)
    \to \mathbb{C}. 
\end{align*}   

Let $(Q,W)$ be the tripled quiver with potential associated to $Q^\circ$, see Subsection \ref{subsec22}.
Denote by $(\partial W)$ the two-sided ideal of $\mathbb{C}[Q]$ generated by $\partial W/\partial e$ for $e\in E$.
Consider the Jacobi algebra $J(Q, \partial W):=\mathbb{C}[Q]/(\partial W)$. Let $w_a\in J(Q, \partial W)$ be the image of the element corresponding to the loop $\omega_a$.
The critical locus 
\begin{align*}
\mathrm{Crit}(\Tr W_0) \subset \X_0(d)    
\end{align*}
is the moduli stack of $(Q, W)$-representations
such that the action of $\theta$ has trace zero, where 
$\theta$ is the element
\begin{align}\label{def:theta}
    \theta:=\sum_{a\in I} w_a \in J(Q, \partial W).
\end{align}
The element $\theta$ is a central element in $J(Q, \partial W)$
from the definition of the potential $W$, see (\ref{poten:W}). 
We have the following diagram 
\begin{align}\label{dia:XY}
    \xymatrix{
    \mathfrak{g}(d)_0/G(d) \ar[r]^-{\pi_{\mathfrak{g}}}   & \mathfrak{g}(d)_0\ssslash G(d)  \ar@{=}[r]   
    & 
    \left(\prod_{a \in I} \mathrm{Sym}^{d^a}(\mathbb{C})\right)_0 \\
\X_0(d) \ar@<0.5ex>[d]^-{\eta} \ar[r]^-{\pi_X} \ar[u]^-{p}  & X(d) \ar[r]^-{\Tr W_0} 
\ar@<0.5ex>[d]^-{\overline{\eta}}\ar[u]^-{\overline{p}} & 
\mathbb{C} \\
\mathscr{Y}(d) \ar[r]^-{\pi_Y}\ar@<0.5ex>[u]^-0 & Y(d)\ar@<0.5ex>[u]^-{\overline{0}} . 
    }
\end{align}
Here, recall the good moduli space map 
$\pi_X=\pi_{X,d}$, 
$p$ and $\eta$ are projections, the morphism $0 \colon \mathscr{Y}(d) \to \X_0(d)$
is the zero-section 
of $\eta \colon \X_0(d) \to \mathscr{Y}(d)$,
the middle vertical arrows are the induced maps on good moduli spaces, 
and $\left(\prod_{a\in I}\mathrm{Sym}^{d^a}(\mathbb{C})\right)_0$
is the fiber at $0\in \mathbb{C}$ of the addition map 
\begin{align*}
    \prod_{a\in I} \mathrm{Sym}^{d^a}(\mathbb{C}) \to \mathbb{C}.
\end{align*}
Let $\overline{\mathrm{Crit}}(\Tr W_0)\hookrightarrow X(d)$ be the good moduli space of $\mathrm{Crit}(\Tr W_0)$.
The above diagram restricts to the diagram
\begin{align}\label{dia:crit}
    \xymatrix{
\mathrm{Crit}(\Tr W_0) \ar@<0.5ex>[d]^-{\eta_C}\ar[r]^-{\pi_C} & \overline{\mathrm{Crit}}(\Tr W_0) \ar[r]^-{\overline{p}_C} \ar@<0.5ex>[d]^-{\overline{\eta}_C}&  \left(\prod_{a\in I}\mathrm{Sym}^{d^a}(\mathbb{C})\right)_0  \\
\mathscr{P}(d)^{\rm{cl}} \ar[r]^-{\pi_P}\ar@<0.5ex>[u]^-{0_C} & P(d)\ar@<0.5ex>[u]^-{\overline{0}_C}, 
    }
\end{align}
where the arrows $\pi_C$ and $\pi_P$ 
are good moduli space morphisms. 
We abuse notation and denote by \[0:=(\overbrace{0,\ldots, 0}^{\dd})\in \left(\prod_{a\in I} \mathrm{Sym}^{d^a}(\mathbb{C})\right)_0.\]
Define 
\begin{equation}\label{def:Nnil}
\mathscr{N}_{\rm{nil}}:=\pi_C^{-1}\overline{p}_C^{-1}(0)
        \subset \mathrm{Crit}(\Tr W_0).
        \end{equation}
        %Here we have used the notation in the diagram (\ref{dia:crit}). 
The substack $\mathscr{N}_{\rm{nil}} \subset \mathrm{Crit}(\Tr W_0)$ 
correspond to $(Q, W)$-representations such that the action of $\theta$ is nilpotent. 
Alternatively, it can be described as follows: 
\begin{lemma}\label{lem:nil}
We have $\overline{p}_C^{-1}(0)=\mathrm{Im}(\overline{0}_C)$ in the diagram (\ref{dia:crit}). 
Hence $\mathscr{N}_{\rm{nil}}=\pi_C^{-1}(\mathrm{Im}(\overline{0}_C))$. 
\end{lemma}
    \begin{proof}
    The inclusion $\mathrm{Im}(\overline{0}_C) \subset \overline{p}_C^{-1}(0)$ is obvious. 
    Below we show that $\overline{p}_C^{-1}(0) \subset \mathrm{Im}(\overline{0}_C)$. 
    
    Let $Q^{\circ, d}$ be the doubled quiver of $Q^{\circ}=(I, E^{\circ})$ and let 
    $(\mathscr{I}) \subset \mathbb{C}[Q^{\circ, d}]$
be the two-sided ideal generated by the relation $\mathscr{I}:=\sum_{e\in E^{\circ}}[e, \overline{e}]$. 
Since $\sum_{e\in E^{\circ}}[e, \overline{e}] \in (\partial W)$, 
we have the functor 
\begin{align*}
    \eta_{\ast} \colon J(Q, \partial W)\text{-mod}
    \to \mathbb{C}[Q^{\circ, d}]/(\mathscr{I})\text{-mod}
\end{align*}
which forgets the action of $\theta$, where $\theta \in J(Q, \partial W)$ is defined in (\ref{def:theta}). 

For a simple $J(Q, \partial W)$-module $R$, 
we show that $\eta_{\ast}R$ is a simple module
over $\mathbb{C}[Q^{\circ, d}]/(\mathscr{I})$. We first note that 
the action of 
$\theta$ on $R$ has equal generalized eigenvalues. Indeed, 
$\theta \in J(Q, \partial W)$ is central, so the
generalized eigenspaces for different eigenvalues 
would give a direct sum decomposition of $R$, which 
contradicts that $R$ is simple. Moreover, 
if $R' \subset R$ is an eigenspace for $\theta$, 
then $R'$ is preserved by the $J(Q, \partial W)$-action, so 
$R'=R$ as $R$ is simple. It follows that 
the action of $\theta$ on $R$ is multiplication by $\lambda$ for 
some $\lambda \in \mathbb{C}$. It follows that any 
submodule of $\eta_{\ast}R$ is preserved by the 
action of $\theta$, so $\eta_{\ast}R$ is also simple. 

Let $\bigoplus_{i=1}^k R_i^{\oplus n_i}$ be a semisimple $J(Q, \partial W)$-module. Then the morphism $\overline{\eta}_C$ in 
the diagram (\ref{dia:crit}) sends 
it to 
the semisimple $\mathbb{C}[Q^{\circ, d}]/(\mathscr{I})$-module 
$\bigoplus_{i=1}^k \eta_{\ast}R_i^{\oplus n_i}$, 
as $\eta_{\ast}R_i$ is simple. 
% Thus, given a semisimple $\mathbb{C}[Q^{\circ, d}]/\mathscr{I}$-module 
% $T=\bigoplus_{i=1}^k T_i^{\oplus m_i}$
% corresponding to a point $r\in P(d)$, 
% the set of points 
% of the fiber of $\overline{\eta}_C$ at 
% $r$ consists of choices of 
% $(\lambda_1, \ldots, \lambda_k)$
% for $\lambda_i \in \mathbb{C}$ such that $\theta$ acts on 
% $T_i$ by multiplication by $\lambda_i$ for $1\leq i\leq k$. 
% If it lies in $\overline{p}_C^{-1}(0)$, we must have 
% $\lambda_1=\cdots=\lambda_k=0$. 
% Therefore $\overline{p}_C^{-1}(0) \subset \mathrm{Im}(\overline{0}_C)$
% holds. 
% \textcolor{blue}{(Should it beas follows?: 
Thus, given a semisimple $\mathbb{C}[Q^{\circ, d}]/(\mathscr{I})$-module 
$T=\bigoplus_{i=1}^k T_i^{\oplus n_i}$
corresponding to a point $r\in P(d)$, 
the set of points 
of the fiber of $\overline{\eta}_C$ at 
$r$ consists of choices of 
$\lambda_{ij}\in\mathbb{C}$ for $1\leq i\leq k$, $1\leq j\leq n_i$, such that $\theta$ acts on the $j$th copy of
$T_i$ in $T$ by multiplication by $\lambda_{ij}$. 
If it lies in $\overline{p}_C^{-1}(0)$, we must have 
$\lambda_{ij}=0$ for all $1\leq i\leq k$, $1\leq j\leq n_i$. 
Therefore $\overline{p}_C^{-1}(0) \subset \mathrm{Im}(\overline{0}_C)$
holds. 
\end{proof}
        
\subsection{Quasi-BPS categories for reduced stacks}
We abuse notation and also denote by $j$ the closed immersion 
\begin{align*}j^r \colon \mathscr{P}(d)^{\rm{red}} \hookrightarrow 
\mathscr{Y}(d):=\overline{R}(d)/G(d).
\end{align*}
We define the subcategory 
 \begin{align}\label{def:quasiBPSred}
 \mathbb{T}(d)_v^{\rm{red}} \subset D^b(\mathscr{P}(d)^{\rm{red}})
 \end{align}
   to be consisting of objects $\mathcal{E}$ such that 
   $j^r_{\ast}\mathcal{E}$ is generated by 
   $\mathcal{O}_{\mathscr{Y}(d)} \otimes \Gamma_{G(d)}(\chi)$
   for $\chi$ a dominant weight satisfying (\ref{chiwd}) for $\delta_d=v\tau_d$, i.e. 
   \begin{align}\label{chi:Wd}
       \chi+\rho-v\tau_d \in \mathbf{W}(d),
   \end{align}
   where $\mathbf{W}(d)$ is the polytope defined by (\ref{polyt:W}) 
   for the tripled quiver $Q$ of $Q^{\circ}$. 

  The Koszul equivalence in Theorem~\ref{thm:Koszul} gives an 
  equivalence 
\begin{align}\label{Koszul}
\Theta_0 \colon 
    D^b(\mathscr{P}(d)^{\rm{red}}) \stackrel{\sim}{\to} \mathrm{MF}^{\rm{gr}}(\X_0(d), \Tr W_0). 
\end{align}
Define the reduced quasi-BPS category
\[\mathbb{S}^{\mathrm{gr}}(d)_v^{\mathrm{red}}\subset \mathrm{MF}^{\rm{gr}}(\X_0(d), \Tr W_0)\] as in \eqref{defsdwgr}, that is, 
\begin{equation}\label{def:quasiBPSredgr}
\mathbb{S}^{\mathrm{gr}}(d)_v^{\mathrm{red}}:= \mathrm{MF}^{\rm{gr}}\left(\mathbb{M}(d)_v^{\mathrm{red}}, \Tr W_0\right),
\end{equation}
where $\mathbb{M}(d)_v^{\mathrm{red}}$ is the full subcategory of $D^b(\X_0(d))$ generated by the vector bundles 
$\mathcal{O}_{\X_0(d)}\otimes \Gamma_{G(d)}(\chi)$ for $\chi$ a dominant weight satisfying
(\ref{chi:Wd}). 
The Koszul equivalence \eqref{Koszul}
restricts to the equivalence, see Lemma~\ref{lem:genJ}:
\begin{align*}
\Theta_0 \colon \mathbb{T}(d)_v^{\rm{red}} \stackrel{\sim}{\to} \mathbb{S}^{\rm{gr}}(d)_v^{\rm{red}}.
\end{align*}
%where the right hand side is defined as in (\ref{defsdw}). 
We use Theorem~\ref{lem:support} to study the singular support of sheaves \cite{AG} in reduced quasi-BPS categories: 
\begin{cor}\label{cor:support}
Let $v\in\mathbb{Z}$ such that $\gcd(\dd, v)=1$. Then any object 
    $\mathcal{E} \in \mathbb{T}(d)_v^{\rm{red}}$ has 
    singular support 
    \begin{align}\label{ssupport}
        \mathrm{Supp}^{\rm{sg}}(\mathcal{E}) \subset
        \mathscr{N}_{\rm{nil}}. 
    \end{align}
\end{cor}
\begin{proof}
The singular support of $\mathcal{E}$ equals to the support of 
$\Theta_0(\mathcal{E})$, see~\cite[Proposition~2.3.9]{T}. 
Then the corollary follows from Theorem~\ref{lem:support}
together with 
\begin{align*}
    \Delta \cap \left(\prod_{a\in I} \mathrm{Sym}^{d^a}(\mathbb{C}) \right)_0=\{0\}. 
\end{align*}
\end{proof}
\begin{remark}\label{rmk:ssupport}
The closed substack $\mathrm{Im}(0_C) \subset \mathscr{N}_{\rm{nil}}$
is given by the equation $\theta=0$, and the condition 
$\mathrm{Supp}^{\rm{sg}}(\mathcal{E}) \subset \mathrm{Im}(0_C)$
is equivalent to $\mathcal{E}$ being perfect, see~\cite[Theorem~4.2.6]{AG}. 
The condition (\ref{ssupport}) is weaker than $\mathcal{E}$ being perfect, 
and indeed there may exist objects in $\mathbb{T}(d)_v^{\rm{red}}$
which are not perfect.
\end{remark}

\subsection{Relative properness of quasi-BPS categories}\label{subsec:relproper}
We continue to assume that the quiver $Q^\circ=(I, E^\circ)$ satisfies Assumption \ref{assum1}. 
We will make a further assumption on the quiver $Q^{\circ}$ that guarantees, for example, 
that the reduced stack $\mathscr{P}(d)^{\rm{red}}$ is classical. 
Recall $\alpha_{a, b}$ defined in 
(\ref{alphaab})
and let 
\[\alpha_{Q^{\circ}}:=\mathrm{min}\{\alpha_{a, b} \mid a, b \in I\}.\]
Recall the good moduli spaces $X(d)$, $Y(d)$, $P(d)$ from Subsections \ref{subsub221} and \ref{subsec22}. 

\begin{lemma}\label{lem:dimP}
(i) If $\alpha_{Q^{\circ}} \geq 1$, then $X(d)$ and
$Y(d)$ are Gorenstein with trivial canonical module. 

(ii) If $\alpha_{Q^{\circ}} \geq 2$, then 
$\mathscr{P}(d)^{\rm{red}}$ is a classical stack and 
$P(d)$ is a normal irreducible variety
with 
\begin{align}\label{dim:Pd}
    \dim P(d)=2+\sum_{a, b \in I}d^a d^b \alpha_{a, b}.
\end{align}
If furthermore $P(d)$ is Gorenstein, then its canonical module is trivial. 

(iii) If $\alpha_{Q^{\circ}} \geq 3$, or if $\alpha_{Q^{\circ}}=2$ and $\underline{d} \geq 3$, 
then $P(d)$ is Gorenstein. 
\end{lemma}
\begin{proof}
(i) We only prove the case of $Y(d)$, as the case of $X(d)$ is similar. 
Let $\mathscr{Y}(d)^{\rm{s}} \subset \mathscr{Y}(d)$ be the open substack 
corresponding to simple representations. By~\cite[Korollary~2]{Knop2}, 
it is enough to show that the codimension of $\mathscr{Y}(d) \setminus \mathscr{Y}(d)^{\rm{s}}$
is at least two. 
Let $\lambda$ be a cocharacter corresponding to $d=d_1+d_2$ such that $d_1$ and $d_2$ are non-zero.
A simple calculation shows that 
\begin{align*}
 \dim \mathscr{Y}(d) -\dim \mathscr{Y}(d)^{\lambda \geq 0}   
 =\sum_{a, b \in I}d_1^a d_2^b \alpha_{a, b} +\sum_{a \in I} d_1^a d_2^a \geq 2 \dd_1 \dd_2 \alpha_{Q^{\circ}}\geq 2. 
\end{align*}
Therefore $\mathscr{Y}(d) \setminus \mathscr{Y}(d)^{\rm{s}}$ is at least of codimension two in $\mathscr{Y}(d)$. 

(ii) If $\alpha_{Q^{\circ}} \geq 2$, then 
 $\mu_0^{-1}(0) \subset \overline{R}(d)$ is 
an irreducible variety of dimension $\dim \overline{R}(d)-\dim \mathfrak{g}(d)_0$
by ~\cite[Proposition~3.6]{KaLeSo}. 
In particular $\mathscr{P}(d)^{\rm{red}}$ is classical. 
Moreover, in the proof of~\cite[Proposition~3.6]{KaLeSo}, it is also proved 
that $\mu_0^{-1}(0)$ contains a dense open subset of points with trivial stabilizer 
groups in $G(d)/\mathbb{C}^{\ast}$. 
Therefore the dimension (\ref{dim:Pd})
is easily calculated from 
\begin{align*}
    \dim P(d)=1+\dim \mathscr{P}(d)^{\rm{red}}=2+\dim \overline{R}(d)-2\dim \mathfrak{g}(d). 
\end{align*}
The last statement holds since the canonical module of $\mu_0^{-1}(0)$ is a $G(d)$-equivariantly 
trivial line bundle. 

(iii) By~\cite{Flenner} and~\cite{Nam}, it is enough to show that the 
codimension of the singular locus of $P(d)$ is at least $4$, 
see~\cite[Proposition~1.1, Theorem~1.2]{Baoh}. 
The singular locus of $P(d)$ is contained in the union of images of 
\begin{align*}
\oplus \colon P(d_1) \times P(d_2) \to P(d)
\end{align*}
for $d=d_1+d_2$ with $d_i \neq 0$. The codimension of 
the image of the above map is at least 
\begin{align*}
\dim P(d)-\dim P(d_1)-\dim P(d_2) &=-2+\sum_{a, b \in I}(d_1^a d_2^b +d_1^b d_2^a)\alpha_{a, b} \\
&\geq 2 \underline{d}_1 \underline{d}_2 \alpha_{Q^{\circ}}-2 \geq 4.
\end{align*}
Here the identity follows from (\ref{dim:Pd}), the first inequality follows from the definition of $\alpha_{Q^{\circ}}$, and the last 
inequality follows from the assumption on $\alpha_{Q^{\circ}}$ and $d$. 
Therefore (iii) holds. 
\end{proof}
\begin{example}\label{exam:gloop2}
As in Example~\ref{exam:gloop}, let $Q^{\circ}$ be a quiver with 
one vertex and $g$-loops. 
Then $\alpha_{Q^{\circ}}=2g-2$. By Lemma~\ref{lem:dimP} the stack $\mathscr{P}(d)^{\rm{red}}$ is 
classical if $g\geq 2$, and $P(d)$ is Gorenstein 
if $g\geq 3$ or $g=2$, $d\geq 3$. 
When $g=d=2$, the singular locus of $P(d)$ is of codimension two, 
but nevertheless it is also Gorenstein because
it admits a symplectic resolution of
singularities~\cite{Ogra1}.     
\end{example}

%\begin{assum}\label{assum2}
%For the 
%quiver $Q^{\circ}=(I, E^{\circ})$, assume:
%\begin{enumerate}
%    \item There exist at least two loops at any vertex $a \in I$. 
%    \item For any $a, b \in I$ with $a \neq b$, there exists 
%    an edge from $a$ to $b$ or from $b$ to $a$. 
%\end{enumerate}
%\end{assum}
%In this case, $\mu_0^{-1}(0) \subset \overline{R}(d)$ is 
%an irreducible variety of dimension $\dim \overline{R}(d)-\dim \mathfrak{g}(d)_0$
%by ~\cite[Proposition~3.6]{KaLeSo}. 
%In particular $\mathscr{P}(d)^{\rm{red}}$ is classical, 
%so we have the good 
%moduli space morphism 
%\begin{align*}
%    \pi_P \colon \mathscr{P}(d)^{\rm{red}} \to P(d),
%\end{align*}
%where $P(d)$ is a normal irreducible affine variety. 
%Moreover in many cases, $P(d)$ has symplectic singularity
%hence it is rational Gorenstein, see~\cite{MR4058942}. 
Below we assume that $\mathscr{P}(d)^{\rm{red}}$ is classical, 
e.g. $\alpha_{Q^{\circ}} \geq 2$. 
Then we have the good 
moduli space morphism 
\begin{align*}
    \pi_P=\pi_{P,d} \colon \mathscr{P}(d)^{\rm{red}} \to P(d).
\end{align*}
In particular for $\mathcal{E}_1, \mathcal{E}_2 \in D^b(\mathscr{P}(d)^{\rm{red}})$, 
the Hom space $\Hom(\mathcal{E}_1, \mathcal{E}_2)$ is a module over $\mathcal{O}_{P(d)}$. 

In general for a dg-category $\mathcal{C}$ linear over a commutative algebra $R$, 
we say it is proper over $\Spec R$ if $\Hom^{\ast}(E, F)$ is a finitely generated $R$-module. 
The categorical support condition in Corollary~\ref{cor:support} implies the 
relative properness of quasi-BPS categories, which is a non-trivial 
statement as objects in $\mathbb{T}(d)_v^{\rm{red}}$ may not be 
perfect, see Remark~\ref{rmk:ssupport}: 
\begin{prop}\label{lem:bound}
Suppose that the stack $\mathscr{P}(d)^{\rm{red}}$ is classical, e.g. $\alpha_{Q^{\circ}} \geq 2$.  
For $v\in\mathbb{Z}$ such that $\gcd(\dd, v)=1$ and objects $\mathcal{E}_i \in \mathbb{T}(d)_v^{\rm{red}}$
for $i=1, 2$, the $\mathcal{O}_{P(d)}$-module 
\begin{align*}
    \bigoplus_{i\in \mathbb{Z}} \Hom^i(\mathcal{E}_1, \mathcal{E}_2)
\end{align*}
is finitely generated, i.e. the category $\mathbb{T}(d)_v^{\rm{red}}$ is proper over 
$P(d)$. In particular we have $\Hom^i(\mathcal{E}_1, \mathcal{E}_2)=0$ for 
$\lvert i \rvert \gg 0$. 
\end{prop}

\begin{proof}
Let 
$F_i$ be defined by 
\begin{align*}
F_i=\mathrm{forg}\circ \Theta_0(\mathcal{E}_i) \in 
\mathrm{MF}(\X_0(d), \Tr W_0)
\end{align*}
where $\Theta_0$ is the equivalence (\ref{Koszul}) and $\mathrm{forg}$
is the forget-the-grading functor. 
Consider its internal Hom:
\begin{align*}
    \mathcal{H}om(F_1, F_2) \in \mathrm{MF}(\X_0(d), 0)=D^{\mathbb{Z}/2}(\X_0(d)). 
\end{align*}
Here $D^{\mathbb{Z}/2}(-)$ is the $\mathbb{Z}/2$-graded derived category 
of coherent sheaves. 
The above object is supported over 
$\pi_C^{-1}(\mathrm{Im}(\overline{0}_C))$ 
by 
Corollary~\ref{cor:support} and Lemma~\ref{lem:nil}. 
As $\pi_X$ is a good moduli space morphism, 
$\pi_{X\ast}$ sends a coherent sheaf to a bounded complex of 
coherent sheaves. Therefore we have 
\begin{align}\label{gast}
    \pi_{X \ast}\mathcal{H}om(F_1, F_2)
    \in D^{\mathbb{Z}/2}(X(d))
\end{align}
and, further, that $\pi_{X\ast}\mathcal{H}om(F_1, F_2)$ is
supported over $\mathrm{Im}(\overline{0}_C) \subset \mathrm{Im}(\overline{0})$, 
see the diagrams (\ref{dia:XY}), (\ref{dia:crit}). 
Since $\overline{0}$ is a section of $\overline{p}$, 
the restriction of $\overline{\eta}$ to $\mathrm{Im}(\overline{0})$ 
is an isomorphism, 
$\overline{\eta} \colon \mathrm{Im}(\overline{0}) 
\stackrel{\cong}{\to} Y(d)$. 
In particular, (\ref{gast}) has a proper support 
over $Y(d)$. 
Therefore, 
we have that
\begin{align*}
    \overline{\eta}_{\ast}\pi_{X\ast}\mathcal{H}om(F_1, F_2)
    \in D^{\mathbb{Z}/2}(Y(d)), 
\end{align*}
so $\Hom^{\ast}(F_1, F_2)$ is finitely generated over $Y(d)$, hence 
over $P(d)$
as $P(d) \hookrightarrow Y(d)$ is a closed subscheme. 
Now, for $i \in \mathbb{Z}/2$, we have 
\begin{align*}
    \Hom^{i}(F_1, F_2)=\bigoplus_{k\in \mathbb{Z}}\Hom^{2k+i}(\mathcal{E}_1, \mathcal{E}_2),
\end{align*}
hence the finite generation over $P(d)$ of the left hand side implies 
the proposition. 
\end{proof}

\subsection{Relative Serre functor on quasi-BPS categories}
We keep the situation of Proposition~\ref{lem:bound}. 
The category $\mathbb{T}=\mathbb{T}(d)_v$
is a 
module over $\mathrm{Perf}(P(d))$.
We have the associated internal homomorphism,  see Subsection~\ref{subsec:Koszul}:
\begin{align*}
 \mathcal{H}om_{\mathbb{T}}(\mathcal{E}_1, \mathcal{E}_2)
\in D_{\rm{qc}}(\mathscr{P}(d)^{\rm{red}}), \ 
\mathcal{E}_i \in \mathbb{T}(d)_v^{\rm{red}}.
\end{align*}
Then Proposition~\ref{lem:bound} implies that 
\begin{align*}
    \pi_{\ast} \mathcal{H}om_{\mathbb{T}}(\mathcal{E}_1, \mathcal{E}_2) \in D^b(P(d)).
    \end{align*}
A functor 
$S_{\mathbb{T}/P} \colon \mathbb{T} \to \mathbb{T}$ is
called \textit{a relative Serre functor}
if it satisfies
the isomorphism 
\begin{align*}
    \Hom_{P(d)}(\pi_{\ast}\mathcal{H}om_{\mathbb{T}}(\mathcal{E}_1, \mathcal{E}_2), \mathcal{O}_{P(d)}) \cong 
    \Hom_{\mathbb{T}}(\mathcal{E}_2, S_{\mathbb{T}/P}(\mathcal{E}_1)). 
\end{align*} 
The following result shows that 
$\mathbb{T}$ is strongly crepant in the sense of~\cite[Section~2.2]{VdB22}. 
\begin{thm}\label{thm:Serretriv}
Suppose that $\alpha_{Q^{\circ}} \geq 2$ and 
$P(d)$ is Gorenstein, e.g. $\alpha_{Q^{\circ}} \geq 3$. Let $v\in\mathbb{Z}$ such that $\gcd(\dd, v)=1$. Then 
the relative Serre functor $S_{\mathbb{T}/P}$
exists and satisfies 
$S_{\mathbb{T}/P} \cong \id_{\mathbb{T}}$. 
\end{thm}
\begin{proof}
    We have the following commutative diagram 
    \begin{align}\label{dia:comM}
    \xymatrix{
        \mathscr{P}(d)^{\rm{red}} \ar[d]_-{\pi_P} \inclusion^{j^r} & \mathscr{Y}(d) \ar[d]_-{\pi_Y}  &
        \X_0(d) \ar[l]_{\eta} \ar[d]_-{\pi_X} \\
        P(d) \inclusion^-{\overline{j}} & Y(d) & X(d) \ar[l]_-{\overline{\eta}}. 
        }
    \end{align}
    Let $\Theta_0$
    be the Koszul duality equivalence in (\ref{Koszul}). 
    Then by Lemma~\ref{lem:internal}, we have 
 \begin{align*}
  j^r_{\ast}\mathcal{H}om_{\mathbb{T}}(\mathcal{E}_1, \mathcal{E}_2)=\eta_{\ast}\mathcal{Q},
  \end{align*}
    where $\mathcal{Q}$ is the internal homomorphism of matrix 
    factorizations 
    \begin{align*}
        \mathcal{Q}=\mathcal{H}om_{\mathrm{MF}}(\Theta_0(\mathcal{E}_1), \Theta_0(\mathcal{E}_2))
        \in \mathrm{MF}^{\rm{gr}}(\X_0(d), 0).
    \end{align*}
    %We note that the left square in (\ref{dia:comM}) is 
    %Cartesian since both of the vertical arrows have 
    %relative dimension $-1$ 
    %\textcolor{blue}{(Should it be $-\dim \mathfrak{g}(d)_0$?)}
    %\textcolor{red}{(The vertical arrows $\pi$, $\pi_Y$ have dimension $-1$. However I'm now not sure whether the left square is Cartesian, since $\mathscr{Y}(d) \times_{Y(d)}P(d)$ may have extra irreducible components. But I think we don't need it if we assume that $P(d)$ is Gorenstein.)}
Since the left vertical arrows $\pi_P$, $\pi_Y$ have relative dimension $-1$ and the 
codimension of the closed substack 
$\mathscr{P}(d)^{\rm{red}}$ in $\mathscr{Y}(d)$ is $\dim \mathfrak{g}(d)_0$, 
the codimension of $P(d)$ in $Y(d)$ is 
$\dim \mathfrak{g}(d)_0$. 
By Lemma~\ref{lem:dimP}, we have the following descriptions of 
dualizing complexes
\begin{align}\label{dualizingcomplexes}
    \omega_{Y(d)}=\mathcal{O}_{Y(d)}[\dim Y(d)], \ 
    \omega_{P(d)}=\mathcal{O}_{P(d)}[\dim Y(d)-\dim \mathfrak{g}(d)_0]. 
\end{align}
As $\overline{j}^{!}\omega_{Y(d)}=\omega_{P(d)}$, we have 
\begin{align*}
\overline{j}^{!}\mathcal{O}_{Y(d)}=\mathcal{O}_{P(d)}[-\dim \mathfrak{g}(d)_0].
\end{align*}
   Then we have the isomorphisms 
   \begin{align*}
    &\Hom_{P(d)}(\pi_{P\ast}\mathcal{H}om_{\mathbb{T}}(\mathcal{E}_1, \mathcal{E}_2), \mathcal{O}_{P(d)}) \\
    &\cong 
    \Hom_{Y(d)}(\overline{j}_{\ast}\pi_{P\ast}\mathcal{H}om_{\mathbb{T}}(\mathcal{E}_1, \mathcal{E}_2), \mathcal{O}_{Y(d)}[\dim \mathfrak{g}(d)_0]) \\
    &\cong \Hom_{Y(d)}(\pi_{Y\ast}j^r_{\ast}\mathcal{H}om_{\mathbb{T}}(\mathcal{E}_1, \mathcal{E}_2), \mathcal{O}_{Y(d)}[\dim \mathfrak{g}(d)_0]) \\
    &\cong \Hom_{Y(d)}(\pi_{Y\ast}\eta_{\ast}\mathcal{Q}, 
    \mathcal{O}_{Y(d)}[\dim \mathfrak{g}(d)_0]) \\
    &\cong \Hom_{Y(d)}(\overline{\eta}_{\ast}\pi_{X\ast}\mathcal{Q}, \mathcal{O}_{Y(d)}[\dim \mathfrak{g}(d)_0]).
   \end{align*}
   By Lemma~\ref{lem:dimP} (i), we have 
   \begin{align*}
   \omega_{X(d)}=\mathcal{O}_{X(d)}[\dim X(d)](-2\dim \mathfrak{g}(d)_0)=
   \mathcal{O}_{X(d)}[\dim Y(d)-\dim \mathfrak{g}(d)_0].
   \end{align*}
     Here, we denote by $(1)$ the grade shift with respect to the 
$\mathbb{C}^{\ast}$-action on $X(d)$, induced by the  
   fiberwise weight two $\mathbb{C}^{\ast}$-action on 
   $\X_0(d) \to \mathscr{Y}(d)$, which is isomorphic to $[1]$ in 
   $\mathrm{MF}^{\rm{gr}}(X(d), 0)$. 
   As in the proof of Lemma~\ref{lem:bound}, the complex
   $\pi_{X\ast}\mathcal{Q}$ has proper support over $Y(d)$.
   Therefore, from  
   $\eta^{!}\omega_{Y(d)}=\omega_{X(d)}$ and \eqref{dualizingcomplexes}, 
   we have 
\begin{align*}
&\Hom_{Y(d)}(\overline{\eta}_{\ast}\pi_{X\ast}\mathcal{Q}, \mathcal{O}_{Y(d)}[\dim \mathfrak{g}(d)_0]) \\
&\cong 
    \Hom_{X(d)}(\pi_{X\ast}\mathcal{Q}, \overline{\eta}^{!}\mathcal{O}_{Y(d)}[\dim \mathfrak{g}(d)_0]) \\
    &\cong \Hom_{X(d)}(\pi_{X\ast}\mathcal{Q}, \mathcal{O}_{X(d)}). 
\end{align*}
Note that $\overline{\eta}$ is not proper, but the first isomorphism above holds  
because $\pi_{X\ast}\mathcal{Q}$ has proper support over $Y(d)$. 
   In the above, $\Hom_{X(d)}(-, -)$ denotes the space of homomorphisms 
   in the category $\mathrm{MF}^{\rm{gr}}(X(d), 0)$; note that $X(d)$ may not be smooth, but its definition and the definition of the functor $\overline{\eta}_*$ on the subcategory of matrix factorizations with proper support over $Y(d)$ are as in the smooth case \cite{MR3112502, EfPo}. 
   
   By the definition of $\mathbb{S}^{\mathrm{gr}}(d)_v$,
   %and the fact that $\mathrm{Ext}^i(\mathcal{V}, \mathcal{V}')=0$ for $\mathcal{V}, \mathcal{V}'$ two vector bundles on $\X(d)$ and for $i\neq 0$
   %\textcolor{blue}{(I think we should add some explanation why we take direct sums in what follows, and not say extensions as this is how $\mathbb{S}(d)$ was defined)}, 
   the object $\mathcal{Q}$ is represented by 
   \[(\mathcal{Q}^0 \leftrightarrows \mathcal{Q}^1),\]
   where $\mathcal{Q}^0$ and $\mathcal{Q}^1$ are direct sums 
   of vector bundles of the form 
   $\Gamma_{G(d)}(\chi_1)^{\vee} \otimes \Gamma_{G(d)}(\chi_2) \otimes \mathcal{O}_{\X_0(d)}$
   such that the weights
   $\chi_i$ satisfy, for $i\in\{1, 2\}$: 
   \begin{align}\label{defW511}
       \chi_i +\rho-v\tau_d \in
       \mathbf{W}:=\frac{1}{2}\mathrm{sum}[0, \beta],
   \end{align}
   where the above Minkowski sum is over all weights $\beta$ in $\overline{R}(d) \oplus \mathfrak{g}(d)_0$. 
      By Lemma~\ref{lem:boundary} below, 
   the weight $\chi_2-\chi_1$ lies in the 
   interior of the polytope $-2\rho+2\mathbf{W}$.
   Therefore applying~\cite[Proposition~4.4.4]{SVdB}
   for the $PG(d):=G(d)/\mathbb{C}^{\ast}$-action on $\overline{R}(d)\oplus 
   \mathfrak{g}(d)_0$, the sheaf 
   $\pi_{X\ast}\mathcal{Q}^i$ is a maximal 
   Cohen-Macaulay sheaf on $X(d)$. 
 It follows that 
   \begin{align*}
  \mathcal{H}om_{X(d)}(\pi_{X\ast}\mathcal{Q}^i, \mathcal{O}_{X(d)})=\pi_{X\ast}(\mathcal{Q}^i)^{\vee}.
  \end{align*}
   The morphism $\pi_{X\ast}\mathcal{Q}^i \to 
   \pi_{X\ast}\mathcal{Q}^{i+1}$ is uniquely
   determined by its restriction to the smooth locus 
   of $X(d)$, and so
   \begin{align*}
    \mathcal{H}om_{X(d)}(\pi_{X\ast}\mathcal{Q}, \mathcal{O}_{X(d)}) \cong \pi_{X\ast}(\mathcal{Q}^{\vee}).
    \end{align*}
   Therefore we have the isomorphisms
   \begin{align*}
       \Hom_{X(d)}(\pi_{X\ast}\mathcal{Q}, \mathcal{O}_{X(d)})
       &\cong \Hom_{X(d)}(\mathcal{O}_{X(d)}, 
       \pi_{X\ast}(\mathcal{Q}^{\vee})) \\
       &\cong 
       \Hom_{\mathrm{MF}(\X(d), \mathrm{Tr}\,W)}(\Theta_0(\mathcal{E}_2), \Theta_0(\mathcal{E}_1)) \\
       &\cong 
       \Hom_{\mathscr{P}(d)}(\mathcal{E}_2, \mathcal{E}_1)
       \\
       &\cong 
       \Hom_{\mathbb{T}}(\mathcal{E}_2, \mathcal{E}_1), 
   \end{align*}
   and the conclusion follows.
\end{proof}

We are left with proving the following:

\begin{lemma}\label{lem:boundary}
In the setting of \eqref{defW511}, 
    let $\chi\in M(d)$ be a dominant weight such that 
    $\chi+\rho-v\tau_d \in \mathbf{W}$. 
    If $\gcd(\dd, v)=1$, then $\chi+\rho-v\tau_d$ is not 
    contained in the boundary of $\mathbf{W}$. 
\end{lemma}
\begin{proof}
 Suppose that $\chi+\rho-v\tau_d$ is on the 
 boundary of $\mathbf{W}$. 
 Use the decomposition (\ref{decompchi})
for $\delta_d=0$ to obtain the decomposition
\begin{align}\label{id:chirho}
    \chi+\rho=-\frac{1}{2}R(d)^{\lambda>0}+\sum_{i=1}^k \psi_i +v\tau_d
\end{align}
for some antidominant cocharacter $\lambda$ such that,
if $d=d_1+\cdots+d_k$ for $k\geq 2$ is the decomposition corresponding to $\lambda$,
then $\langle 1_{d_i}, \psi_i\rangle=0$. 
We write 
\begin{align}\label{id:vi}
v\tau_d=\sum_{i=1}^k v_i \tau_{d_i}, \ 
    \frac{v}{\dd}=\frac{v_i}{\dd_i}
\end{align}
for $v_i \in \mathbb{Q}$. 
Then the identity (\ref{id:chirho}) is written as 
\begin{align}\label{id:chifrac}
    \chi=-\frac{1}{2}\overline{R}(d)^{\lambda>0}+\sum_{i=1}^k v_i\tau_{d_i}+\sum_{i=1}^k (\psi_i-\rho_i).
\end{align}
By Assumption~\ref{assum1}, we have 
$\frac{1}{2}\overline{R}(d)^{\lambda>0} \in M(d)^{W_d}$. 
Therefore the identity (\ref{id:chifrac}) implies that 
\begin{align*}
    v_i=\langle 1_{d_i}, \chi\rangle +\left\langle 1_{d_i}, \frac{1}{2}\overline{R}(d)^{\lambda>0} \right\rangle
    \in \mathbb{Z} 
\end{align*} for all $1\leq i\leq k$.
However, as $v/\dd=v_i/\dd_i$, we obtain a contradiction with the assumption that 
$\gcd(\dd, v)=1$. 
\end{proof}

\subsection{Indecomposability of quasi-BPS categories}

Let $Q^\circ=(I, E^\circ)$ be a quiver and let $d\in\mathbb{N}^I$. Recall the good moduli space map \[\pi_P\colon \mathscr{P}(d)^{\mathrm{red}}=\mu_0^{-1}(0)/G(d)\to P(d)\] from the reduced stack of dimension $d$ representations of the
preprojective algebra of $Q^\circ$.
%, let $\overline{Q}$ be the double quiver of $Q^\circ$. For $d\in\mathbb{N}^I$, denote the stack of representations of $\overline{Q}$ of dimension $d$ by \[\mathscr{Y}(d):=\overline{R}(d)/G(d).\]Let \[\mathscr{P}(d):=\mu^{-1}(0)/G(d)\] be the stack of dimension $d$ representations of the preprojective algebra of $Q^\circ$, where $\mu$ is the moment map \[\mu\colon \overline{R}(d)\to \mathfrak{g}(d)^\vee\cong\mathfrak{g}(d).\]The action of $G(d)$ on the representation spaces of dimension $d$ of $Q^\circ$ and $\overline{Q}$ factor through $G_0:=G/\text{diag}\left(\mathbb{C}^*\right)$.There is thus a reduced moment map:\[\mu_0\colon \overline{R}(d)\to \mathfrak{g}_0,\]where $\mathfrak{g}_0$ is the Lie algebra of $G_0$. Consider the stack\[\mathscr{P}(d)^{\mathrm{red}}:=\mu_0^{-1}(0)/G(d).\]Let $\pi\colon \mathscr{P}(d)^{\mathrm{red}}\to P(d)$ be the good moduli space. %For a partition $A=(d_i)_{i=1}^k$ with $d_i\geq 1$, consider the map \[\oplus_A\colon \otimes_{i=1}^k P(d_i)\to P(d).\] Let \[P(d)':=P(d)\setminus \bigcup_{A\text{ s.t. }k\geq 2}\mathrm{image}(\oplus_A).\] Then $P(d)'$ is an open subset of $P(d)$.

\begin{prop}\label{mainprop:dec}
    Let $Q^\circ$ be a quiver and let $d\in\mathbb{N}^I$ such that $\mathscr{P}(d)^{\mathrm{red}}$ is a classical stack, e.g.
    $\alpha_{Q^{\circ}} \geq 2$. % and $P(d)'$ is not empty. 
     Let $v\in\mathbb{Z}$. Then $\mathbb{T}(d)_v^{\mathrm{red}}$ does not have a non-trivial orthogonal decompositions. 
\end{prop}

Recall the following standard lemma:

\begin{lemma}\label{lemmaortho}
    Let $\mathcal{D}$ be a triangulated category with Serre functor equal to a shift. Then a semi-orthogonal decomposition $\mathcal{D}=\langle \mathbb{A}, \mathbb{B}\rangle$ is an orthogonal decomposition $\mathcal{D}=\mathbb{A}\oplus \mathbb{B}$.
\end{lemma}

We note the following corollary of Proposition \ref{mainprop:dec} and Lemma \ref{lemmaortho}:

\begin{cor}\label{mainprop:dec2} 
Let $Q^\circ$ be a quiver satisfying Assumption \ref{assum1} and let $d\in\mathbb{N}^I$ such that $\mathscr{P}(d)^{\mathrm{red}}$ is a classical stack.
Let $v\in\mathbb{Z}$ such that $\gcd(v, \dd)=1$.
Then the category $\mathbb{T}(d)_v^{\mathrm{red}}$ does not have any non-trivial semiorthogonal decompositions. 
\end{cor}

\begin{proof}
    Assume there is a semiorthogonal decomposition $\mathbb{T}(d)_v^{\mathrm{red}}=\langle \mathbb{A}, \mathbb{B}\rangle$. By Theorem \ref{thm:Serretriv} and Lemma \ref{lemmaortho}, there is an orthogonal decomposition $\mathbb{T}(d)_v^{\mathrm{red}}=\mathbb{A}\oplus \mathbb{B}$. By Proposition \ref{mainprop:dec}, one of the categories $\mathbb{A}$ and $\mathbb{B}$ is zero. 
\end{proof}

Before we begin the proof of Proposition \ref{mainprop:dec}, we note a few preliminary results. We assume in the remaining of the subsection that $\mathscr{P}(d)^{\mathrm{red}}$ is a classical stack.
We say a representation $V$ of $G(d)$ has weight $v\in\mathbb{Z}$ if $1_d$ acts on $V$ with weight $v$.

\begin{prop}\label{prop416}
    Let $V$ be a representation of $G(d)$ of weight zero. Then the sheaf $\pi_*\left(\mathcal{O}_{\mathscr{P}(d)^{\mathrm{red}}}\otimes V\right)$ is non-zero and torsion free.
\end{prop}

\begin{proof}
    The module $\mathcal{O}_{\mathscr{P}(d)^{\mathrm{red}}}\otimes V$ is a non-zero torsion free $\mathcal{O}_{\mu_0^{-1}(0)}$-module of weight zero, thus it is also a torsion free $\mathcal{O}_{P(d)}=\mathcal{O}^{G(d)}_{\mu_0^{-1}(0)}\subset \mathcal{O}_{\mu_0^{-1}(0)}$-module.
\end{proof}

The category $\mathbb{T}(d)_v^{\mathrm{red}}$ is admissible in $D^b\left(\mathscr{P}(d)^{\mathrm{red}}\right)_v$ by an immediate modification of the argument for the admissibility of $\mathbb{T}(d)_v$ in $D^b\left(\mathscr{P}(d)\right)_v$, which follows from the Koszul equivalence \eqref{Kosz} and \cite[Theorem 1.1]{P}. Then there exists a left adjoint of the inclusion $\mathbb{T}(d)_v^{\mathrm{red}}
\hookrightarrow D^b\left(\mathscr{P}(d)^{\mathrm{red}}\right)_v$, which we denote by 
\[\Phi\colon D^b\left(\mathscr{P}(d)^{\mathrm{red}}\right)_v\to \mathbb{T}(d)_v^{\mathrm{red}}.\]

\begin{prop}\label{prop417}
    Let $V$ be a representation of $G(d)$ of weight $v$ and let $\mathscr{A}$ be a direct summand of $\Phi(\mathcal{O}_{\mathscr{P}(d)^{\mathrm{red}}}\otimes V)$. Then $\pi_*(\mathscr{A})$ has support $P(d)$.
\end{prop}

\begin{proof}
    Write $\Phi(\mathcal{O}_{\mathscr{P}(d)^{\mathrm{red}}}\otimes V)=\mathscr{A}\oplus \mathscr{B}$ for $\mathscr{A}, \mathscr{B}\in D^b\left(\mathscr{P}(d)^{\mathrm{red}}\right)_v$. There exists a representation $V'$ of $G(d)$ of weight $v$ such that $\mathrm{Hom}\left(\mathscr{A}, \mathcal{O}_{\mathscr{P}(d)^{\mathrm{red}}}\otimes V'\right)\neq 0$. Then $\mathrm{Hom}\left(\mathscr{A}, \mathcal{O}_{\mathscr{P}(d)^{\mathrm{red}}}\otimes V'\right)$ is a direct summand of 
    \begin{align*}
\mathrm{Hom}\left(\Phi(\mathcal{O}_{\mathscr{P}(d)^{\mathrm{red}}}\otimes V), \mathcal{O}_{\mathscr{P}(d)^{\mathrm{red}}}\otimes V'\right)&=\mathrm{Hom}\left(\mathcal{O}_{\mathscr{P}(d)^{\mathrm{red}}}\otimes V, \mathcal{O}_{\mathscr{P}(d)^{\mathrm{red}}}\otimes V'\right)\\&=\pi_*\left(\mathcal{O}_{\mathscr{P}(d)^{\mathrm{red}}}\otimes V'\otimes V^\vee\right),\end{align*}
which is non-zero and torsion free over $P(d)$ by Proposition \ref{prop416}, and thus the $P(d)$-sheaf $\mathrm{Hom}\left(\mathscr{A}, \mathcal{O}_{\mathscr{P}(d)^{\mathrm{red}}}\otimes V'\right)$ has support $P(d)$. Then also $\pi_*(\mathscr{A})$ has support $P(d)$.\end{proof}

\begin{prop}\label{prop418}
For $\mathscr{A}, \mathscr{A}' \in D^b\left(\mathscr{P}(d)^{\mathrm{red}}\right)_v$, suppose that 
$\pi_{\ast}\mathscr{A}$ and $\pi_{\ast}\mathscr{A}'$
have support $P(d)$. Then $\mathrm{Hom}^{i}(\mathscr{A}, \mathscr{A}')\neq 0$ for some $i\in \mathbb{Z}$. 
\end{prop}

\begin{proof}
    The object $\pi_{P\ast}\mathcal{H}om(\mathscr{A}, \mathscr{A}')\in D_{\rm{qc}}(P(d))$ is non-zero because 
   it is non-zero over a generic point. Then $R\Hom(\mathscr{A}, \mathscr{A}')=R\Gamma(\pi_{P\ast}\mathcal{H}om(\mathscr{A}, \mathscr{A}'))$
   is also non-zero as $P(d)$ is affine, and the conclusion follows.
\end{proof}

\begin{proof}[Proof of Proposition \ref{mainprop:dec}]
    Assume $\mathbb{T}(d)_v^{\mathrm{red}}$ has an orthogonal decomposition in categories $\mathbb{A}$ and $\mathbb{B}$.
    By Propositions \ref{prop417} and \ref{prop418},
    all summands of $\Phi(\mathcal{O}_{\mathscr{P}(d)^{\mathrm{red}}}\otimes V)$ are in the same category, say $\mathbb{A}$, for all representations $V$ of $G(d)$ of weight $v$. Then the complexes $\Phi(\mathcal{O}_{\mathscr{P}(d)^{\mathrm{red}}}\otimes V)$ are in $\mathbb{A}$ for all representations $V$ of $G(d)$ of weight $v$.

    Let $\mathscr{A}\in \mathbb{T}(d)_v^{\mathrm{red}}$ be non-zero and indecomposable. Then there exists $V$ such that $\mathrm{Hom}(\mathcal{O}_{\mathscr{P}(d)^{\mathrm{red}}}\otimes V, \mathscr{A})\neq 0$, and so $\mathrm{Hom}\left(\Phi(\mathcal{O}_{\mathscr{P}(d)^{\mathrm{red}}}\otimes V), \mathscr{A}\right)\neq 0$. The complex $\mathscr{A}$ is indecomposable, so $\mathscr{A}\in \mathbb{A}$, and thus $\mathbb{B}=0$.
\end{proof}

 \bibliographystyle{amsalpha}
\bibliography{math}
\medskip

\textsc{\small Tudor P\u adurariu: Max Planck Institute for Mathematics,
Vivatsgasse 7 
Bonn 53111, Germany.}\\
\textit{\small E-mail address:} \texttt{\small tpadurariu@mpim-bonn.mpg.de}\\

\textsc{\small Yukinobu Toda: Kavli Institute for the Physics and Mathematics of the Universe (WPI), University of Tokyo, 5-1-5 Kashiwanoha, Kashiwa, 277-8583, Japan.}\\
\textit{\small E-mail address:} \texttt{\small yukinobu.toda@ipmu.jp}\\

 \end{document}